\documentclass[12pt,reqno]{amsart}

\usepackage[left=1in, right=1in, top=1.1in,bottom=1.1in]{geometry}
\usepackage[dvipsnames]{xcolor}
\usepackage[mathscr]{eucal}
\usepackage{latexsym,bm,amsfonts,amssymb,pifont,  bm,  stmaryrd}
\usepackage{graphicx}
\usepackage{color}
\usepackage{hyperref}
\hypersetup{
     colorlinks   = true,
     citecolor    = blue,
     linkcolor    = blue
}

\usepackage{pdfsync}
\usepackage[font={scriptsize}]{caption}

\usepackage{enumerate}
\usepackage{pdfsync}
\newcommand{\ep}{\varepsilon}

\newtheorem{theorem}{Theorem}[section]
\newtheorem{Def}[theorem]{Definition}
\newtheorem{notation}[theorem]{Notation}
\newtheorem{thm}[theorem]{Theorem}
\newtheorem{prop}[theorem]{Proposition}
\newtheorem{cor}[theorem]{Corollary}
\newtheorem{lemma}[theorem]{Lemma}

\theoremstyle{remark}
\newtheorem{remark}[theorem]{Remark}
\newtheorem{hyp}[theorem]{Hypothesis}
\numberwithin{equation}{section}

\newcommand{\id}{\text{Id}}

\def\RR{\mathbb{R}}
\def\BB{\mathbb{B}}

\def\NN{\mathbb{N}}
\def\mE{\mathbb{E}}

\def\ZZ{\mathbb{Z}}

\def\XX {\mathbb{X}}
\def\YY {\mathbb{Y}}

\newcommand{\ce}{{\mathcal E}}

\newcommand{\ch}{{\mathcal H}}

\newcommand{\cl}{{\mathcal L}}

\newcommand{\cs}{{\mathcal S}}

\newcommand{\cv}{{\mathcal V}}

\newcommand{\cx}{{\mathcal X}}

\def\al{{\alpha}}

\def\be{{\beta}}

\def\ga{{\gamma}}

 \newcommand{\vp}{\varphi}
 \newcommand{\ka}{\kappa}
 
\newcommand{\lp}{\left(}
\newcommand{\rp}{\right)}
\newcommand{\lc}{\left[}
\newcommand{\rc}{\right]}

\def \eref#1{\hbox{(\ref{#1})}}

\def\ll{\llbracket}
\def\rr{\rrbracket}

\def\wt{\tilde}
 
\def \eref#1{\hbox{(\ref{#1})}}

\begin{document}
\title[Euler scheme for fBm driven SDEs]
{First-order Euler scheme for SDEs driven\\
 by   fractional Brownian motions: the rough case}
\date{}   

\author[Y. Liu \and S. Tindel]{Yanghui Liu \and Samy Tindel} 

 
 \address{Yanghui Liu, Samy Tindel: Department of Mathematics,
Purdue University,
150 N. University Street,
W. Lafayette, IN 47907,
USA.}
\email{liu2048@purdue.edu, stindel@purdue.edu}

\thanks{S. Tindel is supported by the NSF grant  DMS-1613163}

 \begin{abstract}
 In this article, we consider the so-called modified Euler   scheme   for stochastic differential equations (SDEs)   driven by fractional Brownian motions (fBm) with Hurst parameter $\frac13<H<\frac12$.
This is a first-order time-discrete numerical approximation scheme, and has   been introduced in   \cite{HLN1} recently  in order to generalize the classical Euler scheme for It\^o SDEs to  the case  $H>\frac12$. 
The current contribution generalizes the modified Euler scheme to the rough case $\frac13<H<\frac12$. Namely,  we show a convergence rate of order $n^{\frac12-2H}$ for the scheme, and we argue that this rate is exact. We also derive a central limit theorem for the renormalized error of the scheme, thanks to  some new techniques for asymptotics of weighted random sums. Our main idea is based on the following observation: the triple of processes obtained by considering the fBm, the  scheme process and the normalized error process, can be lifted to a new rough path. In addition, the H\"older norm of  this new rough path has  an estimate  which is independent of the step-size of the scheme.
\end{abstract}

\maketitle

{
\hypersetup{linkcolor=black}
\setcounter{tocdepth}{1}
 \tableofcontents 
}

\section{Introduction} 
This note is concerned with the following differential equation driven by a $m$-dimensional fractional Brownian motion (fBm in the sequel) $B$ with Hurst parameter $ \frac13 <   H < \frac12$:
 \begin{eqnarray}\label{e1.1}
dy_{t} &=&b(y_{t})dt + {V}(y_{t}) d B_{t}\,, \quad t\in [0,T],
\\
y_{0}&=&y \in \RR^{d}.
\nonumber
\end{eqnarray}
Assuming that the collection of vector fields $b = (b^{i})_{1\leq i \leq d}$ belongs to $C^{2}_{b} (\RR^{d}, \RR^{d})  $ and ${V} = (V^{i}_{j})_{1\leq i\leq d, 1\leq j\leq m}$ sits in $C^{3}_{b} (\RR^{d}, \cl (\RR^{m}, \RR^{d})) $,  the theory of rough paths   gives a framework allowing to get existence and uniqueness results for  equation \eref{e1.1}. In addition, the unique solution $y$ in the rough paths sense has $\ga$-H\"older continuity for all $\ga<H$.   The reader is referred to~\cite{FH, FV, G}   for further details. 

In this paper, we are interested in the numerical approximation of equation \eref{e1.1} based on a  discretization of the time parameter $t$. For simplicity, we are considering a finite time interval $[0,T]$ and we take  the uniform partition $\pi: 0=t_{0}<t_{1}<\cdots<t_{n}=T$ on $[0,T]$. Specifically, for $k=0,\ldots,n$ we have $t_{k} =  k h$, where we denote $h = \frac{T}{n}$. Our generic approximation is called $y^{n}$, and it starts from the initial condition $y^{n}_{0} = y$. In order to introduce our numerical schemes, we shall also use the following notation:

\begin{notation}\label{not:iterated-vector-field}
Let $U=(U^{1},\dots, U^{d})$ and $V=(V^{1},\dots, V^{d})$ be two smooth vector fields defined on $\RR^{d}$. We denote by $\partial$ the vector $\partial = (\partial_{x_{1}},\dots, \partial_{x_{d}})$, that is, for $x\in\RR^{d}$ we have $[\partial U(x)]^{kl}=\partial_{x_{l}}U^{k}(x)$. With the same matrix convention, the vector field $\partial U V$ is   defined as  
$
[\partial UV(x)]^{k} = \sum_{l=1}^{d} \partial_{x_{l}}U^{k}(x) \, V^{l}(x)$.  
\end{notation}

With those preliminaries in mind, the most classical numerical scheme for stochastic equations is the so-called Euler scheme (or first-order Taylor scheme), which is recursively defined as follows on the uniform partition:
\begin{eqnarray*}
y^{n}_{t_{k+1}} &=& y^{n}_{t_{k}} + b(y^{n}_{t_{k}}) h + V(y^{n}_{t_{k}}) \delta B_{t_{k}t_{k+1}},
\end{eqnarray*}
where $\delta f_{st} $ is defined as $ f_{t} - f_{s} $ for a function $f$. 
However, it is easily seen (see e.g \cite{DNT} for details) that the Euler scheme  is divergent when the Hurst parameter $H$ is less than $\frac12$.  To obtain a convergent numerical approximation in this rough situation,  higher-order  terms from the Taylor expansion need to be included in the scheme.  Having the rough paths construction in mind, the simplest  method of this kind is the Milstein scheme, or second-order Taylor scheme. It can be expressed recursively as:
\begin{equation}\label{eq:second-order-taylor-scheme}
y^{n}_{t_{k+1}} = y^{n}_{t_{k}} + b(y^{n}_{t_{k}}) h+ V(y^{n}_{t_{k}}) \delta B_{t_{k}t_{k+1}} 
+\sum_{i,j=1}^{m}   \partial V_{i} V_{j} (y^{n}_{t_{k}}) \BB^{ij}_{t_{k}t_{k+1}},
\end{equation}
where we have used Notation \ref{not:iterated-vector-field} and where $\BB$ designates the second-order iterated integral of $B$ (See Section \ref{section3.1} for a proper definition). This numerical approximation 
  has first    been considered in \cite{D}, and has been shown to be convergent as long as $H>\frac13$, with an almost sure convergence rate $n^{-(3H-1)+\kappa}$.  Here and in the following $\kappa >0$ represents an arbitrarily  small constant. An extension of the result to $n$th-order Taylor schemes and to an abstract rough path with arbitrary regularity is contained in \cite{FV}; see also~\cite{HLN} for the optimized $n$th-order Taylor scheme when $H>\frac12$. 
  
    The $n$th-order Taylor schemes of the form \eqref{eq:second-order-taylor-scheme} are, however,    not  implementable in general.   This is due to the fact that when $i\neq j$ the terms  $\BB^{ij}_{t_{k}t_{k+1}}$ cannot be simulated exactly and have to be approximated on their own. We now mention some contributions giving implementable versions of \eqref{eq:second-order-taylor-scheme} for    stochastic differential equation \eref{e1.1}.   They all rely on some cancellation of the randomness in the error process $y-y^{n}$ related to our standing equation.
    
\noindent\emph{(i)}
The first second-order implementable scheme for \eref{e1.1} has been introduced in \cite{DNT}. It can be expressed in the following form:
  \begin{equation}\label{e2i}
y^{n}_{t_{k+1}} = y^{n}_{t_{k}} + b(y^{n}_{t_{k}}) h + V(y^{n}_{t_{k}}) \delta B_{t_{k}t_{k+1}} 
+\frac12 \sum_{i,j=1}^{m}\partial V_{i} V_{j} (y^{n}_{t_{k}}) \delta B^{i}_{t_{k}t_{k+1}}
\delta B^{j}_{t_{k}t_{k+1}}   .
\end{equation}
This scheme has been shown to have convergence rate of order $n^{-(H-\frac13)+\kappa}$,  and the proof relies on the fact that \eqref{e2i} is the 2nd-order Taylor scheme for the Wong-Zakai approximation of equation \eref{e1.1}.  The approximation \eqref{e2i} has been  extended in \cite{BFRS, FR} to a third-order scheme defined as follows:
\begin{multline}\label{e3}
y^{n}_{t_{k+1}} = y^{n}_{t_{k}} + b(y^{n}_{t_{k}}) h + V(y^{n}_{t_{k}}) \delta B_{t_{k}t_{k+1}} 
+\frac12 \sum_{i,j=1}^{m}\partial V_{i} V_{j} (y^{n}_{t_{k}})    \delta B^{i}_{t_{k}t_{k+1}}\delta B^{j}_{t_{k}t_{k+1}}
\\
+\frac16 \sum_{i,j=1}^{m}\partial (\partial V_{i} V_{j} )V_{k}(y^{n}_{t_{k}})  \delta B^{i}_{t_{k}t_{k+1}}\delta B^{j}_{t_{k}t_{k+1}}   \delta B^{k}_{t_{k}t_{k+1}} 
\, .
\end{multline}
Thanks to    a thorough analysis of   differences of    iterated integrals between two Gaussian processes, a convergence rate $n^{-(2H-\frac12) +\kappa}$ has been achieved for the scheme \eref{e3}. One should also notice that \cite{FR} handles in fact very general Gaussian processes, as long as their covariance function is regular enough in the $p$-variation sense.

\noindent\emph{(ii)}
A different direction has been considered in \cite{HLN1}, where the following first-order (meaning first-order with respect to the increments of $B$) scheme  has been introduced:
   \begin{equation}\label{e4}
y^{n}_{t_{k+1}} = y^{n}_{t_{k}} + b(y^{n}_{t_{k}}) h+ V(y^{n}_{t_{k}}) \delta B_{t_{k}t_{k+1}} 
+\frac12 \sum_{ j=1}^{m}\partial V_{j} V_{j} (y^{n}_{t_{k}}) h^{2H}.
\end{equation}
This approximation is called modified Euler scheme in \cite{HLN1}. As has been explained in \cite{HLN1},   the modified Euler scheme is a natural generalization of the classical Euler scheme of the Stratonovich SDE to the rough SDE \eref{e1.1}. For this reason, we will call \eref{e4} the Euler scheme from now on.  As the reader might also see from relation~\eqref{e4}, one gets the   Euler scheme from the second-order Taylor scheme \eqref{eq:second-order-taylor-scheme} by changing the terms $\BB^{ij}_{t_{k}t_{k+1}}$ into their respective expected values. 
Note that  since the   Euler scheme does not involve products of increments of the underlying fBm, its computation cost is much lower than those of~\eref{e2i} and~\eref{e3}. 
In spite of this cost reduction, an exact rate of convergence   $n^{-(  2H - \frac12)}$ has been achieved in~\cite{HLN1}. The asymptotic error distributions and the weak convergence of the scheme have also been considered, and those results heavily hinge on Malliavin calculus considerations. Notice however that the results in \cite{HLN1} are restricted to the case $H >\frac12$.

Having recalled those previous results, the aim of the current paper is quite simple: we wish to extend the results concerning the   Euler scheme~\eref{e4} to a truly rough situation. Namely, we will consider equation \eqref{e1.1} driven by a fractional Brownian motion $B$ with $\frac13 < H < \frac12$. For this equation, we show that the   Euler scheme maintains the   rate of convergence $n^{-(2H-\frac12)}$, which is the   same as the third-order implementable scheme in \eref{e3}.   
We also  obtain some asymptotic results for the error distribution of the numerical scheme~\eref{e4}, which generalize the corresponding results in \cite{HLN1, Kurtz} to the case $H<\frac12$.
More specifically, we will prove the following results (see Theorems \ref{thm7.3} and \ref{thm9.1} for more precise statements).
\begin{theorem}\label{thm:cvgce-scheme}
Let $y$ be the solution of equation \eref{e1.1}, and consider the Euler approximation scheme $y^{n}$ defined in \eref{e4}. Then

\noindent\emph{(i)} 
For any arbitrarily small $\kappa>0$, the following almost sure convergence holds true:
\begin{eqnarray*}
n^{2H-\frac12 -\kappa}  \sup_{t\in [0,T]} |y_{t}-y^{n}_{t}| \rightarrow 0 \quad \text{ as } n\rightarrow \infty;
\end{eqnarray*}

\noindent\emph{(ii)} 
The sequence of processes $n^{2H-\frac12} (y-y^{n})$ converges weakly in $D([0,T])$ to a process $U$ which solves the following equation:
\begin{eqnarray}\label{eq:dyn-U}
U_{t} &=& \int_{0}^{t} \partial b(y_{s}) U_{s}ds + \sum_{j=1}^{m} \int_{0}^{t} \partial V_{j} (y_{s}) U_{s} dB^{j}_{s} + \sum_{i,j=1}^{m} \int_{0}^{t} \partial V_{i}V_{j } (y_{s})   d W^{ij}_{s},
\end{eqnarray}
where $W = (W^{ij})$ is a $\RR^{m\times m}$-valued Brownian motion with correlated components, independent of $B$.
\end{theorem}
 Let us highlight the fact that our
  approach does  not rely on the special structure of the numerical scheme and does not require the analysis of the Wong-Zakai approximation. In fact, it   provides a general procedure for studying   time-discrete numerical approximations of RDEs, including the  implementable schemes \eref{e2i} and \eref{e3} we just mentioned, the backward Euler scheme, the Crank-Nicolson scheme and its modifications, Taylor schemes and their modifications introduced in \cite{HLN} and so on.  Let us now explain briefly our strategy in order to prove Theorem \ref{thm:cvgce-scheme}:

\noindent $(a)$
The fact that we are replacing the terms $\BB^{ij}_{t_{k}t_{k+1}}$ in \eqref{eq:second-order-taylor-scheme} by their expectations forces us to thoroughly analyze some discrete weighted sums in the second chaos of $B$. This will be a substantial part of our computations.

\noindent $(b)$
Then we have to get some uniform bounds (in $n$) for the numerical scheme \eqref{e4} and the related linear RDEs like \eqref{eq:dyn-U}. This will be done thanks to some discrete time rough paths type techniques, which replace  the Malliavin calculus arguments invoked in \cite{HLN1} . A special attention has obviously to be paid in order to handle the terms $V_{j} V_{j} (y^{n}_{t_{k}}) h^{2H}$ in~\eqref{e4}. Also note that the numerical scheme \eref{e4} will be reduced to a differential equation driven by a process in the second chaos (see \eref{e.10}). Hence we cannot rely on the integrability of  the Malliavin derivative of $y^{n}$, which is a linear RDEs driven by the scheme process.

\noindent $(c)$
One of the main ingredient of our proofs is to consider $(B, y^{n}, n^{\al} ( y-y^{n}) )$ as a rough path, where $\al>0$ is a constant less than $2H-\frac12$. This property will be proved by analyzing some crossed integrals between the 3 components of the process $(B, y^{n}, n^{\al} ( y-y^{n}) )$.

\noindent $(d)$
Eventually, our limit theorems are obtained by considering the asymptotic behavior of a rough linear equation describing the evolution of the error $y-y^{n}$, the center of which is  a weighted-variation term, or in our point of view, a ``discrete'' rough integral.  
 
 \noindent
If we compare the above ingredients with those in \cite{HLN1},
one can see that the gap between the rough and the non-rough situations is substantial. New ideas are required for every step (a) to (d).   This corresponds to the steps where the gap between the rough and the non rough situations is particularly large.

Among the ingredients we have alluded to above, the asymptotic behavior of weighted variations has been received a lot of attention in recent works; see e.g. \cite{LL, Nourdin, NN, NNT, NP1, NR}.  
   Our approach to this problem   relies on a combination of rough paths and Malliavin calculus tools, and might have an interest in its own right; see Theorem~\ref{prop9.2} for the precise statement. Indeed, with respect to the aforementioned results, it seems that we can reach a more general class of weights.   We are also able to consider the variations for  multi-dimensional fBms, thanks to a simple approximation argument on the simplex.

Eventually let us stress the fact that, though we have restricted our analysis to equations driven by a fractional Brownian motion here for sake of simplicity, we believe that our results can be extended to a general class of Gaussian processes whose covariance function satisfies reasonable assumptions (such as the ones exhibited in \cite{CHLT,FGGR}). In this case, if $X$ denotes the centered Gaussian process at stake, we expect the numerical scheme \eqref{e4} to become:
\begin{eqnarray*}
y^{n}_{t_{k+1}} = y^{n}_{t_{k}} + b(y^{n}_{t_{k}}) h+ V(y^{n}_{t_{k}}) \delta X_{t_{k}t_{k+1}} 
+
\frac12 \sum_{ j=1}^{m}\partial V_{j} V_{j} (y^{n}_{t_{k}}) R(h),
\end{eqnarray*}
where $R(h)$ is a deterministic constant defined by     $\mE(|X_{h}|^{2}  )$ and where we have assumed that $X$ has stationary increments. Notice however that (similar to what is done in \cite{DNT}) one can prove that in general the Euler scheme \eqref{e4} is divergent when one considers an equation driven by a fractional Brownian motion with Hurst parameter $\frac14 < H \le \frac13$. Our analysis is thus restricted to a rough signal with H\"older regularity greater than $\frac13$.

 Here is how our paper is structured: In Section \ref{section2}, we recall some results from the theory of rough paths, and prove a discrete version of the sewing map lemma.  In Section \ref{section3}, we consider the fractional Brownian motion as a rough path and derive   some elementary results.
 In Section \ref{section4i} we first develop some useful upper-bound estimates,   and then we introduce a general limit theorem on the asymptotic behavior of weighted random sums. 
  In Section \ref{section4} we consider the couple $(y^{n}, B)$  as a rough path and show that it is uniformly bounded in $n$. In Section~\ref{section5} we show that the Euler scheme $y^{n}$ is convergent, and we derive our first result on the rate of strong convergence of $y^{n}$. We also derive some estimates on the error process $y-y^{n}$.  Section \ref{section6} is devoted to an elaboration of the estimates for the error process under some new conditions. This leads, in Section \ref{section7}, to consider the rate of strong convergence of the Euler scheme again, improving the results obtained in Section \ref{section5} up to an optimal rate. In Section  \ref{section9} we prove our main result on the asymptotic error distribution of $y^{n}$. 
   In the Appendix we   prove some auxiliary results.
 
\paragraph{\textbf{Notation}}
Let $\pi:0=t_{0}<t_{1}<\cdots<t_{n}=T$ be a partition on $[0,T]$. Take $s,t\in [0,T]$. We write $\ll s,t \rr$ for the discrete interval that consists of   $t_{k} $'s such that   $ t_{k}\in [s,t] $.  We denote by $\cs_{k}([s,t])$ the simplex $\{ (t_{1},\dots, t_{k}) \in [s,t] ;\, t_{1}\leq \cdots\leq t_{k}  \}$. In contrast, whenever we deal with a discrete interval, we set $\cs_{k}(\ll s,t\rr)=
\{ (t_{1},\dots, t_{k}) \in \ll s,t\rr ;\, t_{1}< \cdots< t_{k}  \}$. For $t=t_{k}$ we   denote $t-:=t_{k-1}$, $t+:= t_{k+1}$.

Throughout the paper we work on a  probability space $(\Omega, \mathscr{F}, P)$. If $X$ is a random variable, we denote by $\| X \|_{p}  $ the $L_{p}$-norm of $X$.
The letter $K$ stands for a constant which can change from line to line, and $\lfloor a \rfloor $ denotes the integer part of   $a$.

\section{Elements of rough paths theory}\label{section2}

This section is devoted to introduce the main rough paths notations which will be used in the sequel. We refer to \cite{FH,FV} for further details. We shall also state and prove a discrete sewing lemma which is a simplified version of an analogous result contained in \cite{DPT}.

\subsection{H\"older continuous rough paths and rough differential equations} 

In this subsection, we introduce some basic concepts of the rough paths theory. Let $\frac13<\ga\leq \frac12$, and call $T>0$ a fixed finite time horizon. The following notation will prevail until the end of the paper: for a finite dimensional vector space $\cv$ and two functions $f\in C([0,T],\cv)$ and $g\in C(\cs_{2}([0,T]),\cv)$ we set 
\begin{equation}\label{eq:def-delta}
\delta f_{st} = f_{t}-f_{s},
\quad\text{and}\quad
\delta g_{sut} = g_{st}-g_{su}-g_{ut}.
\end{equation}

We start with the definition of some H\"older semi-norms:  consider here two paths $x \in C([0,T], \RR^{m})$ and $\XX \in C(\cs_{2}([0,T]), (\RR^{m})^{\otimes 2})$. Then we denote
\begin{equation}\label{eq:def-holder-seminorms}
\|x\|_{[s,t], \ga}:=\sup_{(u,v)\in\cs_{2}([s,t])}\frac{|\delta x_{uv}|}{|v-u|^{\ga}} , 
\qquad\qquad 
\|\XX\|_{[s,t],  2\ga}:= \sup_{(u,v)\in\cs_{2}([s,t])}\frac{|  \XX_{uv}| }{|v-u|^{ 2\ga}} ,
\end{equation}
where we stress the fact that the regularity of $\XX$ is measured in terms of $|t-s|$. When the semi-norms in \eqref{eq:def-holder-seminorms} are finite we say that $x$ and $\XX$ are respectively in $C^{\ga}([0,T], \RR^{m})$ and $C^{2\ga}(\cs_{2}([0,T]), (\RR^{m})^{\otimes 2})$.
For convenience, we denote $ \|x\|_{\ga}:= \|x\|_{[0,T], \ga} $ and $ \|\XX \|_{2\ga }:= \|\XX\|_{[0,T], 2\ga } $.

With this preliminary notation in hand, we can now turn to the definition of rough path.

\begin{Def}\label{def:rough-path}
Let $x \in C([0,T], \RR^{m})$, $\XX \in C(\cs_{2}([0,T]), (\RR^{m})^{\otimes 2})$, and $\frac13<\ga\leq \frac12$. 
We call $ S_{2}(x):=(x, \XX)  $ a (second-order) $\ga$-rough path if $ \|x\|_{\ga} <\infty $ and $\|\XX\|_{2\ga}<\infty$, and the following algebraic relation holds true: 
\begin{equation*}
\delta \XX_{sut}:=\XX_{st} - \XX_{su} - \XX_{ut} = x_{su}\otimes x_{ut} ,
\end{equation*}
where we have invoked \eqref{eq:def-delta} for the definition of $\delta\XX$. For a $\ga$-rough path $S_{2}(x)$, we define a $\ga$-H\"older semi-norm as follows:
\begin{equation}\label{eq:def-norm-rp}
\|S_{2}(x)\|_{\ga} := \|x\|_{\ga}+ \|\XX\|_{2\ga}^{\frac12}\,. 
\end{equation}
An important subclass of rough paths are the so-called \emph{geometric $\ga$-H\"older rough paths}.  A geometric $\ga$-H\"older rough path is a  rough path $  (x, \XX)$  such that  there exists a sequence of smooth $\RR^{d}$-valued paths $(x^{n}, \XX^{n})$ verifying:
\begin{eqnarray}\label{eq:cvgce-for-geom-rp}
\| x-x^{n}\|_{\ga} + \|\XX-\XX^{n}\|_{2\ga}  \rightarrow 0 \quad \text{ as } n \rightarrow \infty . 
\end{eqnarray}
We will mainly consider geometric rough paths in the remainder of the article.
\end{Def}

\noindent
In relation to \eqref{eq:cvgce-for-geom-rp}, notice that  when $x$ is a smooth $\RR^{d}$-valued path, we can choose \begin{eqnarray}\label{e2.2}
\XX_{st} &=& \int_{s}^{t}\int_{s}^{u} dx_{v} \otimes dx_{u}.
\end{eqnarray}
It is then easily verified  that $S_{2}(x) = (x, \XX)$, with $\XX$ defined in \eref{e2.2}, is a $\ga$-rough path with $\ga=1$.  In fact, this is also the unique way to lift a smooth path to a $\ga$-rough path.

Recall now that we interpret equation \eqref{e1.1} in the rough paths sense. That is, we shall consider the following general rough differential equation (RDE):
\begin{eqnarray}\label{e2.1}
dy_{t}&=&b(y_{t})dt+ V(y_{t})dx_{t}\,,
\nonumber
\\
y_{0} &=& y,
\end{eqnarray}
where $b$ and $V$ are smooth enough coefficients and $x$ is a rough path as given in Definition~\ref{def:rough-path}. We shall interpret equation \eqref{e2.1} in a way introduced by Davie in \cite{D}, which is conveniently compatible with numerical approximations.

 \begin{Def}\label{def:diff-eq-davie}
 We say that $y$ is a solution of \eref{e2.1} on $[0,T]$ if $y_{0} = y$ and there exists a constant $K>0$ and $\mu>1$  such that 
\begin{equation}\label{eq:dcp-Davie}
\Big| \delta y_{st} -  \int_{s}^{t} b(y_{u}) \, du - V(y_{s}) \delta x_{st} 
- \sum_{i,j=1}^{m} \partial V_{i}V_{j} (y_{s} ) \XX^{ij}_{st} \Big| 
\leq 
K |t-s|^{\mu}
\end{equation}
for all $(s,t) \in \mathcal{S}_{2}([0,T])$, where we recall that $\delta y$ is defined by~\eref{eq:def-delta}. 
\end{Def}

\noindent
Notice that if $y$ solves \eqref{e2.1} according to Definition \ref{def:diff-eq-davie}, then it is also a controlled process as defined in \cite{FH,G}. Namely, if $y$ satisfies  relation \eqref{eq:dcp-Davie}, then we also have:
\begin{equation}\label{eq:dcp-controlled-process}
\delta y_{st} = V(y_{s}) \delta x_{st} + r_{st}^{y},
\end{equation}
where $r^{y}\in C^{2\ga}(\cs_{2}([0,T]))$. We can thus define iterated integrals of $y$ with respect to itself thanks to the sewing map; see   Proposition 1 in \cite{G}. This yields the following decomposition:
\begin{eqnarray*}
\Big| 
\int_{s}^{t}   y_{u}^{i} d y^{j}_{u} - y^{i}_{s} \delta y^{j}_{st} -  \sum_{i', j'=1}^{m}  V^{i}_{i'}    V^{j}_{j'}(y_{s}) \XX_{st}^{i'j'}
\Big|&\leq& K(t-s)^{3\ga},
\end{eqnarray*}
for all $(s,t) \in \mathcal{S}_{2}([0,T])$ and $i,j=1,\ldots,n$. In other words, the signature type path $S_{2}(y) = (y, \YY )$ defines a rough path according to Definition \ref{def:rough-path}.

We can now state an existence and uniqueness result for rough differential equations. The reader is referred again to \cite{FH,G} for further details.
\begin{thm}\label{thm 3.3}
Assume that  $V= (V_j)_{1\leq j\leq m}$ is a collection of  $C^{\lfloor 1/\ga \rfloor+1}_{b}$-vector fields on $\RR^d$.
 Then there exists a unique   RDE solution to equation \eref{e2.1}, understood as in Definition~\ref{def:diff-eq-davie}. In addition, the unique solution $y$ satisfies the following estimate:
 \begin{eqnarray*}
|  S_{2}(y)_{st}| &\leq & K (1\vee \| S_{2}(x) \|_{\ga, [s,t] }^{1/\ga}) (t-s)^{\ga}.
\end{eqnarray*}
Whenever   $V = (V_j)_{1\leq j\leq m} $ is a collection of linear vector fields, the existence and uniqueness results still hold, and we have the estimate:
    \begin{eqnarray*}
|S_{2}(y)_{st}| &\leq& K_{1} \|S_{2}(x)\|_{\ga, [s,t]} \exp(K_{2} \|S_{2}(x)\|_{\ga }^{1/\ga}  ) (t-s)^{\ga}.
\end{eqnarray*}

 \end{thm}

\subsection{A discrete-time sewing map lemma }\label{section2.2}

In this subsection we derive a discrete version of the sewing map lemma which will play a prominent role in the analysis of our numerical scheme.  Let  $\pi : 0=t_{0}< t_{1}<\cdots <t_{n-1}< t_{n} =T $ be a generic partition of the interval $[0,T]$ for $n \in \NN$. For $0\leq s < t \leq T$, we denote by $ \llbracket  s, t \rrbracket  $ the discrete interval  $\{t_{k} : s\leq t_{k} \leq t \}$. We also label the following definition for further use.

\begin{Def}\label{def:discrete-increments}
We denote by  $\mathcal{C}_{2} ( \pi   ,\mathcal{X})$  the collection of functions    $R$   on $\cs_{2}(\llbracket 0,T \rrbracket)$ with values in a Banach space $(\mathcal{X}, |\cdot|)$ such that $R_{t_{k} t_{k+1}} = 0$ for $k=0,1,\dots, n-1$.   Similarly to the continuous case (relations~\eqref{eq:def-delta} and~\eqref{eq:def-holder-seminorms}), we define the operator $\delta$ and some H\"older semi-norms on $ \mathcal{C}_{2} ( \pi   ,\mathcal{X}) $   as follows:
\begin{equation*} 
\delta R_{sut} = R_{st}-R_{su}-R_{ut},
\quad\text{and}\quad
\|R\|_{\mu} = \sup_{ (u, v) \in \cs_{2}(\llbracket 0,T \rrbracket)}  \frac{|R_{uv}|}{|u-v|^{\mu}}\,.
\end{equation*}
For $R\in \mathcal{C}_{2} ( \pi   ,\mathcal{X})$ and $\mu>0$, we also  set
\begin{equation}\label{e1}
\|\delta R\|_{\mu} = \sup_{(s,u,t) \in  \cs_{3}(\llbracket 0,T \rrbracket )} 
\frac{|\delta R_{sut}|}{|t-s|^{\mu}}\, .
\end{equation}
The space of functions $R\in \mathcal{C}_{2} ( \pi   ,\mathcal{X})$ such that $\|\delta R\|_{\mu}<\infty$ is denoted by $\mathcal{C}_{2}^{\mu} ( \pi   ,\mathcal{X})$
\end{Def}

We can now turn to our discrete version of the sewing map lemma. This result is inspired by \cite{DPT}, but is included here since our situation is simpler and leads to a straightforward proof.
 
 \begin{lemma}\label{lem2.4}
 For a Banach space $\cx$, an exponent $\mu>1$ and $R \in \mathcal{C}_{2}^{\mu} ( \pi   ,\mathcal{X})$ as in Definition~\ref{def:discrete-increments}, the following relation holds true:
 \begin{equation*}
\|R\|_{\mu} \leq K_{\mu} \|\delta R\|_{\mu}\,,
\quad\text{where}\quad
K_{\mu} = 2^{\mu} \, \sum_{l=1}^{\infty} l^{-\mu}.
\end{equation*}
\end{lemma}

\begin{proof}  
Take $t_{i}, t_{j} \in \pi $.  Let
$\pi_{l}   $, $ l=1,\dots, j-i $ 
be partitions  on $\ll 0,T \rr$  defined recursively as follows: Set $\pi_{1} = \{ t_{i}, t_{j} \} $ and $\pi_{j-i} = \llbracket t_{i}, t_{j} \rrbracket \cap \pi  $. Given a partition $\pi_{ l}   = \{t_{i} = t^{l}_{0}<\cdots< t^{l}_{l}=t_{j}\} $ on $\llbracket t_{i}, t_{j} \rrbracket$, $l=2,\dots, j-i $,  we can find $ t^{l}_{k_{ l}} \in \pi_{l  }  \setminus \{t_{i}, t_{j}\} $ such that  
 \begin{eqnarray}\label{e1ii}
t^{l}_{k_{l}+1} - t^{l}_{k_{l}-1} \leq \frac{2(t_{j}-t_{i})}{ l-1}  .
\end{eqnarray}
 We denote
 by $\pi_{ l-1}$ the partition $\pi_{ l } \setminus \{t^{l}_{k_{l}}\}$. For $l=1,\ldots,j-i$, we also set 
 \begin{equation*}
R^{\pi_{l}} = \sum_{k=0}^{l-1} R_{t^{l}_{k}  t^{l}_{k+1}},
\end{equation*}
and we observe that
$
R^{\pi_{1}} = R_{t_{i} t_{j}}$,
and
$R^{\pi_{j-i}} = \sum_{   {k} ={i} }^{{j-1}}  R_{t_{k} t_{k+1}}=0$,
where the last relation is due to the fact that $R \in \mathcal{C}_{2} ( \pi   ,\mathcal{X})$.

With those preliminaries in hand, we can decompose $R_{t_{i} t_{j}}$ as follows: we write
\begin{eqnarray}\label{e1i}
R_{t_{i} t_{j}} = R_{t_{i} t_{j}}  - \sum_{   {k} ={i} }^{{j-1}}  R_{t_{k}  t_{k+1}}   = \sum_{l=2}^{j-i}(R^{\pi_{l-1}} - R^{\pi_{l }}).
\end{eqnarray}
Now, according to the definition of $\pi_{l}$, we have:
 \begin{eqnarray*}
\left|R^{\pi_{l-1} } -R^{\pi_{l }}\right|=\left| \delta R_{t^{l}_{k_{l} -1}  t^{l}_{k_{l} }  t^{l}_{k_{l} +1}  }  \right|
 \leq  \|\delta R \|_{\mu} \lp  t^{l}_{k_{l} +1}  -  t^{l}_{k_{l} -1} \rp^{\mu}
 \leq   \|\delta R \|_{\mu}   \frac{2^{\mu}(t_{j}-t_{i})^{\mu}}{ (l-1)^{\mu}}    ,
\end{eqnarray*}
where the first inequality follows from \eref{e1} and the second   from \eref{e1ii}. Applying the above estimate of $\left|R^{\pi_{l-1} } -R^{\pi_{l }}\right|$ to \eref{e1i}, we obtain
\begin{eqnarray*}
|R_{t_{i} t_{j}}| \leq  2^{\mu}(t_{j}-t_{i})^{\mu} \|\delta R \|_{\mu}  \sum_{l=1}^{j-i-1}  \frac{1}{ l^{\mu}}  
\leq
 K_{\mu} (t_{j}-t_{i})^{\mu} \|\delta R \|_{\mu}\,. 
\end{eqnarray*}
Dividing both sides of the above inequality by $ (t_{j}-t_{i})^{\mu}  $ and taking supremum over $(t_{i}, t_{j})$ in $\cs_{2}(\llbracket 0, T \rrbracket)$, we obtain the desired estimate. 
\end{proof}

\section{Elements of Fractional Brownian motions} \label{section3}

In this section we briefly recall the construction of a rough path above our fBm $B$. The reader is referred to \cite{FV} for further details. In the second part of the section, we turn to some  estimates for the L\'evy area of $B$ on a discrete grid, which are essential in the analysis of our scheme.  

\subsection{Enhanced fractional Brownian motion}\label{section3.1}
 
 Let $B=(B^1,\dots, B^m)$ be a standard $m$-dimensional fBm  on $[0,T]$ with Hurst parameter $H \in (\frac13, \frac12)  $.  Recall that the covariance function of each coordinate of $B$ is defined on $\mathcal{S}_{2}([0,T])$ by:
\begin{equation}\label{eq:cov-fbm}
R(s,t) = \frac{1}{2} \lc |s|^{2H} + |t|^{2H} - |t-s|^{2H}  \rc .
\end{equation}
We start by reviewing some properties of the covariance function of $B$ considered as a function on $(\mathcal{S}_{2}([0,T]))^{2}$. Namely, take $u,v,s,t$ in $[0,T]$ and set 
\begin{eqnarray}\label{eq3.1}
R\Big(\begin{matrix} u & v \\ s & t \end{matrix}  \Big) 
&=& \mE\lc \delta B^{j}_{uv} \delta B^{j}_{st} \rc.
\end{eqnarray}
Then, whenever $H>1/4$, it can be shown that the integral $\int R \, dR$ is well-defined as a Young integral in the plane. Furthermore, if   intervals $[u,v]$ and $[s,t]$ are disjoint, we have 
\begin{equation}\label{eq:covariance-disjoint-intv}
R \Big(\begin{matrix} u & v \\ s & t \end{matrix}  \Big)  
= 
 \int_{u}^{v}\int_{s}^{t} \mu( dr'dr) .
\end{equation}
 Here and in the following, we denote  \begin{eqnarray}\label{eq:def-msr-mu}
\mu(dr'dr) &=& - H(1-2H) |r-r'|^{2H-2} dr'dr.
\end{eqnarray}

Using the elementary properties above, it is shown in \cite[Chapter 15]{FV} that for any piecewise linear or mollifier approximation $ B^{n} $ to $B$, the geometric rough path $S_{2}(B^{n})$ converges in the $\ga$-H\"older semi-norm to a $\ga$-geometric rough path $S_{2}(B):=(B, \BB)$ (given as in Definition \ref{def:rough-path}) for $\frac13 <\ga <H$.   In addition, for $i\neq j$ the covariance of $\BB^{ij}$ can be expressed in terms of a $2$-dimensional Young integral: 
\begin{eqnarray}\label{e3.1}
\mE(\BB^{ij}_{u v} \BB^{ij}_{s t})  &=&  \int_{u}^{v} \int_{s}^{t} R\Big(\begin{matrix} u & r \\ s & r' \end{matrix}  \Big)   d  R(r',r)
.
\end{eqnarray}
It is also established in \cite[Chapter 15]{FV} that $S_{2}(B)$ enjoys the following integrability property.
\begin{prop}\label{prop:integrability-signature}
Let $S_{2}(B):=(B, \BB)$ be the rough path above $B$, and $\ga\in(\frac13, H)$. Then there exists a random variable $L_{\ga}\in \cap_{p\ge 1}L^{p}(\Omega)$ such that $\|S_{2}(B)\|_{\ga}\le L_{\ga}$, where $\| \cdot \|_{\ga}$ is defined by \eref{eq:def-norm-rp}.
\end{prop}

We now specialize \eqref{e3.1} to a situation where $(u,v)$ and $(s,t)$ are disjoint intervals  such that $u<v<s<t$. In this case relation \eqref{eq:covariance-disjoint-intv} enables to write:
\begin{eqnarray}\label{e3.3}
\mE(\BB^{ij}_{u v} \BB^{ij}_{s t})  &=&
  \int_{u}^{v} \int_{s}^{t} 
\int_{u}^{r}\int_{s}^{r'} \mu(dw'dw)\mu( dr'dr),
\end{eqnarray}
where $\mu$ is the measure given by \eqref{eq:def-msr-mu}. Note that the left-hand side of \eref{e3.3} converges to $\mE(\BB^{ij}_{u v} \BB^{ij}_{s t}) $
as $v\rightarrow s$. Therefore,  the quadruple integral in \eref{e3.3} converges as $v\rightarrow s$. This implies that  the quadruple integral exists   and   identity \eref{e3.3} still holds when $s=v$. 
 
 Having relation \eqref{e3.3} in mind, let us label the following definition for further use.

\begin{Def}\label{def3.2}
 Denote by $\mathcal{E}_{[a,b]}$ the set of step functions on an interval $[a,b] \subset [0,T]$. We call $\mathcal{H}_{[a,b]}$ the Hilbert space defined as  the closure of $ \mathcal{E}_{[a,b]} $ 
with respect to the scalar product
\begin{eqnarray*}
\langle \mathbf{1}_{[u,v]}, \mathbf{1}_{[s,t]} \rangle_{\mathcal{H}_{[a,b]}} &=&  R\Big(\begin{matrix} u & v \\ s & t \end{matrix}  \Big) .
\end{eqnarray*}
In order to alleviate notations, we will still write $\mathcal{H}=\mathcal{H}_{[a,b]}$ when $[a,b]=[0,T]$. 
Notice that the mapping $\mathbf{1}_{[s,t]} \rightarrow \delta B_{st}$ can be extended to an isometry between $\mathcal{H}_{[a,b]}$ and the Gaussian space associated with $\{ B_{t}, t\in [a,b]\}$. 
 We denote this isometry by $h \rightarrow \int_{a}^{b} h\delta B $. The random variable $\int_{a}^{b} h\delta B$   is   called the (first-order) Wiener integral and is also denoted by $I_{1}(h)$. 
 \end{Def}
 
Owing to the fact that $H<1/2$ throughout this article,
we have the following identity:
\begin{eqnarray}\label{e3.5i}
\|h\|_{\mathcal{H}_{[a,b]}} &=& \Big\| d_{H} s^{\frac12 - H} ( D^{\frac12 -H}_{T-} u^{H-\frac12} \phi (u) ) (s) \Big\|_{L^{2}([a,b])},
\end{eqnarray}
where $d_{H} $ is a constant depending on $H$ and $D^{\frac12 -H}_{T-}$ is the right-sided fractional differentiation operator; see (5.31) in \cite{N}. With the help of     \eref{e3.5i} it is easy to derive 
 the following relation  for $1>\kappa >0$ and $\ga>\frac12 - H$:
\begin{eqnarray}\label{e3.5}
K_{1}\|h\|_{L^{2-\kappa}([a,b])} \leq \|h\|_{\mathcal{H}_{[a,b]}} \leq K_{2}(\sup_{t\in [a,b]}h(t) +\|h\|_{C^{\ga} ([a,b])})\, .
\end{eqnarray} 
Indeed, the lower-bound inequality can be obtained by  the Hardy-Littlewood inequality, while the upper-bound estimate follows from the definition of the   fractional derivative.

In order to generalize relations \eqref{eq:covariance-disjoint-intv}   to a more general situation,
recall that for  $h \in \mathcal{H}_{[a,b]}$ we have  $h_{n} \in \mathcal{E}_{[a,b]}$ such that $h_{n} \rightarrow h$ in $\mathcal{H}_{[a,b]}$. We denote by $h^{e}_{n}$ the extension of $h_{n}$   on $[0,T]$ such that $h_{n}^{e}=h_{n}$ on $[a,b]$ and $h_{n}^{e} =0$ on $[0,T] \setminus [a,b]$. 
Then $\int_{0}^{T} h_{n}^{e}dB = \int_{[a,b]}  h_{n}  dB   $  is a Cauchy sequence in $L^{2}(\Omega)$, and thus so is $h_{n}^{e}$  in $\mathcal{H}$. We denote the limit of $h^{e}_{n}$ by $h^{e} $.
 It is easy to see that 
$h^{e} \in \mathcal{H}$ satisfies
$h^{e}|_{[a,b]} = h$, $h^{e}|_{[a,b]^{c}} = 0$, and
 $\|h^{e}\|_{\mathcal{H}} = \|h\|_{\mathcal{H}_{[a,b]}}$. 
 \begin{lemma}\label{lem3.1}
Take $f \in \mathcal{H}_{[a,b]}$ and $g \in \mathcal{H}_{[c,d]}$, where $[a,b]$ and $[c,d]$ are disjoint subintervals of $[0,T]$ such that  $a<b<c<d$. Then  the following identity holds true:
\begin{eqnarray}\label{e3.4iii}
\mE\Big(  \int_{[a,b]} f \delta B \int_{[c,d]} g \delta B  \Big) =\mE\Big(  \int_{0}^{T} f^{e} \delta B \int_{0}^{T} g^{e} \delta B  \Big) = \int_{[a,b]} \int_{[c,d]} f_{t} g_{s} \mu(ds\, dt),
\end{eqnarray}
where $\mu$ is the measure defined by \eqref{eq:def-msr-mu}.
\end{lemma}
\begin{proof}
 Take $f_{n}  \in \mathcal{E}_{[a,b]} $ and $g_{n} \in \mathcal{E}_{[c,d]}$ such that $f_{n} \rightarrow f$ in $\mathcal{H}_{[a,b]}$ and $g_{n} \rightarrow g $ in $\mathcal{H}_{[c,d]}$. Then we have 
\begin{eqnarray}\label{e3.4ii}
\langle f_{n}^{e}, g_{n}^{e} \rangle_{\mathcal{H}}
&\rightarrow&  \langle f^{e}, g^{e} \rangle_{\mathcal{H}} \quad \text{as} \quad n\rightarrow \infty.
\end{eqnarray}
On the other hand, take  $d([a,b],[c,d])>\kappa>0$, then owing to \eref{eq:def-msr-mu} we have
   \begin{eqnarray}\label{e3.4i}
-\int_{[a,b]}\int_{[c,d]} |f_{t} | \cdot |g_{s}| \mu(ds\,dt) &\leq& H(1-2H)\kappa^{2H-2}  \|f^{e}\|_{L^{1}([0,T])}  \|g^{e}\|_{L^{1}([0,T])}  .
\end{eqnarray}
In particular, the left hand side of \eref{e3.4i} is finite. Since 
\begin{eqnarray*}
\langle f_{n}^{e}, g_{n}^{e} \rangle_{\mathcal{H}}
 &=& \int_{[a,b]}\int_{[c,d]} f_{n}(t) g_{n}(s) \mu(ds\,dt), 
\end{eqnarray*}
we can write
\begin{eqnarray}\label{e3.9ii}
&&
\langle f_{n}^{e}, g_{n}^{e} \rangle_{\mathcal{H}}
- \int_{[a,b]}\int_{[c,d]} f (t) g (s) \mu(ds\,dt)
\nonumber
\\
&&=  \int_{[a,b]}\int_{[c,d]}( f_{n}  -f  )(t) g_{n} (s)\mu(ds\,dt) +  \int_{[a,b]}\int_{[c,d]} f (t) (g_{n} -g)(s) \mu(ds\,dt).
\end{eqnarray}
Applying \eref{e3.4i} to the right-hand side of \eref{e3.9ii}  and taking into account that $\|f_{n} - f\|_{L^{1}([a,b])}\rightarrow 0$ and $\|g_{n}-g\|_{L^{1}([c,d])}\rightarrow 0$ as $n\rightarrow \infty$, which follow from \eref{e3.5}, we obtain
\begin{eqnarray}\label{e3.6}
\langle f_{n}^{e}, g_{n}^{e} \rangle_{\mathcal{H}}
&\rightarrow&  
\int_{[a,b]}\int_{[c,d]} f_{t} g_{s} \mu(ds\,dt) 
 \quad \text{as} \quad n\rightarrow \infty.
\end{eqnarray}
The identity \eref{e3.4iii} then follows from   \eref{e3.4ii} and \eref{e3.6} and the uniqueness of the limit of $\langle f_{n}^{e}, g_{n}^{e} \rangle_{\mathcal{H}}$. 
\end{proof}

\subsection{Upper-bound estimates for a L\'evy area type process}
Let $(B, \BB)$ be an enhanced  fractional Brownian motion as in the previous subsection. We now go back to the discrete interval $\ll 0,T \rr$ considered in Section \ref{section2.2}.
We denote by $F^{ij}_{t} $ the process on $\ll 0,T \rr $ such that 
\begin{eqnarray}\label{e4.1}
F^{ij}_{0}=0,
\quad\quad
F^{ij}_{t} &=& \begin{cases}
  \sum\limits_{t_{k}=0}^{t_{-}} \BB^{ij}_{t_{k},t_{k+1}}  \quad & i\neq j  
\vspace{.2cm} \\
    \sum\limits_{t_{k}=0}^{t_{-}}  \left( \BB^{ii}_{t_{k}t_{k+1}}- \frac12 h^{2H}\right)   \quad & i= j  
 \end{cases} 
\end{eqnarray}
 for $ t >0$, where $t_{-}=t_{j-1}$ if $t=t_{j}$ and where we recall that $h=t_{j}-t_{j-1}=\frac{T}{n}$.  
 
\begin{lemma}\label{lem4.1}
For $F^{ij} $ defined as in \eref{e4.1}, we have the following estimate
\begin{eqnarray}\label{e4.2}
\| \delta F^{ij}_{s t}\|_{p} &\leq& K_{p} n^{\frac12 -2H} (t-s)^{\frac12}, \quad \quad (s,t) \in \cs_{2} (\ll0,T\rr),
\end{eqnarray}
where $K_{p} $ is a constant depending on $p$, $H$ and $T$, and $\|\cdot\|_{p}$ denotes the $L_{p}(\Omega)$-norm.
\end{lemma}

\begin{proof} We only consider   the case $i=j$. The case $i\neq j$ can be considered similarly. Since $\delta F^{ij}_{s t}$ is a random variable in the second chaos of $B$, some hyper-contractivity arguments show that it suffices to consider the case   $p=2$ in \eref{e4.2}.  On the other hand, it is clear that $ \BB^{ii}_{t_{k}t_{k+1}} - \frac12 h^{2H} $ is equal to $\frac12 h^{2H}H_{2}(B_{k,k+1})$ in distribution, where $H_{2}(x)=x^{2}-1$. So we are reduced to estimate the following quantity:
\begin{eqnarray}\label{eq3.16}
\|F^{ij}_{t}\|_{2}= \frac12 h^{2H} \big\|\sum_{t_{k}=0}^{t-}H_{2}(B_{k,k+1}) \big\|_{2}\,.
\end{eqnarray}
A classical result in \cite{BM} shows that $\big\|\sum_{t_{k}=0}^{t-}H_{2}(B_{k,k+1}) \big\|_{2} \leq K(t-s)^{\frac12}\sqrt{n}$; see also Theorem 7.4.1 in \cite{NP}. Applying this relation to \eref{eq3.16} we obtain the estimate \eref{e4.2}.
\end{proof}

The following result provides a way to find a uniform almost sure upper-bound estimate for   a sequence of  stochastic processes. 
\begin{lemma}\label{lem4.2}
Let $\{X^{n}; \, n \in \NN\}$ be a sequence of stochastic processes such that
\begin{eqnarray}\label{e4.3}
\|X^{n}_{t} - X^{n}_{s}\|_{p} &\leq& K_{p} n^{-\al} (t-s)^{\beta}
\end{eqnarray}
for all $p\geq 1$, where $K_{p}$ is a constant depending on $p$.
 Then for $0<\ga<\beta$ and $\kappa>0$, we can find an integrable random variable $G_{\ga, \kappa}$ independent of $n$ such that
\begin{eqnarray*}
 \|X^{n}\|_{\ga} &\leq& G_{\ga,\kappa}n^{-\al+\kappa}.
\end{eqnarray*}
\end{lemma}

\begin{proof} 
Take $p\geq 1$ such that $0<\ga<\beta-1/p$. The Garsia-Rodemich-Rumsey  lemma (see~\cite{GRR}) implies that 
\begin{eqnarray*}
\|   X^{n}  \|_{\ga}^{p} &\leq& K_{p}     \int_{0}^{T}\int_{0}^{T} \frac{|X^{n}_{u} - X^{n}_{v} |^{p}}{|u-v|^{2+p\ga}} dudv  . 
\end{eqnarray*}
Taking expectation on both sides  and taking into account the inequality \eref{e4.3}  we obtain
\begin{eqnarray*}
\mE  \left[ \|   X^{n}  \|_{\ga}^{p}\right] &\leq& K_{p}  \int_{0}^{T}\int_{0}^{T} \frac{ \mE \left[ |X^{n}_{u} - X^{n}_{v} |^{p}\right]  }{|u-v|^{2+p\ga}} dudv  
\\
&\leq& K_{p} n^{-p \al },
\end{eqnarray*}
and the last inequality can be recast as:
\begin{eqnarray*}
\mE \left[ \| n^{\al-\kappa}  X^{n}  \|_{\ga}^{p} \right]&\leq& K_{p} n^{-p\kappa}.
\end{eqnarray*}
We now choose $p$ such that $p>1/\kappa$. Therefore the above estimate implies that:
\begin{equation*}
\mE\Big[ \sup_{n \in \NN}   \| n^{\al-\kappa}  X^{n}  \|_{\ga}^{p} \Big]
\leq
\mE\Big[ \sum_{n \in \NN}    \| n^{\al-\kappa}  X^{n}  \|_{\ga}^{p} \Big]
\\
\leq
K_{p}\sum_{n \in \NN}   n^{-p\kappa}<\infty.
\end{equation*}
In partitular, we obtain that $\sup_{n \in \NN}   \| n^{\al-\kappa}  X^{n}  \|_{\ga}^{p}$ is an integrable random variable. By taking $G_{\ga, \kappa} =\sup_{n \in \NN} \| n^{\al-\kappa}  X^{n}  \|_{\ga}$, we obtain the desired estimate for $\|X^{n}\|_{\ga}$.
\end{proof}

\begin{remark}\label{remark3.6}
Let $\ga$ be a parameter such that $\frac13<\ga<H$. Starting from relation \eqref{e4.2} and taking into account the fact that $|t-s|\ge \frac{T}{n}$ for all $(s,t )\in \mathcal{S}_{2}(\ll 0,T \rr)$, it is readily checked that the increment $F $ introduced in  Lemma \ref{lem4.1} satisfies 
\begin{equation}\label{eq:bnd-Lp-F-scale-beta}
\| \delta F_{st}\|_{p}\leq K_{p}n^{\beta -2H} (t-s)^{\beta},
\end{equation}
for all $2\ga < \beta<2H$ and $(s,t) \in \cs_{2}(  \ll 0,T \rr) $.  By considering  the linear interpolation of $F $ on $[0,T]$, inequality~\eqref{eq:bnd-Lp-F-scale-beta}  also holds for all $ (s,t) \in \cs_{2}( [0,T]) $.
Owing to Lemma \ref{lem4.2},      we can thus find an integrable random variable $G_{\ga}  $ such that for any $\ga:\frac13<\ga<H$ we have
\begin{eqnarray}\label{e4.21}
|   \delta F^{ij}_{st}| &\leq &   G_{\ga}    (t-s)^{ 2\gamma }  \quad a.s.
\end{eqnarray}
 \end{remark}

  \section{Weighted random sums via the rough path approach}\label{section4i}
  
In this section, we   derive some useful upper-bound estimates for weighted random sums related to $B$. In the second part of the section, we   prove a general limit theorem, which is  our   main result of this section. 
\subsection{Upper-bound estimates for weighted random sums}
We now derive some estimates for weighted random sums. As has been mentioned in the introduction,  these results only require the weight function to satisfy some proper regularity conditions. 
 \begin{prop}\label{prop3.4}
 Let $f$ and $g$ be processes on $\ll0,T\rr$  such that $|\delta f_{st}| \leq G(t-s)^{\al}$ and $|\delta g_{st}| \leq G(t-s)^{\beta}$, where $\al +\beta>1$. We define an increment $R$ on $\cs_{2}(\ll0,T\rr)$ by:
 \begin{eqnarray*}
R_{st} &=& \sum_{t_{k} = s}^{t-} \delta f_{st_{k}} \delta g_{t_{k}t_{k+1}}.
\end{eqnarray*}
 Then the following estimate holds true:
 \begin{eqnarray*}
| R_{st}  | &\leq& G (t-s)^{\al+\beta}, 
\quad\text{for all}\quad (s,t)\in \cs_{2}(\ll0,T\rr).
\end{eqnarray*}

 \end{prop}

 \begin{proof} 
It is clear that $R_{t_{k}t_{k+1}} = \delta f_{t_{k}t_{k}} \delta g_{t_{k}t_{k+1}} = 0 $. In addition, the following relation is readily checked, where we recall that $\delta R$ is defined by \eqref{eq:def-delta}:
\begin{eqnarray*}
\delta R_{sut} &=& \delta f_{su} \, \delta g_{ut},
\quad\text{for all}\quad (s,u,t)\in \cs_{3}(\ll0,T\rr)
\,.
\end{eqnarray*}
Therefore we have $|\delta R_{sut}| \le G |t-s|^{\al+\beta}$. Since we have assumed $\al+\beta>1$,
we can invoke the discrete sewing map Lemma \ref{lem2.4}, which yields:
\begin{eqnarray*}
\|R\|_{\al+\beta} \leq K\| \delta R \|_{\al+\beta} \leq G. 
\end{eqnarray*}
The proposition then follows immediately. 
\end{proof}
 \begin{remark}\label{remark3.8}
 The Riemann-Stieltjes  sum $ \sum_{t_{k} = s}^{t-} \delta f_{st_{k}} \delta g_{t_{k}t_{k+1}}$ in Proposition \ref{prop3.4} can be thought of as a $\RR$-valued ``discrete'' Young integral. One can also consider $L_{p}$-valued ``discrete'' Young integrals in a similar way by viewing $f$ and $g$   as functions with values in $L_{p}$. This will lead us to an $L_{p}$-estimate of $\sum_{t_{k} = s}^{t-} \delta f_{st_{k}} \delta g_{t_{k}t_{k+1}}$. Precisely, suppose that $f$ and $g$ are processes such that $\|\delta f_{st}\|_{p} \leq K (t-s)^{\al}$ and $ \|\delta g_{st}\|_{p} \leq K (t-s)^{\beta} $ for all $p\geq 1$. Then we have
 \begin{eqnarray*}
\Big\|\sum_{t_{k} = s}^{t-} \delta f_{st_{k}} \delta g_{t_{k}t_{k+1}}\Big\|_{p} &\leq& K (t-s)^{\al+\beta}.
\end{eqnarray*}
 
 \end{remark}
 In the sequel, we consider an application of Proposition  \ref{prop3.4} to third-order terms in our Taylor expansion for equation \eqref{e1.1}. Towards this aim, we first need the following estimate in $L_{p}(\Omega)$. They are somehow reminiscent of the estimates for triple integrals in~\cite{FR}, though our main focus here is on cumulative sums of triple integrals. 
 
 \begin{lemma} \label{lem11.2}
 Let $B$ be an $\RR^{m}$-valued fractional Brownian motion with Hurst parameter $H>\frac14$.
For a fixed set of coordinates $i,j,l \in\{1,\dots, m\}$, we define two increments $\zeta=\zeta^{ijl}$ and $\delta g=\delta g^{ijl}$on $\cs_{2}(\ll 0,T\rr)$ as follows:
\begin{equation}\label{eq:def-zeta-and-g}
\zeta_{st}^{ijl}  =  
\int_{ s}^{ t }  \int_{ s}^{u} \int_{ s}^{v} d B^{i}_{r}   d B^{j}_{v}   d B^{l}_{u},
\quad\text{and}\quad
\delta g_{st} = \sum_{t_{k}=s}^{t-}   \zeta_{t_{k} t_{k+1}}\,.
\end{equation}
 Then the following estimate is valid for $(s,t)\in\cs_{2}(\ll 0,T\rr)$:
\begin{eqnarray}\label{e11.0}
\| \delta g_{st} \|_{p} &\leq& K n^{\frac12 -3H} (t-s)^{\frac12}. 
\end{eqnarray}
 \end{lemma}
 
 \begin{proof}  When $i=j=l$ we have $\zeta_{st} = \frac16 ( \delta B^{i}_{st})^{3}$ and  the estimate \eref{e11.0} follows from the classical results in \cite{BM, GS}. In the following we consider the case when $i,j$,  $l$ are not all equal.
 
Let us further reduce our problem. First, since the fBm has stationary increment it suffices to prove the lemma for $s=0$.  Furthermore, by self-similarity of the fBm, we have the following equation in distribution:
 \begin{eqnarray*}
\delta g_{0t}     & {=}& T^{3H} n^{-3H} \sum_{ {k}=0}^{\frac{nt}{T} -1}    \zeta_{k,k+1} \, .
\end{eqnarray*}
As a last preliminary step, note that $\zeta$ takes values in the third chaos of $B$, on which all $L^{p}$-norms are equivalent. Hence our claim \eqref{e11.0} boils down to prove: 
\begin{eqnarray}\label{e11.1}
\Big\|   \sum_{ {k}=0}^{\frac{nt}{T} -1} \zeta _{k,k+1}   \Big\|_{2}^{2} &\leq& Knt .
\end{eqnarray}
We now focus on this inequality.

We first consider the case when $i, j$ and $ l$ are different from each other. In order to prove relation \eqref{e11.1}, write
 \begin{eqnarray}\label{e11.1i}
\Big\|   \sum_{ {k}=0}^{\frac{nt}{T} -1} \zeta_{k,k+1}   \Big\|_{2}^{2} &=&  \sum_{ |{k}- k'|\leq 1} \mE (  \zeta_{k,k+1}   \zeta_{k' ,k'+1})+    \sum_{ |{k}- k'|>1} \mE (  \zeta_{k,k+1}   \zeta_{k' ,k'+1}) \,.
\end{eqnarray}
For the sum $\sum_{ |{k}- k'|>1}$ above,
thanks to the independence of $B^{i}$, $B^{j}$ and $B^{l}$, one can apply   Lemma \ref{lem3.1} twice in order to get:   
\begin{eqnarray*}
\mE( \zeta_{k,k+1}   \zeta_{k',k'+1}) 
&=&   \int_{k}^{k+1} \int_{k'}^{k'+1} 
\int_{k}^{u} \int_{k'}^{u'} \int_{k}^{v} \int_{k'}^{v'}  \mu( dr'dr) \mu( dv'dv) \mu( d u'du).
\end{eqnarray*}
 In particular, since \eqref{eq:def-msr-mu} reveals that $\mu$ is a negative measure whenever $H<1/2$, we have
 \begin{eqnarray}\label{e11.2}
\sum_{ |{k}- k'|>1} \mE( \zeta_{k,k+1}   \zeta_{k',k'+1}) &\leq&0.
\end{eqnarray}
 Moreover,  since $  \|     \zeta_{k, k+1}    \|_{2} =  \|     \zeta_{0,1}    \|_{2}  $ and $\mE (  \zeta_{k,k+1}   \zeta_{k+1 ,k+2}) = \mE (  \zeta_{0,1}   \zeta_{1 ,2})$\,,   the following bound is easily checked:
 \begin{eqnarray}\label{e11.3}
  \sum_{ |{k}- k'|\leq 1} | \mE (  \zeta_{k,k+1}   \zeta_{k' ,k'+1}) | &\leq& Knt.
\end{eqnarray}
Applying \eref{e11.2} and \eref{e11.3} to the right-hand side of \eref{e11.1i}, we obtain \eref{e11.1}. 

Assume now that $i=j\neq l$. Then 
\begin{equation*}
\zeta_{k,k+1} =\frac12 \int_{k}^{k+1}(\delta B^{i}_{ku})^{2} dB^{l}_{u}
=
\frac12\lp \zeta^{1,il}_{k,k+1} +  \zeta^{2,l}_{k,k+1} \rp,
\end{equation*}
where we have set 
\begin{equation}\label{eq:def-zeta1-zeta2}
\zeta^{1,il}_{k,k+1}
=
\int_{k}^{k+1}  [  (\delta B^{i}_{ku})^{2} - (u-k)^{2H} ]dB^{l}_{u},
\quad\text{and}\quad
\zeta^{2,l}_{k,k+1}
=
\int_{k}^{k+1}    (u-k)^{2H} dB^{l}_{u}.
\end{equation}
We now treat the terms $\zeta^{1,il}$ and $\zeta^{2,l}$ similarly to the case of different indices $i,j,l$. Namely, we decompose $\|   \sum_{ {k}=0}^{\frac{nt}{T} -1} \zeta_{k,k+1}^{2,l}   \|_{2}^{2}$ as in 
\eref{e11.1i}. Then in the same way as for \eqref{e11.2} and \eqref{e11.3}, we can show that:
\begin{eqnarray*} 
\sum_{ |{k}- k'| > 1}\mE( \zeta_{k,k+1}^{2,l}   \zeta_{k',k'+1}^{2,l}) \leq 0,
\quad\text{and}\quad
\sum_{ |{k}- k'|\leq 1} \mE (  \zeta^{2,l}_{k,k+1}   \zeta^{2,l}_{k' ,k'+1}) &\leq& Knt.
\end{eqnarray*}
One can thus easily show that    $\zeta^{2,l}_{k,k+1}$ satisfies the inequality
\eref{e11.1}.  The same argument can be applied to $ \zeta^{1,il}_{k,k+1}$, which yields \eref{e11.1}   for the case  $i=j\neq l$. We let the patient reader check that the same inequality holds true in the case $i\neq j, i =l$, resorting to a stochastic Fubini lemma. This completes the proof.
\end{proof}
 
 \begin{remark}\label{remark3.9}
 The estimate of $\zeta^{2,l}_{k,k+1}$ obtained in the proof of Lemma \ref{lem11.2} implies that
 \begin{eqnarray*}
\Big\|
\sum_{t_{k}=s}^{t-} \int_{t_{k}}^{t_{k+1}} (u-t_{k})^{2H} dB_{u} 
\Big\|_{p} 
&\leq & K n^{\frac12 -3H} (t-s)^{\frac12}.
\end{eqnarray*}
This inequality will be used below in order to prove Lemma  \ref{cor3.13}.
 \end{remark}

 We can now 
 deliver a path-wise bound on
  weighted sums of the process $\zeta$.
 
 \begin{lemma}\label{lem3.11}
Consider the increment $\zeta$ defined by \eqref{eq:def-zeta-and-g}. Let $f$ be a process on $[0,T]$ such that,   for any $  \ga< H$, there exists a random variable $G$ such that $  \|   f  \|_{\ga} \leq G  $. 
  Then for any $\kappa>0$ we have the estimate 
\begin{eqnarray}\label{e7.2}
\Big| \sum_{t_{k}=s}^{t-} f_{t_{k}}  \zeta_{t_{k}t_{k+1}}\Big| &\leq& G n^{1-4\ga+2\kappa }(t-s)^{1-\ga} ,
\quad\text{for all}\quad
(s,t)\in \cs_{2}(\ll 0,T \rr),
\end{eqnarray}
where $G$ is an integrable random variable independent of $n$.
\end{lemma}

\begin{proof}
Consider $(s,t)\in \cs_{2}(\ll 0,T \rr)$, and observe that the following decomposition holds true:
\begin{eqnarray}\label{e3.11i}
\Big| \sum_{t_{k}=s}^{t-} f_{t_{k}}  \zeta_{t_{k}t_{k+1}}\Big| &\leq& 
\Big| \sum_{t_{k}=s}^{t-} \delta f_{s t_{k}}  \zeta_{t_{k}t_{k+1}}\Big| 
+ |f_{s}|\cdot  \left| \delta g_{st}\right|,
\end{eqnarray}
where   increment $g$ has been defined by \eqref{eq:def-zeta-and-g}.
In addition, thanks to Lemma \ref{lem11.2}, we obtain:
\begin{eqnarray*}
\| \delta g_{st} \|_{p}   \leq Kn^{\frac12 - 3H} (t-s)^{\frac12} \leq K n^{1-4\ga+\kappa} (t-s)^{1-\ga+\kappa},
\end{eqnarray*} 
where the last inequality is due to the fact that $\frac{T}{n}\le t-s$. Here $\ga<H$ and $\kappa>0$. Applying Lemma \ref{lem4.2} to $g$ we thus get:
\begin{eqnarray}\label{e3.10}
| \delta g_{st}|  &\leq& G_{1} n^{1-4\ga+2\kappa} (t-s)^{1-\ga},
\end{eqnarray}
where   $G_{1}$ is a  random variable independent of $n$.  
Now observe that a direct application of Proposition \ref{prop3.4}  (notice that $\zeta_{t_{k}t_{k+1}}=\delta g_{t_{k}t_{k+1}}$ in the relation below)   enables to write:
\begin{equation}\label{e3.10.1}
\Big| \sum_{t_{k}=s}^{t-} \delta  f_{st_{k}}  \zeta_{t_{k}t_{k+1}}\Big|
\le
G_{2} n^{1-4\ga+2\kappa } (t-s) ,
\end{equation}
where $G_{2}$ is another integrable random variable independent of $n$.
Plugging \eref{e3.10} and  \eref{e3.10.1} into the right-hand side of  \eref{e3.11i}, we obtain the desired estimate \eqref{e7.2}.
\end{proof}
 
 We now consider the case of a weighted sum involving a Wiener integral with respect to~$B$.
 
 \begin{lemma}\label{cor3.13}
 Let   $f$ be as in Lemma \ref{lem3.11} and $\ga<H$. Then   the  following estimate holds true:
 \begin{eqnarray*} 
\Big\| \sum_{t_{k}=s}^{t-} \delta f_{st_{k}}\otimes \int_{t_{k}}^{t_{k+1} } (u-t_{k})^{2H} d B_{u} \Big\|_{p} &\leq& Kn^{1 - 4\ga-2\kappa} (t-s) ,
\quad\quad 
(s,t) \in \cs_{2}(\ll 0,T \rr)
.
\end{eqnarray*}
\end{lemma}
\begin{proof}
The corollary is a direct application of Proposition \ref{prop3.4} and is similar to the proof of Lemma \ref{lem3.11}. The details are omitted.
\end{proof}
 
 We turn to  controls of weighted sums in cases involving rougher processes. They provide  our first instances where we apply rough path methods for weighted sums, as announced in the introduction.
  In the following,  $\cv$ and $\cv'$ stands for some finite dimensional vector spaces.  
 
 \begin{prop} \label{prop3.6}
Let $f$, $g$ be two processes defined on $[0,T]$ with values in $\cv$ and $\cl(\RR^{m}, \cv)$, respectively, and $h$ be an increment defined on $\cs_{2}([0,T])$ with values in $\cv'$.   Assume that there is a constant $K$ and an exponent  $\ga>0$ such that the following conditions are met for $(s,t)\in\cs_{2}([0,T])$ and all $p\ge 1$: 
\begin{equation}\label{e3.8ii}
\|f_{t}\|_{p} + \|g_{t}\|_{p} \leq K,
\quad
\|\delta f_{st}  - g_{s} \delta B_{st}\|_{p} \leq K(t-s)^{2\ga},
\quad
\|\delta g_{st} \|_{p} \leq K (t-s)^{\ga}.
\end{equation}
We also suppose that $h$ satisfies
\begin{eqnarray}\label{e3.8i}
\| h_{st}  \|_{p} \leq K (t-s)^{\al}, 
\quad \quad \text{and} \quad\quad
\Big\| \sum_{t_{k} = s}^{t-} 
\delta B_{st_{k}} \otimes h_{t_{k}t_{k+1}}   \Big\|_{p} \leq K (t-s)^{\ga+\al},
\end{eqnarray}
for $(s,t)\in\cs_{2}(\ll 0,T\rr)$ and any $p\ge 1$, where $\al$ is such that $\al+2\ga>1$. Then we have  
\begin{eqnarray}\label{e3.9}
\Big\| \sum_{t_{k}=s}^{t-} \delta f_{st_{k}} \otimes h_{t_{k}t_{k+1}} \Big\|_{p} &\leq& K (t-s)^{\ga+\al},
\end{eqnarray}
and
\begin{eqnarray}\label{e3.9i}
\Big\| \sum_{t_{k} = s+ }^{t-} \sum_{t_{k'} = s }^{t_{k-1}} f_{t_{k'}} \otimes h_{t_{k'} t_{k'+1}} \otimes \delta B_{t_{k}t_{k+1}} \Big \|_{p}
&\leq& K   (t-s)^{\ga+\al}, 
\end{eqnarray}
which are valid for $(s,t)\in\cs_{2}(\ll 0,T\rr)$ and all $p\ge 1$. Furthermore, set 
\begin{eqnarray}\label{eq3.35}
R_{st} &=&
  \sum_{t_{k} = s }^{t-} (\delta f_{st_{k}}- g_{s}\delta B_{s t_{k}}) \otimes h_{t_{k}t_{k+1}},
\end{eqnarray}
 then we have the estimate
\begin{eqnarray}\label{eq3.36}
\|R_{st}\|_{p } &\leq & K   (t-s)^{2\ga+\frac12}.
\end{eqnarray}
\end{prop}

 \begin{proof}
We start by proving inequality \eqref{e3.9}. To this aim, set $A_{st}=\sum_{t_{k}=s}^{t-} \delta f_{st_{k}} \otimes h_{t_{k}t_{k+1}}$ for  $(s,t)\in\cs_{2}(\ll 0,T\rr)$, and consider $p\ge 1$. We decompose the increment $A$ into $A=M+R$, where $R$ is the increment defined in \eref{eq3.35} and $M$ is defined by:
\begin{equation*}
M_{st}
=
g_{s}\sum_{t_{k} = s }^{t-} \delta B_{s t_{k}}  \otimes h_{t_{k}t_{k+1}} .
\end{equation*}
Then it is immediate from \eqref{e3.8ii} and \eqref{e3.8i} that
\begin{equation}\label{e9.3i}
\|M_{st}\|_{p} \le K (t-s)^{\ga+\al}.
\end{equation}
In order to bound the increment $R$, we note that $R_{t_{k}t_{k+1}} = 0$. Let us now calculate $\delta R$: for $(s,u,t)\in\cs_{3}(\ll 0,T\rr)$, it is readily checked that:
\begin{eqnarray*}
\delta R_{sut} &=& \delta f_{su} \otimes h_{ut} 
-        \Big(      g_{s} \delta B_{su} \otimes h_{ut} -  \delta g_{su} \sum_{t_{k} = u}^{t-} \delta B_{ut_{k}} \otimes h_{t_{k}t_{k+1}}                                                   \Big)        
\\
&=&        
(\delta f_{su}   - g_{s} \delta B_{su} ) \otimes h_{ut}     +\delta g_{su} \sum_{t_{k} = u}^{t-} \delta B_{ut_{k}} \otimes h_{t_{k}t_{k+1}}     .                                    
\end{eqnarray*}
Therefore invoking \eqref{e3.8ii} and \eqref{e3.8i} again, we get:
\begin{eqnarray*}
\| \delta R_{sut} \|_{p} &\leq& K(u-s)^{2\ga}   (t-u)^{\al} +K (u-s)^{\ga}   (t-u)^{\ga+\al}
\le (t-s)^{\mu},
\end{eqnarray*}
where $\mu=2\ga+\al$, and where by assumption we have $\mu>1$.
Hence, owing to the discrete sewing   Lemma \ref{lem2.4} (applied to the Banach space $\cx=L^{p}(\Omega)$) we have:
\begin{eqnarray}\label{e9.3}
\|R\|_{p,\mu}  \leq  K \|\delta R\|_{p,\mu} \leq K  ,
\end{eqnarray}
where $\|\cdot\|_{p,\mu}$ designates the $\mu$-H\"older norm for $L^{p}(\Omega)$-valued functions.
Putting together our estimates \eqref{e9.3i} and \eqref{e9.3} on $M$ and $R$, inequality \eqref{e3.9} is proved.

In the following we derive our second claim \eqref{e3.9i}.  The method is similar to the proof of~\eqref{e3.9}, so that it will only be sketched for sake of conciseness. We resort to the following decomposition:
\begin{equation*}
\wt{A}_{st}
\equiv
\sum_{t_{k} = s+ }^{t-} \sum_{t_{k'} = s }^{t_{k-1}} 
f_{t_{k'}}\otimes h_{t_{k'} t_{k'+1}} \otimes \delta B_{t_{k}t_{k+1}}
=
\wt{M}_{st} + \wt{R}_{st}\,, 
\end{equation*}
where
\begin{equation*}
\wt{M}_{st}
=
\sum_{t_{k} = s+ }^{t-} \sum_{t_{k'} = s }^{t_{k-1}}   
f_{s} \otimes  h_{t_{k'} t_{k'+1}} \otimes \delta B_{t_{k}t_{k+1}},
\quad\text{and}\quad
\wt{R}_{st}
=
\sum_{t_{k} = s+ }^{t-} \sum_{t_{k'} = s }^{t_{k-1}}   
\delta f_{st_{k'}} \otimes  h_{t_{k'} t_{k'+1}} \otimes \delta B_{t_{k}t_{k+1}}.
\end{equation*}
In order to bound $\wt{M}_{st}$ we change the order of summation, which allows to exhibits some terms of the form $\delta B_{st_{k}}$. Then we let the reader check that inequality  \eref{e3.8i} can be applied directly. As far as $\wt{R}$ is concerned, notice again that $\wt{R}_{t_{k}t_{k+1}} = 0$.
It is then readily seen, as in the previous step, that our estimate boils down to a bound on $\delta\wt{R}$. Furthermore, $ \delta \wt{R}$ can be computed as follows:
\begin{eqnarray*}
\delta \wt{R}_{sut} &=& 
\Big(
\sum_{t_{k'} = s}^{u-} \delta f_{st_{k'}}\otimes h_{t_{k'}t_{k'+1}}  
\Big) \otimes 
 \delta B_{ut} + \delta f_{su}\otimes  \sum_{t_{k} = u+ }^{t-} \sum_{t_{k'} = u }^{t_{k-1}}      h_{t_{k'} t_{k'+1}} \otimes \delta B_{t_{k}t_{k+1}}.
\end{eqnarray*}
We can thus resort to \eqref{e3.8ii}, \eref{e3.9} and \eref{e3.8i} in order to get:
\begin{eqnarray*}
\| \delta \wt{R}_{sut} \|_{p} &\leq& K   (t-s)^{2\ga+\al }.
\end{eqnarray*}
The proof is now completed as for relation \eqref{e3.9}.
\end{proof}
  \begin{remark}\label{remark4.7}
   In Proposition \ref{prop3.6} the weighted sum $\sum_{t_{k}=s}^{t-} \delta f_{st_{k}} \otimes h_{t_{k}t_{k+1}}$ is viewed as a $L_{p}$-valued ``discrete'' rough integral. Similarly as in Remark \ref{remark3.8}, a $\RR$-valued ``discrete'' rough path can also be considered. 
  \end{remark}
Proposition \ref{prop3.6} can be applied to the sum $F$ of L\'evy area increments of $B$. This is the contents of the corollary below.

\begin{cor}\label{lem9.1}
Let $\frac14<\ga<H$, $f$ and $g$ be as in Proposition \ref{prop3.6}.  Let $F$ be the process defined by  \eref{e4.1}, considered as a path taking values in $\cv'=\RR^{d\times d}$.  Then   the following estimates hold true for $(s,t) \in \cs_{2}(\ll 0,T \rr)$:
\begin{eqnarray*}
\Big\| \sum_{t_{k}=s}^{t-} f_{t_{k}} \otimes   \delta F_{t_{k}t_{k+1}}\Big\|_{p} &\leq& Kn^{\frac12 - 2H}(t-s)^{\frac12},
\end{eqnarray*}
and
 \begin{eqnarray*}
\Big\| \sum_{t_{k} = s+ }^{t-} \sum_{t_{k'} = s }^{t_{k-1}} f_{t_{k'}}   \otimes
\delta F_{t_{k'} t_{k'+1}} \otimes \delta B_{t_{k}t_{k+1}}  \Big\|_{p}
&\leq& K n^{\frac12 - 2H} (t-s)^{H+\frac12}. 
\end{eqnarray*}
Furthermore, set $R_{st}= 
  \sum_{t_{k} = s }^{t-} (\delta f_{st_{k}}- g_{s}\delta B_{s t_{k}}) \otimes F_{t_{k}t_{k+1}}$, then we have
\begin{eqnarray}\label{e3.35i}
\|R_{st}\|_{p } &\leq & K n^{\frac12 - 2H} (t-s)^{2\ga+\frac12}.
\end{eqnarray}

\end{cor}

\begin{proof}   Take $h = n^{2H-\frac12} F $ and $\al = \frac12$. It follows from Lemma \ref{lem4.1} and 
      Lemma \ref{lem:sum-delta-B-F} that $h$ satisfies the conditions in Proposition \ref{prop3.6}. 
      In addition, if $\ga>\frac14$, the condition $2\ga+\frac12 >1$ is trivially satisfied. 
      The corollary then follows  immediately from  Proposition \ref{prop3.6} and taking into account the decomposition $f_{t_{k}} = f_{s} +\delta f_{st_{k}}$. The estimate of $R_{st}$ follows directly from relation \eref{eq3.36}.
\end{proof}

 \subsection{Limiting theorem results via rough path approach}

Take two uniform partitions on $[0,T]$: $t_{k} = \frac{T}{n}k$ and $u_{l}=\frac{T}{\nu} l$, $n,\nu \in \NN$ for $k=0,\dots, n$ and $l=0,\dots, \nu$. Let $k_{l}$ be such that $t_{k_{l}+1} > u_{l} \geq t_{k_{l}} $.  In the following, we set for each $t\in [0,T]$:
\begin{eqnarray}\label{eq9.7}
D_{l}= \{t_{k}: u_{l+1}> t_{k}\geq u_{l},\, t\geq  t_{k}   \} \quad \text{and} \quad \tilde{D}_{l}=\{ t_{k}:  t_{k_{l+1}}> t_{k}\geq t_{k_{l}},\, t\geq t_{k} \} .
\end{eqnarray}
Our main result in this section is a central limit theorem for sums weighted by a controlled process $f$.

\begin{theorem}\label{prop9.2}
Let the assumptions in Proposition \ref{prop3.6} prevail, and suppose that
\begin{eqnarray}\label{eq9.6}
(h, B) \xrightarrow{\rm{f.d.d.}} (W,B), \quad n\to \infty,
\end{eqnarray}
  where `` $  \xrightarrow{\rm{f.d.d.}}  $'' stands for convergence of finite dimensional distributions and $W$ is a    Brownian motion independent of $B$.  
Set 
\begin{eqnarray}\label{eq9.8}
\zeta_{l}^{n} &=& \sum_{ t_{k} \in \tilde{D}_{l} }  \delta B_{t_{k_{l}} t_{k}} \otimes h_{t_{k}t_{k+1}},
\end{eqnarray}
and suppose that 
\begin{eqnarray}\label{eq9.9}
 \Big\| \sum_{l= \frac{\nu r}{T} }^{  \frac{\nu r'}{T} -1 }   \zeta^{n}_{l} \Big\|_{p} 
  \leq   K \nu^{-\ka}  (r'-r)^{\al+\ga- \ka}
\end{eqnarray}
for  $r,r' \in \{u_{1},\dots,u_{\nu}\}$ and for an arbitrary $\kappa>0$.
Set 
$
\Theta^{n}_{t} =   \sum_{k=0}^{\left\lfloor \frac{nt}{T} \right\rfloor} f_{t_{k}} \otimes   h_{t_{k}t_{k+1}}
$ and $
\Theta_{t} = \int_{0}^{t} f_{s}\otimes dW_{s} 
$.
 Then the following relation holds true:
 \begin{eqnarray*}
\left( \Theta^{n}, B \right)  &\xrightarrow{\rm{f.d.d.}}& \left( \Theta , B \right)\quad\quad \text{as} \quad n\to \infty.
\end{eqnarray*}
\end{theorem}
\begin{remark} There are several possible generalizations of the statement of Theorem \ref{prop9.2}. If one has the convergence of $(h,B)$ in $L_{p}$ instead of the weak convergence~\eref{eq9.6}, then by a similar proof  one can show that   $\Theta^{n} $ converges to $ \Theta$ in $L^{p}$. On the other hand, similarly to what has been mentioned in Remark \ref{remark4.7}, if
conditions \eref{e3.8ii}, \eref{e3.8i} and \eref{eq9.9} are replaced by 
 the corresponding  almost sure upper-bound conditions, then one can show that   Theorem~\ref{prop9.2} still holds true. 
\end{remark}
\begin{remark}\label{remark4.12} 
In the case $\ga>\frac12$ and $\al\geq\frac12$, conditions  \eref{e3.8ii}, \eref{e3.8i} and \eref{eq9.9} are reduced to $\|\delta f_{st}\|_{p}\leq K(t-s)^{\ga}$ and $\| \delta h_{st} \|_{p}\leq K(t-s)^{\al}$.   Theorem \ref{prop9.2} then recovers the central and non-central limit theorem results in \cite{CNP} and \cite{HLN}.
\end{remark}
\begin{remark}
According to our proof of Theorem \ref{prop9.2}, in   general   the  limit of the ``Riemann sum'' $\Theta^{n}$ is independent of  the choices of the representative points. In the situation of  Remark \ref{remark4.12} this fact can be proved directly from the expression of $\Theta^{n}$.   
\end{remark}

\begin{proof}[Proof of Theorem \ref{prop9.2}]  By   definition of the f.d.d. convergence, 
it suffices to show the following weak convergence for $r_{1},\dots, r_{l} \in [0,T]$:
\begin{eqnarray*}
(\Theta^{n}_{r_{1}},\dots, \Theta^{n}_{r_{l}}, B_{r_{1}},\dots, B_{r_{l}} )  \xrightarrow{ \ (d) \ } (\Theta_{r_{1}},\dots, \Theta_{r_{l}}, B_{r_{1}},\dots, B_{r_{l}} )   \quad\quad n\rightarrow \infty.
\end{eqnarray*}

\noindent{\it Step 1: A coarse graining argument.} 
Consider an extra parameter $\nu\!<\!<\!n$   and take $\{ u_{0},\dots, u_{\nu} \}$ to be the uniform $\nu$-step partition of $[0,T]$.  We make the following decomposition:
 \begin{eqnarray*}
\Theta^{n}_{t} &=&   \tilde{\Theta}^{n}_{t}  + \hat{\Theta}^{n}_{t}  ,
\end{eqnarray*}
where 
\begin{eqnarray*}
 \tilde{\Theta}^{n}_{t} =   \sum_{l=0}^{\nu-1} \sum_{t_{k}\in D_{l}  }  \delta f_{u_{l}t_{k}}    \otimes h_{t_{k}t_{k+1}}\,, 
\quad\quad\quad
    \hat{\Theta}^{n}_{t}=  \sum_{l=0}^{\nu-1} 
    \sum_{t_{k}\in D_{l}  }  
   f_{u_{l}}  \otimes h_{t_{k}t_{k+1}}\,.
\end{eqnarray*}

Let us first handle the convergence of $\hat{\Theta}^{n}$:
by letting $n \rightarrow \infty$ and taking into account the convergence \eref{eq9.6}, and then letting $\nu\rightarrow \infty$, we easily obtain the weak convergence:
\begin{eqnarray*}
(\hat{\Theta}^{n}_{r_{1}},\dots, \hat{\Theta}^{n}_{r_{l}}, B_{r_{1}},\dots, B_{r_{l}} ) \xrightarrow{ \ (d) \ } (\Theta_{r_{1}},\dots, \Theta_{r_{l}}, B_{r_{1}},\dots, B_{r_{l}} )  .
\end{eqnarray*}
Therefore, 
in order to prove our claim, we are reduced  to show that for $t\in [0,T]$
\begin{eqnarray}\label{e9.4iii}
\lim_{\nu\rightarrow \infty} \limsup_{n\rightarrow \infty}  
\mE
\big(\big|
 \tilde{\Theta}^{n}_{t}  
 \big|^{2} \big) 
  =0.
\end{eqnarray}

\noindent{\it Step 2: First-order approximation of $f$.} 
Let $k_{l}$ be such that $t_{k_{l}+1} > u_{l} \geq t_{k_{l}} $. We compare the two sets: $D_{l} $ defined by 
\eref{eq9.7}
 and $\tilde{D}_{l}=\{ t_{k}:  t_{k_{l+1}}> t_{k}\geq t_{k_{l}},\, t\geq t_{k} \}$. It is easy to see that $D_{l}\,\Delta \, \tilde{D}_{l} \subset \{ t_{k_{l}}, t_{k_{l+1} } \} $.
It is also readily checked from   conditions \eref{e3.8ii} and \eref{e3.8i} that $ \| \delta f_{u_{l}t_{k}} \|_{L_{4}} \leq K \nu^{-\ga} $ and  $\|h_{t_{k}t_{k+1}}\|_{L_{4}} \leq K n^{-\al}$. A simple use of Cauchy-Schwarz inequality thus yields:
\begin{eqnarray}\label{e9.4ii}
\left\|   \delta f_{u_{l}t_{k}} \otimes  h_{t_{k}t_{k+1}}  \right\|_{2} &\leq &K \nu^{-\ga} n^{-\al},
\end{eqnarray}
for any $k=0,1,\dots,n-1 $. In particular, \eref{e9.4ii} holds for $ t_{k} \in D_{l}\,\Delta \, \tilde{D}_{l}  $. Therefore, in order 
to show \eref{e9.4iii}, it is sufficient to show  that
\begin{eqnarray}\label{e9.4i}
\lim_{\nu\rightarrow \infty} \limsup_{n\rightarrow \infty}  
\mE
\Big(\Big|
 \sum_{l=0}^{\nu-1}  R_{t_{k_{l}}t_{k_{l+1}}} 
 \Big|^{2} \Big) 
  =0,
\end{eqnarray}
where
\begin{eqnarray*}
R_{t_{k_{l}}t_{k_{l+1}}} &= &  \sum_{ t_{k} \in \tilde{D}_{l} }  \delta f_{u_{ {l}}t_{k}}  \otimes  h_{t_{k}t_{k+1}} \,.
\end{eqnarray*}
Now
in order to get relation \eref{e9.4i}, consider the following decomposition for $R $: 
\begin{eqnarray}\label{eq9.10}
R_{t_{k_{l}}t_{k_{l+1}}} &= &   \bar{R}_{t_{k_{l}}t_{k_{l+1}}} +   \tilde{R}_{t_{k_{l}}t_{k_{l+1}}} +\check{R}_{t_{k_{l}}t_{k_{l+1}}}+\hat{R}_{t_{k_{l}}t_{k_{l+1}}},
\end{eqnarray}
where the increments $ \bar{R}_{t_{k_{l}}t_{k_{l+1}}} $, $   \tilde{R}_{t_{k_{l}}t_{k_{l+1}}} $, $\check{R}_{t_{k_{l}}t_{k_{l+1}}}$ and $\hat{R}_{t_{k_{l}}t_{k_{l+1}}}$ are defined by
\begin{align*} 
 &
  \tilde{R}_{t_{k_{l}}t_{k_{l+1} } } = \sum_{ t_{k} \in \tilde{D}_{l} }    \delta f_{u_{ {l}}t_{k_{l}}}    \otimes h_{t_{k}t_{k+1}}\,,
 \quad\quad
 \bar{R}_{t_{k_{l}}t_{k_{l+1} }} = \sum_{ t_{k} \in \tilde{D}_{l}  }  (\delta f_{t_{k_{l}} t_{k}}- g_{t_{k_{l}} }\otimes\delta B_{ t_{k_{l}} t_{k}}) \otimes h_{t_{k}t_{k+1}}
 \,,
 \end{align*}
and
\begin{equation}\label{eq9.12}
\check{R}_{t_{k_{l}}t_{k_{l+1}}}
=    \delta g_{u_{l} t_{k_{l}} }\otimes  \sum_{ t_{k} \in \tilde{D}_{l}  }  \delta B_{t_{k_{l}} t_{k}} \otimes  h_{t_{k}t_{k+1}}\,, 
\quad\quad
 \hat{R}_{t_{k_{l}}t_{k_{l+1}}}
=    g_{u_{ {l}}}\otimes  \sum_{ t_{k} \in \tilde{D}_{l}  }  \delta B_{t_{k_{l}} t_{k}} \otimes  h_{t_{k}t_{k+1}} 
 \,. 
\end{equation}
It follows from  \eref{eq3.36} that
\begin{eqnarray}\label{e9.7i}
\Big \| \bar{R}_{t_{k_{l}}t_{k_{l+1} } }\Big\|_{2}
 &\leq& K    \nu^{-2\ga-\al}
 .
\end{eqnarray}
On the other hand, it follows from \eref{e3.8ii} and   \eref{e3.8i} that
\begin{eqnarray}\label{e9.8}
\Big\|
\tilde{R}_{t_{k_{l}} t_{k_{l+1}}} 
\Big\|_{2 }
= 
\Big\|
  \delta f_{u_{ {l}}t_{k_{l}}}    \otimes h_{t_{k_{l}}\,, \,t_{k_{l+1}}\wedge t+}
\Big\|_{2 }
 \leq 
  K n^{-\ga } \nu^{-\al}  ,
\end{eqnarray}
where recall that we   denote $t+ = t+\frac{T}{n}$ and $t_{k_{l+1}}\wedge t+ = \min ( t_{k_{l+1}},  t+ )$.
Similarly, applying  \eref{e3.8ii} and   \eref{e3.8i}  we obtain:
\begin{eqnarray}\label{eq9.15}
\Big\|
\check{R}_{t_{k_{l}} t_{k_{l+1}}} 
\Big\|_{2 }&\leq& K n^{-\ga } \nu^{-\ga-\al}  . 
\end{eqnarray}
 It follows immediately from \eref{e9.7i}, \eref{e9.8} and \eref{eq9.15} that
\begin{eqnarray}\label{e9.4}
\lim_{\nu\rightarrow \infty} \limsup_{n\rightarrow \infty}  
\mE\Big(\Big| \sum_{l=0}^{\nu-1} \Big( \bar{R}_{t_{k_{l}}t_{k_{l+1} } } +\tilde{R}_{t_{k_{l}}t_{k_{l+1} } }+\check{R}_{t_{k_{l}} t_{k_{l+1}}}\Big) \Big|^{2}\Big) &=&0.
\end{eqnarray}
In view of \eref{e9.4} and taking into account the decomposition \eref{eq9.10}, in order to show \eref{e9.4i} it suffices to show   that
\begin{eqnarray}\label{e9.10}
\lim_{\nu\rightarrow \infty} \limsup_{n\rightarrow \infty}  \mE\Big(\Big|
\sum_{l=0}^{\nu-1} \hat{R}_{t_{k_{l}}t_{k_{l+1}}}
\Big|^{2}\Big) &=& 0.
\end{eqnarray}

\noindent{\it Step 3: Study of $\hat{R}$.}
 We will see that $ \sum_{l=0}^{\nu-1} \hat{R}_{t_{k_{l}}t_{k_{l+1}}}$ can be considered as a  discrete ``Young'' integral in $L_{2}$ in the sense  of Remark \ref{remark3.8}  (see also Proposition \ref{prop3.4}), which then leads to the convergence~\eref{e9.10}. 
 Namely, starting from the expression \eref{eq9.12} of $\hat{R}$,
 let us first consider the ``weight-free'' sum
 \begin{eqnarray*}
\hat{\zeta}^{n}_{r} : = \sum_{l=0}^{\frac{\nu r}{T}-1} \zeta_{l}^{n}
\end{eqnarray*}
 where     $r \in \{u_{1},\dots,u_{\nu}\}$ and recall that $\zeta_{l}^{n}$ is defined in \eref{eq9.8}.
 Observe that    \eref{eq9.12}   can be recast as 
 \begin{eqnarray}\label{e9.18ii}
\hat{R}_{t_{k_{l}}t_{k_{l+1}}} = g_{u_{l}}\otimes \zeta^{n}_{l} = g_{u_{l}}   \delta \hat{\zeta}^{n}_{u_{l}u_{l+1}}\,,
\end{eqnarray}
for all $l=0,\dots, \nu-1$. 
According to \eref{e9.18ii}, we have
\begin{eqnarray*}
\sum_{l=0}^{\nu-1} \hat{R}_{t_{k_{l}}t_{k_{l+1}}}  = \sum_{l=0}^{\nu-1} g_{{u_{l}}} \otimes \delta \hat{\zeta}^{n}_{u_{l}u_{l+1}} 
= \sum_{l=0}^{\nu-1} \delta g_{0{u_{l}}} \otimes \delta \hat{\zeta}^{n}_{u_{l}u_{l+1}} +\sum_{l=0}^{\nu-1} g_{0} \otimes \delta \hat{\zeta}^{n}_{u_{l}u_{l+1}} . 
\end{eqnarray*}
Then our assumption \eref{eq9.9} and the bound \eref{e3.8ii} ensures that we are in a position to apply Proposition   \ref{prop3.4}. This immediately yields our claim \eref{e9.10}, which concludes the proof. 
\end{proof}

\section{Euler scheme process as a rough path}\label{section4}
In this section, we  consider a continuous time interpolation of the   Euler scheme $y^{n}$ given by~\eqref{e4}. Namely, we introduce a  sequence of processes $y^{n}$ indexed by $[0,T]$ in the following way: for $t\in[t_{k},t_{k+1})$ we set
\begin{eqnarray}\label{e1.2}
y^{n}_{t } &=& y^{n}_{t_{k}} + b(y^{n}_{t_{k}}) (t-t_{k}) +{V}(y^{n}_{t_{k}}) \delta B_{t_{k}t } +\frac12 \sum_{j=1}^{m} \partial {V}_{j}{V}_{j} (y^{n}_{t_{k}}) (t-t_{k})^{2H},
\end{eqnarray}
where we recall that   $\pi : 0=t_{0}< t_{1} <\cdots<t_{n}=T $ designates the uniform  partition of the interval $[0,T]$.  The remainder of the section is devoted to get some uniform bounds on $y^{n}$, and then to prove that the couple $(y^{n},B)$ can be lifted as a rough path. Throughout the section, we assume that $b \in C^{2}_{b}$ and $V\in C^{4}_{b}$.

\subsection{H\"older-type bounds for the   Euler scheme}

Our main results on H\"older regularity of the sequence $y^{n}$ is summarized in the following proposition.

\begin{prop}\label{lem5.1}
 Let $y^{n}$ be the process defined by the Euler scheme \eref{e1.2}. Take $\frac13<\ga<H$.  Then for all $(s,t) \in \cs_{2}(\ll0,T\rr)$, the following relations are satisfied:
 \begin{eqnarray}\label{eq:bnd-delta-yn}
|\delta y^{n}_{st}|  \leq  G|t-s|^{\ga} , 
\quad \quad \quad
|\delta y^{n}_{st} -V(y^{n}_{s}) \delta B_{st} | \leq  G (t-s)^{2\ga},
\end{eqnarray}
where $G$ stands for an integrable random variable which is independent of the parameter $n$.
\end{prop}

\begin{proof}  

We divide this proof into several steps. For the sake of conciseness, we omit the drift~$b$ in the proof, so that we analyze a scheme defined successively by 
\begin{eqnarray}\label{eq4.4}
y^{n}_{t_{k+1}} &=& y^{n}_{t_{k}} +V(y^{n}_{t_{k}}) \delta B_{t_{k}t_{k+1}} +\frac12 \sum_{j=1}^{m} \partial V_{j}V_{j} (y^{n}_{t_{k}}) h^{2H}.
\end{eqnarray}

\noindent
{\it Step 1:  Definition   of the remainder.}
We  first define some increments of interest for the analysis of the scheme given by~\eqref{e1.2}. Let us start with a 2nd-order increment $q$ defined  by:
 \begin{eqnarray*}
 q_{st} &=&  \sum_{i,j=1}^{m} ( \partial {V}_{i}{V}_{j})(y^{n}_{s}) \delta F^{ij}_{st}  \,, \quad\quad  \cs_{2}(\ll0,T\rr) ,
\end{eqnarray*}
where recall that $F^{ij}$ is defined in \eref{e4.1}. 
 Next our remainder term for \eqref{e1.2} is given   by:
\begin{eqnarray}\label{e.10}
 R_{st} &=&
\delta y^{n}_{s  t} -   {V}(y^{n}_{s}) \delta B_{s t} -\sum_{i,j=1}^{m} ( \partial {V}_{i}{V}_{j})(y^{n}_{s}) \BB_{st}^{ij} + q_{st}, \quad \cs_{2}(\ll0,T\rr).
\end{eqnarray}
Since $R$ is expected to be regular in $|t-s|$ and $R_{t_{k}t_{k+1}} =0$ by the very definition of $y^{n}$, we will analyze $R$ through an application of the discrete sewing Lemma \ref{lem2.4}. To this aim, we calculate $\delta R$, which is easily decomposed as follows:  
\begin{eqnarray}\label{e4.5}
\delta R_{sut}   &= & A_{1}+A_{2}+A_{3}+A_{4}\,,
\end{eqnarray}
where for $(s, u, t) \in \cs_{3}(\ll 0, T\rr)$, the quantities $A_{1}, A_{2}$ are given by:
\begin{equation*}
A_{1} = \delta {V} (y^{n})_{su} \delta B_{ut},
\qquad
A_{2} =  \sum_{i,j=1}^{m}\delta ( \partial {V}_{i}{V}_{j})(y^{n})_{su} \BB_{ut}^{ij},
\end{equation*}
and where $A_{3},A_{4}$ are defined by:
\begin{equation*}
A_{3} = - \sum_{i,j=1}^{m}  ( \partial {V}_{i}{V}_{j})(y^{n}_{s})  \delta B_{su}^{j} \delta B_{ut}^{i},
\quad\text{and}\quad
A_{4} = -  \sum_{i,j=1}^{m} \delta ( \partial {V}_{i}{V}_{j})(y^{n}_{\cdot})_{su} \delta F^{ij}_{ut}\,.
\end{equation*}
Observe that in order to compute $A_{e}$, $e=1,2,3,4$ we have used the fact that $\delta \delta B = 0$, $\delta \delta F = 0$, and $\delta \BB = \delta B\otimes \delta B$.  
Moreover, note that owing to an elementary Taylor type expansion we have:
\begin{eqnarray*}
\delta {V} (y^{n})_{su} \delta B_{ut}&=&
\sum_{i=1}^{m}[\partial V_{i}(y^{n})]_{su} \delta y^{n}_{su} \delta B^{i}_{ut},
\end{eqnarray*}
where we denote $[\partial V_{i}(y^{n})]_{su} =  \int_{0}^{1} \partial {V}_{i}( y^{n}_{s} + \lambda  \delta y^{n}_{su} )  d \lambda $.
So invoking relation \eref{e.10}, we can further decompose $A_{1}$ as follows:
\begin{eqnarray*}
A_{1} = \delta {V} (y^{n})_{su} \delta B_{ut}&=&
  A_{11}+A_{12}+A_{13}+A_{14},
\end{eqnarray*}
where $A_{11}$ and $A_{12}$ are defined by
\begin{equation*}
A_{11} = \sum_{i=1}^{m}[\partial V_{i}(y^{n})]_{su}  {V}(y^{n}_{s}) \delta B_{s u}
 \delta B^{i}_{ut},
\qquad
A_{12} 
= 
\sum_{i=1}^{m}[\partial V_{i}(y^{n})]_{su} \sum_{i',j'=1}^{m}( \partial {V}_{i'}{V}_{j'})(y^{n}_{s}) \BB_{su}^{i'j'}
 \delta B^{i}_{ut}
\end{equation*}
and where
\begin{equation*}
A_{13} = -\sum_{i=1}^{m}[\partial V_{i}(y^{n})]_{su}  q_{su}
 \delta B^{i}_{ut},
\quad\text{and}\quad
A_{14} = \sum_{i=1}^{m}[\partial V_{i}(y^{n})]_{su}  R_{s u }
 \delta B^{i}_{ut}\,.
\end{equation*}
We will bound those terms separately.

\noindent{\it Step 2: Upper-bound for $y$ and $R$ on small intervals.}
Consider the following deterministic constant: %\hb{Specify the meaning of $\|{V}\|_{\infty}$ for a vector valued field $V$}: 
\begin{eqnarray*}
K_{{V}} &=& (1+\|{V}\|_{\infty} +\| \partial{V}\|_{\infty}+\|\partial^{2}{V}\|_{\infty} +K_{3\ga})^{3},
\end{eqnarray*}
where $K_{3\ga}$ is the constant appearing in Lemma \ref{lem2.4}, and $\|V\|_{\infty}$ for a vector-valued field $V$ denotes the supnorm of the function $|V|$. We also introduce the following random variable: 
\begin{equation}\label{eq:def-G}
G  = G_{\ga} + L_{\ga} +1,
\end{equation}
where $G_{\ga}$ is defined in \eref{e4.21} and $L_{\ga}$ is introduced in Proposition \ref{prop:integrability-signature}.  Assume that $n$ is large enough so that
\begin{eqnarray}\label{e4.6i} 
     G{n^{-\gamma}} \leq   {(8K_{{V}}^{2})^{-1} } \,.
\end{eqnarray}
In this step, we show by induction that for $(s,t )\in \mathcal{S}_{2}(\ll 0,T \rr) $ such that 
\begin{eqnarray}\label{e4.6} 
{G} (t-s)^{\gamma} \leq   {(8K_{{V}}^{2} )^{-1}} \,,
\end{eqnarray}
 we have 
 \begin{eqnarray}\label{e11}
\|y\|_{[s,t],  \gamma, n}   \leq 2K_{{V}}   {G}\,, \quad  \|R\|_{[s,t],  3\gamma, n} \leq 8K_{{V}}^{3}   G^{3}  \, .
\end{eqnarray}
Notice that here and in the following, we adopt the notations: 
\begin{equation*} 
\|y\|_{[s,t], \al, n}:=\sup_{(u,v)\in\cs_{2}(\ll s,t \rr )}\frac{|\delta y_{uv}|}{|v-u|^{\al}} , 
\quad\text{and}\quad 
\|R\|_{[s,t],  \al, n}:= \sup_{(u,v)\in\cs_{2}(\ll s,t \rr )}\frac{|  R_{uv}| }{|v-u|^{  \al}} 
\end{equation*}
for $\al>0$.
The relations \eref{e11} will be achieved by bounding successively the terms in \eref{e4.5}. 

Specifically, we assume that relation \eqref{e11} holds true when $(s,u,t)\in\cs_{3}(\ll 0, (N-1)h\rr)$ and verify \eref{e4.6}. Our aim is to extend this inequality to $\cs_{3}(\ll 0, Nh\rr)$. We thus start from our induction assumption, and we consider $(s,u,t)\in\cs_{3}(\ll 0, Nh\rr)$ such that   \eqref{e4.6} is satisfied  and $t=Nh$. Then we start by bounding the terms $A_{11}$ and $A_{3}$ as follows:
\begin{eqnarray*}
A_{11}+A_{3} &=&  
\sum_{i =1}^{m}[\partial V_{i}(y^{n})]_{su}      {V} (y^{n}_{s}) \delta B_{s u}   \delta B^{i}_{ut}
- \sum_{i,j=1}^{m}  ( \partial {V}_{i}{V}_{j})(y^{n}_{s})  \delta B_{su}^{j} \delta B_{ut}^{i}
\\
&=& 
\sum_{i,j=1}^{m} ( [\partial V_{i}(y^{n})]_{su} - \partial V_{i}(y^{n}_{s}) )     {V}_{j}(y^{n}_{s}) \delta B^{j}_{s u}   \delta B^{i}_{ut}
\end{eqnarray*}
By the induction assumption \eqref{e11} on $\|y^{n}\|_{[s,t],\ga,n}$ and the definition \eqref{eq:def-G} of our random variable $G$, we thus have 
  \begin{eqnarray*}
\| A_{11}+A_{3}\|_{[s,t],3\gamma, n} &\leq& 2 K_{{V}}^{2} G^{3}.
\end{eqnarray*}
Along the same lines, since $\| \BB \|_{\ga} \leq G$ and invoking the induction assumption again we obtain
\begin{eqnarray*}
 \| A_{2}\|_{[s,t],3\gamma, n}  \leq   2 K_{{V}}^{2}  G^{3}\quad \quad \text{and}
 \quad\quad
 \|A_{12}\|_{[s,t],3\gamma, n}  \leq  K_{{V}} G^{3}.
\end{eqnarray*}
Similarly, the estimate \eref{e4.21} and the induction assumption implies that 
\begin{eqnarray*}
 \|A_{4}\|_{[s,t],3\gamma, n}  \leq   2 K_{{V}}^{2}  G^{3}
 \quad \quad 
 \text{and}
 \quad\quad
 \|A_{13}\|_{[s,t],3\gamma, n}  \leq    K_{{V}}  G^{3} .
\end{eqnarray*}
Finally, by the induction assumption \eqref{e11} on $R$ we obtain
\begin{eqnarray*}
  \|A_{14}\|_{[s,t],3\gamma, n} &\leq& 8K_{{V}}^{4}   G^{4} (t-s)^{\ga} \leq K_{{V}}^{2} G^{3} ,
\end{eqnarray*}
where we have used the assumption \eqref{e4.6} for the second inequality.
Applying the above estimates on $A_{1},\ldots,A_{4}$ to \eref{e4.5} we have thus obtained:
\begin{eqnarray*} 
\|\delta R\|_{[s,t],3\ga, n} &\leq& 8K_{{V}}^{2}  G^{3}\,.
\end{eqnarray*}
Since $R_{t_{k}t_{k+1}} =0$ and $3\ga>1$, we are now in a position to apply the discrete sewing Lemma~\ref{lem2.4}. This yields:
\begin{eqnarray}\label{e4.7}
\|  R\|_{[s,t],3\ga, n} &\leq& 8K_{{V}}^{3}  G^{3}\,.
\end{eqnarray}
Otherwise stated, our induction assumption \eqref{e11} is propagated for the term $R$.

Let us turn to the propagation of the induction assumption \eref{e11} for the norm of $y$. Plugging the bound   \eref{e4.7} into relation \eref{e.10}, taking onto account the definition of the random variable $G$ and recalling relation \eqref{e4.6}, it is readily checked that: 
\begin{eqnarray}\label{e5.6i}
 | \delta y^{n}_{st} - V(y^{n}_{s}) \delta B_{st} |  &\leq&     K_{{V}} G^{2} (t-s)^{2\ga} +  K_{{V}} G^{2}  (t-s)^{2\ga} +   8K_{{V}}^{3} G^{3}(t-s)^{3\ga}
 \nonumber
\\
&\leq& 3K_{{V}}G^{2} (t-s)^{2\ga}.
\end{eqnarray}
Therefore, since we have $\| B\|_{\ga}\leq G$, we obtain
\begin{eqnarray}\label{e5.3ii}
|\delta y^{n}_{st}|   \leq K_{{V}} G (t-s)^{\ga} +      2K_{{V}} G^{2} (t-s)^{2\ga} +   8K_{{V}}^{3} G^{3}(t-s)^{3\ga}
 \leq  2K_{{V}}G(t-s)^{\ga},
\end{eqnarray}
where we have invoked our hypothesis \eref{e4.6} again. 
This achieves the propagation of the induction \eref{e11} for the term $\|y^{n}\|_{\ga}$.

\noindent{\it Step 3:  Upper bound estimates on $\ll 0,T \rr$.} 
Recall that we have proved relation \eqref{e11} on small intervals $\ll s,t\rr$ satisfying \eqref{e4.6}. 
In order to extend this result  to the whole interval $\ll 0, T \rr$,  we use a partition of the form $\ll kT_{0}, (k+1)T_{0} \rr$. 
Namely, consider $T_{0} \in \ll0,T\rr$ such that  
\begin{eqnarray}\label{eq:cdt-T0}
 {(2^{\gamma}8 K_{{V}}^{2})^{-1}}\leq GT_{0}^{\ga} \leq  {(8K_{{V}}^{2})^{-1}}.
\end{eqnarray}
%By \eref{e11} and the fact that $ s$ and $s+T_{0}$ satisfies \eref{e4.6}  we have
%\begin{eqnarray*}
%|\delta y^{n}_{s, s+T_{0}}| &\leq& \frac{1}{4K_{{V}}}
%\end{eqnarray*}
%for all $s\in \ll 0,T \rr$.
Also consider  $s,t\in[0,T]$ such that $t-s> T_{0}$, and  denote $k=\lfloor \frac{t-s}{T_{0}}\rfloor$ and $s_{i} = s+iT_{0}$, $i=0,\dots, k$. Then we obviously have
\begin{equation*}
|\delta y^{n}_{s t}| 
\leq  
|\delta y^{n}_{s_{0}s_{1}}| +|\delta y^{n}_{s_{1} s_{2}} | +\cdots+  |\delta y^{n}_{s_{k}t}| 
\end{equation*}
Furthermore, on each subinterval $[s_{j},s_{j+1}]$ one can apply \eqref{e11} in order to get:
\begin{equation*}
|\delta y^{n}_{st}| 
\leq  
2K_{{V}} G k T_{0}^{\ga} + 2K_{{V}} G (t-s-kT_{0})^{\ga}.
\end{equation*}
Now resort to the fact that $k \le \frac{t-s}{T_{0}}$ and inequality \eqref{eq:cdt-T0}. This yields, for another   $K$,
\begin{equation}\label{e5.3i}
|\delta y^{n}_{st}|
\leq 4K_{{V}}G \frac{t-s}{T_{0}^{1-\ga}}
\leq K G^{\frac{1}{\ga}} (t-s) .
\end{equation}
Hence, gathering our estimates \eref{e5.3ii} and \eref{e5.3i}, we end up with 
 \begin{eqnarray}\label{e5.5ii}
|\delta y^{n}_{st}| &\leq&  K   G^{\frac{1}{\ga}} (t-s)^{\ga},
\end{eqnarray}
for $(s,t )\in \mathcal{S}_{2}(\ll 0,T \rr)$. That is, we have extended the first part of \eqref{e11} to the whole interval $\ll 0,T \rr$, and thus we have proved the first relation in \eqref{eq:bnd-delta-yn} when $n$ satisfies \eqref{e4.6i}.

We now prove the second-order estimate in \eqref{eq:bnd-delta-yn} when $n$ satisfies \eqref{e4.6i}. We start by a new decomposition of the form:
\begin{eqnarray}\label{e4.10i}
 |\delta y^{n}_{st} - V(y^{n}_{s}) \delta B_{st} | 
&\leq& |r_{1}|+|r_{2}| ,
\end{eqnarray}
where 
\begin{equation*}
r_{1}= \delta y^{n}_{st} - \sum_{i=0}^{k } V(y^{n}_{s_{i}}) \delta B_{s_{i}, t\wedge s_{i+1}} ,
\quad\text{and}\quad
r_{2}=   \sum_{i=0}^{k } V(y^{n}_{s_{i}}) \delta B_{s_{i}, t\wedge s_{i+1}}- V(y^{n}_{s}) \delta B_{st} .
\end{equation*}
Now the term $r_{1}$ can be bounded as follows:
\begin{eqnarray*}
|r_{1} | &\leq& \sum_{i=0}^{k } | \delta y^{n}_{s_{i}, t\wedge s_{i+1}} - V(y^{n}_{s_{i}}) \delta B_{s_{i}, t\wedge s_{i+1}}  |. 
\end{eqnarray*}
Therefore, by \eref{e5.6i} we obtain
\begin{eqnarray*}
|r_{1}| &\leq& 2(k+1)K_{{V}}G^{2} T_{0}^{2\ga} .
\end{eqnarray*}
Moreover, since $k\leq \frac{t-s}{T_{0}}$ and $\frac{1}{T_{0}} \leq  ({2^{\ga}8 K_{V}^{2} G })^{\frac{1}{\ga}} $ owing to \eqref{eq:cdt-T0}, we can recast the previous equation as:
\begin{eqnarray}\label{e4.10}
|r_{1}| &\leq& 4K_{V}(2^{\ga} 8 K_{V}^{2})^{\frac{1}{\ga} -2} G^{\frac{1}{\ga}} (t-s). 
\end{eqnarray}
In order to bound $r_{2}$, observe that we have 
\begin{eqnarray*}
|r_{2}| &\leq& \sum_{i=0}^{k } | V(y^{n}_{s_{i}}) - V(y^{n}_{s})| \cdot |\delta B_{s_{i}, t\wedge s_{i+1}}|.
 \end{eqnarray*}
Thanks to \eref{e5.5ii},  we thus have
\begin{equation*}
|r_{2}| \leq  2k  K_{V} (K   G^{\frac{1}{\ga}} (t-s)^{\ga} ) G T_{0}^{\ga}.
\end{equation*}
Invoking again the inequalities $k\leq \frac{t-s}{T_{0}}$ and $\frac{1}{T_{0}} \leq  ({2^{\ga}8 K_{V}^{2} G })^{\frac{1}{\ga}}$, we thus get:
\begin{equation}\label{e4.11}
|r_{2}|
\le
2K_{V}K (2^{\ga}8K_{V}^{2} )^{\frac{1}{\ga}-1} (  G^{ \frac{2}{\ga}})(t-s)^{1+\ga}.
\end{equation}
Applying \eref{e4.10} and \eref{e4.11} to relation \eref{e4.10i}, this yields:
\begin{eqnarray}\label{e5.6}
|\delta y^{n}_{st} - V(y^{n}_{s}) \delta B_{st} | &\leq&  K   G^{ \frac{2}{\ga}} (t-s)^{2\ga}.
\end{eqnarray}
We have now proved \eref{eq:bnd-delta-yn}
  under the assumption \eref{e4.6i}.

\noindent{\it Step 4: Upper-bound estimate for small $n$.}  
We are now reduced to prove inequalities \eref{eq:bnd-delta-yn}
  when \eref{e4.6i} is not satisfied. Namely, we assume in this step that
 \begin{equation}\label{e5.3}
  {G}{n^{-\ga}} >   {(8K_{{V}}^{2})^{-1}},
 \quad\text{i.e}\quad
 n < (8K_{{V}}^{2}G)^{\frac{1}{\ga}}. 
\end{equation}
For $(s,t )\in \mathcal{S}_{2}(\ll 0,T \rr)$, we will also resort to the same partition $t_{0},\ldots,t_{k+1}$ as in the previous step. In this case, due to the very definition \eref{e1.2} of $y^{n}$, it is readily checked that 
\begin{eqnarray}\label{e4.19}
|\delta y^{n}_{t_{k}t_{k+1}} | &\leq&   K_{{V}} G\Big(\frac Tn\Big)^{H} + K_{{V}} \Big(\frac Tn\Big)^{2H}.
\end{eqnarray}
Therefore, summing \eref{e4.19} between $s$ and $t$ for $(s,t) \in \cs_{2}(\ll 0,T \rr)$ we get:
\begin{equation*}
|\delta y^{n}_{s t} | = |\sum_{t_{k}=s}^{t-} \delta y^{n}_{t_{k} t_{k+1}} |
\leq  {n(t-s)}{T^{-1}} \Big[    K_{{V}} G\Big(\frac Tn\Big)^{H} + K_{{V}} \Big(\frac Tn\Big)^{2H} \Big].
\end{equation*}
Taking into account the estimate of $n$ in \eref{e5.3}, this yields:
\begin{eqnarray}\label{e5.5i}
|\delta y^{n}_{st} |  &\leq& KG^{\frac{1}{\ga}}|t-s|, 
\end{eqnarray}
for $(s,t )\in \mathcal{S}_{2}(\ll 0,T \rr)$. We have thus proved the first relation in \eref{eq:bnd-delta-yn} when \eref{e4.6i} is not met.

In order to handle the second relation in \eref{eq:bnd-delta-yn} for $n$ small, just decompose the increment at stake along our partition $t_{0},\ldots,t_{k+1}$:
\begin{equation*}
|V(y^{n}_{s}) \delta B_{st} | 
\leq |V(y^{n}_{s}) |\sum_{t_{k}=s}^{t-} | \delta B_{t_{k}t_{k+1}} |
\leq KG n^{1-\ga}(t-s)
\leq
 K G^{\frac{1}{\ga}} (t-s). 
\end{equation*}
Taking into account inequality \eqref{e5.5i}, we thus easily get:
\begin{eqnarray}\label{e5.8}
| \delta y^{n}_{st} -  V(y^{n}_{s}) \delta B_{st} | &\leq& K G^{\frac{1}{\ga}} (t-s), 
\end{eqnarray}
which achieves the second relation in \eref{eq:bnd-delta-yn} for small $n$.

\noindent{\it Step 5: Conclusion.}
Gathering the estimates \eref{e5.5ii}, \eref{e5.5i}, we have obtained the desired  estimate for $\|y^{n}\|_{\ga}$ on $\ll 0, T \rr$, for all $n$. In the same way, putting    \eref{e5.6} and \eref{e5.8} together implies the   second estimate in \eref{eq:bnd-delta-yn} on $\ll 0, T \rr$ and for any $n$. The proof is complete. 
\end{proof}

\subsection{The couple $  (y^{n},B) $ as a rough path}

Our next aim is to prove that $(y^{n},B)$ can be lifted as a rough path, which amounts to a proper definition of the signature $S_{2}(y^{n},B)$ as given in Definition \ref{def:rough-path}.
 The   result below, providing an estimate of the integral $\int_{s}^{t} \delta y^{n}_{su} \otimes dB_{u}$, can be seen as an important step in this direction.   Note that on each interval $[t_{k},t_{k+1}]$, the process $y^{n}$ is a controlled process with respect to $B$, as alluded to in \eqref{eq:dcp-controlled-process}. For each $n$, the integral $\int_{s}^{t} \delta y^{n}_{su} \otimes dB_{u}$ is thus defined as
\begin{equation}\label{eq:piecewise-rough-intg-yn-dB}
\int_{s}^{t} \delta y^{n}_{su} \otimes dB_{u}
= \sum_{t_{k}=s}^{t_{-}} \int_{t_{k}}^{t_{k+1}} \delta y^{n}_{su} \otimes dB_{u},
\end{equation}
thanks to classical rough paths considerations.

 \begin{lemma} \label{lem5.2}
 Let $y^{n}$ be the process defined by the   Euler scheme \eref{e1.2}, and consider $\frac13<\ga<H$.
Then we can find a random variable $G\in\cap_{p\ge 1}L^{p}(\Omega)$ independent of $n$, such that for the integral $\int_{s}^{t} \delta y^{n}_{su} \otimes dB_{u}$ in the sense of \eref{eq:piecewise-rough-intg-yn-dB}, we have  the estimate
 \begin{eqnarray}\label{eq:bound-iter-intg-y-dB}
 \Big|\int_{s}^{t} \delta y^{n}_{su} \otimes dB_{u} \Big|  &\leq &  G  (t-s)^{2 \ga},
 \quad\text{for}\quad (s,t )\in \mathcal{S}_{2}(\ll 0,T \rr).
\end{eqnarray}
  \end{lemma}

\begin{proof}  
Similarly to what we did for Proposition  \ref{lem5.1}, we will assume that $b=0$ for this proof, and analyze the scheme given by \eref{eq4.4}. Next,
 in order to bound the integral $\int_{s}^{t} \delta y^{n}_{su} \otimes dB_{u}$, let us define two increments: first, just as in the definition \eqref{eq:def-zeta1-zeta2}, we set
\begin{equation*}
\zeta_{st}^{2} 
=
\int_{s}^{t} (u-s)^{2H}  dB_{u}.
\end{equation*}
Then we define a remainder type increment $\wt{R}$ on $\cs_{2}(\ll 0,T\rr)$ by:
 \begin{eqnarray}\label{e5.10}
\wt{R}_{st} &=&  \int_{s}^{t} \delta y^{n}_{su} \otimes dB_{u} - {V}(y^{n}_{s})   \BB_{st} - \frac12\sum_{j=1}^{m} \partial {V}_{j}{V}_{j} (y^{n}_{s})\otimes \zeta_{st}^{2} \,.
\end{eqnarray}
According to the definition \eref{e1.2} of our scheme, it is clear that $\wt{R}_{t_{k}t_{k+1}} = 0$, for all $k=0,1,\dots, n-1$. Moreover, applying $\delta$ to $\wt{R}$
and recalling the elementary rule $ \delta ( \int dy \otimes dB ) = \delta y \otimes \delta B $, 
 we obtain:
\begin{multline*}
\delta \wt{R}_{srt}  = (\delta y^{n}_{sr}  -  {V} (y^{n}_{s}) \delta B_{sr}) \otimes \delta B_{rt} + \delta {V}(y^{n}_{\cdot})_{sr} \BB_{rt}  
\\
-    \frac12\sum_{j=1}^{m} \partial {V}_{j}{V}_{j} (y^{n}_{s}) \otimes \delta\zeta_{srt}^{2}
 + \frac12\sum_{j=1}^{m}
\delta  \left(   \partial {V}_{j}{V}_{j} (y^{n})\right)_{sr} \otimes \zeta^{2}_{rt}    ,
\end{multline*}
where we remark that $\delta\zeta_{srt}^{2} = \int_{r}^{t}[ (u-s)^{2H} - (u-r)^{2H} ]  dB_{u}$.
Starting from the above expression, one can thus apply Proposition  \ref{prop:integrability-signature} and Proposition \ref{lem5.1}  in order to get:
\begin{eqnarray*}
\|\delta \wt{R} \|_{3\ga} &\leq& G.
\end{eqnarray*}
Note that for the Young integral $\delta \zeta^{2}$ we used the following estimate, valid for $(s,r, t) \in  \cs_{3}([0,T])$,
\begin{eqnarray*}
\delta \zeta^{2}_{srt} \leq \Big| \int_{r}^{t}  (u-s)^{2H}   dB_{u}\Big| +  \Big|\int_{r}^{t}  (u-r)^{2H}   dB_{u}\Big|
 \leq \|B\|_{\ga} (t-s)^{2H+\ga}.
\end{eqnarray*}
Therefore, since $\| \delta \tilde{R} \|_{3\ga} \leq G$, it follows from the sewing Lemma \ref{lem2.4} that
\begin{eqnarray*}
|  \wt{R}_{st} | &\leq& G(t-s)^{3\ga}.
\end{eqnarray*}
Our claim \eqref{eq:bound-iter-intg-y-dB} is now easily deduced from the above estimate of $\wt{R}$ and expression \eref{e5.10}.  
\end{proof}

Now we provide some estimates for the iterated integral $\int_{s}^{t} \delta y^{n}_{su}\otimes dy^{n}_{u}$, which is also part of the rough path above $(y^{n},B)$. Note that $\int_{s}^{t} \delta y^{n}_{su}\otimes dy^{n}_{u}$ is  defined as
\begin{eqnarray*}
\int_{s}^{t} \delta y^{n}_{su}\otimes dy^{n}_{u} &=& \sum_{t_{k}=s}^{t-} \int_{t_{k}}^{t_{k+1}} \delta y^{n}_{su}\otimes dy^{n}_{u},
\end{eqnarray*}
  in the same way  as for the integral $\int_{s}^{t} \delta y^{n}_{su}\otimes dB_{u}$.

\begin{lemma} \label{lem5.3}
 Let the assumptions be as in Lemma \ref{lem5.2}.  Then the following estimate holds true:
 \begin{eqnarray*}
 \Big| \int_{s}^{t} \delta y^{n}_{su}\otimes dy^{n}_{u} \Big|  &\leq &    G (t-s)^{2 \ga},
 \quad\text{for}\quad (s,t )\in \mathcal{S}_{2}(\ll 0,T \rr),
\end{eqnarray*}
where $G$ is a random variable in $\cap_{p\ge 1}L^{p}(\Omega)$, independent of $n$.
  \end{lemma}

\begin{proof}  
We proceed similarly as in the proof of Lemma  \ref{lem5.2}. Namely, we still assume $b=0$ for sake of conciseness, and  the existence of the integral $\int_{s}^{t} \delta y^{n}_{su}\otimes dy^{n}_{u}$ is justified as for \eqref{eq:piecewise-rough-intg-yn-dB}.
Next we define a remainder type increment $\bar{R}$ on $\cs_{2}(\ll 0,T\rr)$ by:
\begin{equation}\label{e5.11}
\bar{R}_{st} =  \int_{s}^{t} \delta y^{n}_{su} \otimes dy^{n}_{u} - {V}(y^{n}_{s})   \int_{s}^{t} \delta B_{su} \otimes d y^{n}_{u}
- \frac12\sum_{j=1}^{m} \partial {V}_{j}{V}_{j} (y^{n}_{s})\otimes \int_{s}^{t} (u-s)^{2H}  dy^{n}_{u}\,.
\end{equation}
As previously, it is clear that $\bar{R}_{t_{k}t_{k+1}} = 0$. In the same way as in Lemma \ref{lem5.2}, we can also show that $\|\delta \bar{R}\|_{3\ga} \leq G$, so by Lemma \ref{lem2.4} we obtain
\begin{eqnarray*}
|  \bar{R}_{st}| &\leq& G(t-s)^{3\ga}.
\end{eqnarray*}
Applying this estimate to \eref{e5.11}, we obtain the desired estimate for $\int_{s}^{t} \delta y^{n}_{su}\otimes dy^{n}_{u} $. 
\end{proof}

We can now conclude and get a uniform bound on $(y^{n},B)$ as a rough path. 

\begin{prop}\label{prop5.2}
Let $y$ be the solution of equation \eref{e1.1} and $y^{n}$ be the solution of the   Euler scheme \eref{e1.2}. Consider $\frac13 <\ga<H<\frac12$, and set:
\begin{equation*}
 S_{2}( B, y^{n} )_{st}
= 
\left(  ( \delta B_{st}\,, \delta y^{n}_{st}), \int_{s}^{t} (\delta B_{su}\,,\delta y^{n}_{su} ) \otimes d (\delta B_{u}\,, \delta y^{n}_{u}) \right) .
\end{equation*}
Then $S_{2}( B, y^{n} )$ can be considered as a $\ga$-rough path according to Definition \ref{def:rough-path}. In addition,   there exists a random variable $G\in\cap_{p\ge 1}L^{p}(\Omega)$ independent of $n$ such that:
\begin{eqnarray*}
 \| S_{2}( B, y^{n} ) \|_{\ga} \leq G \,,
\end{eqnarray*}
where $\| \cdot \|_{\ga}$ is defined by \eref{eq:def-norm-rp}.
\end{prop}
\begin{proof}
Putting together the results of Proposition \ref{lem5.1}, Lemma \ref{lem5.2} and Lemma \ref{lem5.3}, we easily get the definition of $S_{2}( B, y^{n} )$, together with the bound:
\begin{eqnarray*}
|S_{2}( B, y^{n} )_{st}| &\leq& G(t-s)^{\ga},
\quad\text{for}\quad
(s,t) \in \cs_{2}(\ll 0,T \rr) .
\end{eqnarray*}
On the other hand, by the definition of $y^{n}$  it is clear the same estimate holds for $s,t \in [t_{k},t_{k+1}]$, $k=0,\dots, n-1$. The proposition then follows by applying  Lemma \ref{lem4.5} to $S_{2}( B, y^{n} )$.
\end{proof}

\section{Almost sure convergence of the   Euler scheme}\label{section5}

We now take advantage of the information gathered up to now, and show  the almost sure convergence of the   Euler scheme~\eref{e1.2}.   Notice however that the convergence rate obtained in this section is not optimal, and has to be seen as a preliminary step; see Section~\ref{section7.2} for a more accurate result.

\begin{remark}\label{rmk:strategy-linear-eq}
The approximation process $y^{n}$ is discrete by nature, and the reader might wonder why we have spent some effort trying to show that $(y^{n},B)$ is a rough path. The answer will be clearer within the landmark of the current section. Indeed, our analysis of the numerical scheme mainly hinges on the fact that the renormalized error satisfies a linear equation driven by both $y^{n}$ and $B$. The best way we have found to properly define this equation is by showing that $(y^{n},B)$ can be seen as a rough path. Let us mention however two alternative ways to get the same kind of result:

\noindent
\emph{(i)}
We could have relied on the fact that $y^{n}$ is a controlled process   with respect to $B$; see~\eref{eq:dcp-controlled-process} and \cite{FH, G} for the notion of controlled process. However, due to the fact that $y^{n}$ is defined on a discrete grid, we haven't been able to find a satisfactory way to see $y^{n}$ as a continuous time controlled process.

\noindent
\emph{(ii)}
We could also have dealt with a discrete version of the linear equation, which governs the error process on our discrete grid. Nevertheless we believe that the continuous time version exhibited below is more elegant, and this is why we have sticked to the continuous time strategy.
\end{remark}

With Remark \ref{rmk:strategy-linear-eq} in mind, we will now introduce the linear equation which will govern the error process, and then analyze the Euler scheme. Throughout the section  we assume that $b \in C^{2}_{b}$ and $V\in C^{4}_{b}$. 

\subsection{A linear rough differential equation}
Recall that we are dealing with the unique solution $y$ to the following equation: 
\begin{eqnarray}\label{e5.1}
dy_{t} &=& b(y_{t}) dt + {V} (y_{t}) d B_{t},\quad t\in [0,T].
\end{eqnarray}
Its numerical approximation  $y^{n}$ is given by  the   Euler scheme \eref{e1.2}.  As we shall see later in the paper, the error process is   governed by a kind of discrete equivalent of the Jacobian for equation \eref{e5.1}. Specifically, we consider the following linear equation:  
\begin{eqnarray}\label{e6.1i}
\Phi^{n}_{t} &=&\id+\int_{0}^{t} \{ \partial b(y^{n}) \}_{s}\Phi^{n}_{s} ds + \sum_{j=1}^{m}  \int_{0}^{t} \{ \partial V_{j}(y^{n}) \}_{s}  \Phi^{n}_{s} dB_{s}^{j},
\end{eqnarray}
where $\id$ is the $d\times d$ identity matrix, and where we have set
\begin{equation*}
\{ \partial V_{j}(y^{n}) \}_{s} 
=  \int_{0}^{1} \partial {V}_{j}( y^{n}_{s} + \lambda  (y_{s}- y^{n}_{s} ) )  d \lambda,
\quad\text{and}\quad
\{ \partial b(y^{n}) \}_{s} 
= \int_{0}^{1} \partial b ( y^{n}_{s} + \lambda  (y_{s}- y^{n}_{s} ) ) d\lambda.
\end{equation*}
In this subsection, we derive an upper-bound estimate of $\Phi^{n}$ and its inverse $\Psi^{n}$ based on Proposition \ref{prop5.2}.

 \begin{prop}\label{prop:bnd-Phi-n}
The linear equation \eref{e6.1i} has a unique solution $\Phi^{n}$, and there exists an integrable random variable $G$ such that the following estimate holds true:  
\begin{eqnarray}\label{e5.3iii}
 |  S_{2}({\Phi}^{n}, z^{n} )_{st}|     &\leq& e^{G}(t-s)^{\ga},
\end{eqnarray}
where we recall that 
$z^{n} = (y^{n}, B)$ and
 the signature $S_{2}$ is introduced in Definition \ref{def:rough-path}.
Furthermore,    $\Phi^{n}$ admits an inverse process $\Psi^{n}\equiv(\Phi^{n})^{-1}$  and    estimate~\eqref{e5.3iii} also holds for  $\Psi^{n}$.  
 \end{prop}
 
\begin{proof}
Define two $\RR^{d\times d}$-valued processes $\theta$ and $\xi$ respectively by 
\begin{eqnarray*}
\theta^{il}_{t} =  \sum_{j=1}^{m}  \int_{0}^{t}  \{\partial_{i}V^{l}_{j}(y^{n})\}_{s}    dB_{s}^{j} 
\quad \text{ and } \quad \xi^{il}_{t} = \int_{0}^{t} \{\partial_{i}b^{l}(y^{n})\}_{s}ds.
\end{eqnarray*}
    Then we can easily recast equation~\eqref{e6.1i} as:
\begin{eqnarray}\label{e6.1}
\Phi^{n,l}_{t} &=& \delta_{l}+ \sum_{i=1}^{d} \int_{0}^{t} \Phi^{n,i}_{s} d\xi^{il}_{s} + \sum_{i=1}^{d} \int_{0}^{t} \Phi^{n,i }_{s}  d\theta^{il}_{s}\,.
\end{eqnarray}
In particular, $\Phi^{n}$ satisfies a linear equation driven by $\zeta = \{\theta^{il}, \, \xi^{il}; \, i,l=1,\dots, d\}$.
By the estimate of $\|S_{2}(   z^{n}  )\|_{\ga}$ contained in Proposition \ref{prop5.2}, 
we can show that for $   S_{2}(\zeta,  z^{n}) $ we have
$
\|   S_{2}(\zeta,  z^{n} )  \|_{\ga} \leq G$, 
where $G$ is an integrable random variable independent of $n$. 
So applying Theorem \ref{thm 3.3} to equation~\eref{e6.1},  we obtain  
\begin{equation*}
 |  S_{2}({\Phi}^{n} ,  z^{n})_{st}|     
 \leq 
 K_{1} \|  S_{2}(\zeta, z^{n})    \|_{\ga} (t-s)^{\ga} \exp \left( K_{2} \|  S_{2}(\zeta, z^{n})  \|_{\ga}^{1/\ga} \right) 
   ,
\end{equation*}
and the estimate \eref{e5.3iii} then follows.
The estimate of the inverse of $    \Phi^{n}   $ can be obtained in the same way. 
 \end{proof}
 \begin{remark}
 Note that from the proof of Proposition \ref{prop:bnd-Phi-n}, it is not clear that the random variable $e^{G}$ in \eref{e5.3iii} is integrable. However, the almost sure bound  \eref{e5.3iii} will be enough for our use in deriving the almost sure convergence rate of the   Euler scheme \eref{e1.2}. Let us mention that the methodology adopted in \cite{CLL} in order to get the integrability   of the Jacobian of a RDE driven by Gaussian processes does not apply  to  equation \eref{e6.1i}. This is due to the fact that \eref{e6.1i} involves the process $y^{n}$, which is the solution of a ``discrete'' RDE  driven by both $B$ and $F$  (recall that $F$ is defined in \eref{e4.1}). We believe that a discrete strategy in order to bound $\Phi^{n}$ would lead to the integrability of $ |  S_{2}({\Phi}^{n})_{st}|  $, but we haven't delved deeper into this direction for sake of conciseness. 
 \end{remark}

 \subsection{Error process as  a rough path}

In this subsection, we derive some estimates on the error process of the   Euler scheme. To this aim, we will first write the process $y^{n}$ as the solution of a differential equation in continuous time. Namely, it is readily checked that one can recast equation \eref{e1.2} as follows:
\begin{equation}\label{eq:yn-diff-eq}
y_{t}^{n} = y_{0} + \int_{0}^{t} b(y^{n}_{\eta(s)}) \, ds 
+ \int_{0}^{t} V(y^{n}_{\eta(s)}) \, dB_{s} -  {A_{t}^{1}}{ },
\end{equation}
where we have set
\begin{equation}\label{eq:def-eta-A1}
\eta(s) = t_{\lfloor \frac{ns}{T}\rfloor},
\quad\text{and}\quad
A_{t}^{1}
=
- \frac12 \sum_{k=0}^{\lfloor \frac{nt}{T} \rfloor} \sum_{j=1}^{m} \partial {V}_{j}{V}_{j} (y^{n}_{t_{k}}) (t\wedge t_{k+1}-t_{k})^{2H}.
\end{equation}
Note that the dependence of $A^{1}$ on $n$ is omitted for simplicity. With this simple algebraic decomposition in hand, we can state the following bound on the error process:

 \begin{lemma}\label{prop6.1}
 Let $y$, $y^{n}$, and $\Phi^{n}$ be the solution of equations \eref{e5.1},   \eref{e1.2}, and    \eref{e6.1i}, respectively, and $\Psi^{n}$ be the inverse process of $\Phi^{n}$.  Consider the path $\ep$ defined by
\begin{equation}\label{eq:def-ep-t}
\ep_{t} = \Psi^{n}_{t}(y_{t} - y^{n}_{t}).
\end{equation}
Then for all $\frac13 <\ga<H<\frac12$, we can find an almost surely finite random variable $G$ independent of $n$ such that:
 \begin{eqnarray}\label{e6.7}
| \delta \ep_{st}| &\leq& G (t-s)^{1- \ga} n^{1-3\ga}, \quad (s,t) \in \cs_{2}(\ll 0,T \rr).
\end{eqnarray}
 \end{lemma}

\begin{proof}
Putting together equations \eqref{e5.1} and  \eqref{eq:yn-diff-eq}, it is easily seen that:
\begin{equation*}
y_{t}-y^{n}_{t} =  \int_{0}^{t} (b(y_{s}) -b(y^{n}_{\eta(s)})) ds + \int_{0}^{t}( {V}(y_{s}) -{V}(y^{n}_{\eta(s)}) ) d B_{s} +A^{1}_{t}. 
\end{equation*}
In addition, the chain rule for rough integrals enables to write:
\begin{equation*}
\vp(y_{s}^{n}) - \vp(y_{\eta(s)}^{n}) 
=
\int_{\eta(s)}^{s} \partial \vp(y^{n}_{u}) dy^{n}_{u},
\end{equation*}
for any $\vp\in C^{1/\ga}(\RR^{d};\RR^{d})$ and $s\in[0,T]$, and where $\partial \vp$ designates the gradient of $\vp$. Owing to this relation, applied successively to $b$ and $V$, we get:
\begin{equation*} 
y_{t}-y^{n}_{t} 
=  
\int_{0}^{t} (b(y_{s}) -b(y^{n}_{s})) ds + \int_{0}^{t}( {V}(y_{s}) -{V}(y^{n}_{s}) ) d B_{s} 
+ \sum_{e=1}^{3} A_{t}^{e},
\end{equation*}
where we recall that $A^{1}$ is defined by \eqref{eq:def-eta-A1}, and where we have set:
\begin{equation}\label{e6.4ii}
A_{t}^{2} = \int_{0}^{t} \int_{\eta(s)}^{s} \partial b(y^{n}_{u}) dy^{n}_{u} ds,
\quad\text{and}\quad
A_{t}^{3}=   \sum_{j=1}^{m}\int_{0}^{t} \int_{\eta(s)}^{s} \partial  {V}_{j}(y^{n}_{u}) dy^{n }_{u} d B^{j}_{s}.
\end{equation}
Notice that $A_{t}^{3}$ above can be considered as a rough integral, thanks to Proposition \ref{prop5.2}. 
Taking into account the identity 
\begin{eqnarray}\label{eq5.10}
b(y_{s}) - b(y^{n}_{s}) = \{ \partial b (y^{n}) \}_{s} (y_{s} - y^{n}_{s}) \quad \text{ and } \quad   V_{j}(y_{s}) - V_{j}(y^{n}_{s}) = \{ \partial V_{j} (y^{n}) \}_{s} (y_{s} - y^{n}_{s}), 
\end{eqnarray}
we have
\begin{equation}\label{e6.4iii}
y_{t}-y^{n}_{t} 
=  
\int_{0}^{t}  \{ \partial b (y^{n}) \}_{s} (y_{s} - y^{n}_{s})  ds + \int_{0}^{t}\{ \partial V_{j} (y^{n}) \}_{s} (y_{s} - y^{n}_{s}) d B_{s} 
+ \sum_{e=1}^{3} A_{t}^{e}\,.
\end{equation}
Now starting from expression \eqref{e6.4iii} and applying the variation of parameter method to the equation \eref{e6.1i} governing $\Phi^{n}$,   
 it is easy to verify that
\begin{eqnarray}\label{e6.4j}
y_{t}-y^{n}_{t} &=&  \sum_{e=1}^{3} \Phi^{n}_{t} \int_{0}^{t} \Psi^{n}_{s} dA_{s}^{e} \,.
\end{eqnarray}
Therefore, we can also write:
\begin{eqnarray}\label{e6.4i}
\ep_{t} = \Psi^{n}_{t} (y_{t}-y^{n}_{t}) = \sum_{e=1}^{3}   \int_{0}^{t}  \Psi^{n}_{s}  dA_{s}^{e}\, , \quad\quad t \in [0,T].
\end{eqnarray}
Our claim \eqref{e6.7} thus follows  from Proposition \ref{prop:bnd-Phi-n}, together with  Lemma~\ref{lem6.2} below. 
\end{proof}

\begin{lemma}\label{lem6.2}
Let $A^{e}$, $e=1,2,3$, be as in \eqref{eq:def-eta-A1} and \eref{e6.4ii}. 
Let   $f$ be a continuous function with values in a finite dimensional vector space $\cv$ such that the path
\begin{eqnarray*}
S_{2}( f,B) :=  \Big( (f_{t}, B_{t}), \int_{0}^{t} (f_{s}, B_{s}) \otimes d(f_{s}, B_{s}) \Big)     
\end{eqnarray*}
is well defined. We also assume that there exists an a.s finite random variable $G$
 satisfying the upper bound $\| S_{2}( f,B) \|_{\ga} \leq G$     for any $\frac13<\ga<H<\frac12$. Then for all $(s,t )\in \mathcal{S}_{2}(\ll 0,T \rr)$ we have:
\begin{eqnarray*}
\Big| \sum_{e=1 }^{3} \int_{s}^{t} f_{u} \otimes d A_{u}^{e} \Big| &\leq& G(t-s)^{1-\ga} n^{1-3\ga}. 
\end{eqnarray*}
\end{lemma}

\begin{proof}
We divide this proof in several steps.

\noindent{\it Step 1: Decomposition of $A^{e}$.} 
Applying  the chain rule to $A^{3}$ we obtain:
 \begin{eqnarray*}
A_{t}^{3}
 &=&A_{t}^{31}  
+R_{t}^{2} \,,
\end{eqnarray*}
where the paths $A_{t}^{31}$ and $R_{t}^{2}$ are respectively defined by:
\begin{eqnarray*}
A_{t}^{31} =  \sum_{j=1}^{m} \int_{0}^{t} \int_{\eta(s)}^{s} \partial  {V}_{j}(y^{n}_{\eta(s)}) dy^{n }_{u} d B^{j}_{s} \,,
\quad\quad
R_{t}^{2} =  \sum_{j,j'=1}^{m}\int_{0}^{t} \int_{\eta(s)}^{s} \int_{\eta(s)}^{u}\partial_{j'}\partial  {V}_{j}(y^{n}_{v}) dy^{n,j'}_{v}dy^{n }_{u} d B^{j}_{s}.
\end{eqnarray*}
Moreover, recalling the equation \eqref{eq:yn-diff-eq} governing  $y^{n}$,  we obtain:
\begin{eqnarray*}
A_{t}^{31}&= &   A_{t}^{310}
 + R_{t}^{3}
 + R_{t}^{4},
\end{eqnarray*}
where we have set:
\begin{equation*}
A_{t}^{310}= \sum_{i,j=1}^{m} \int_{0}^{t} \int_{\eta(s)}^{s} \partial  {V}_{j}(y^{n}_{\eta(s)}) {V}_{i}(y^{n}_{\eta(s)})d B^{i}_{u} d B^{j}_{s}\,,
\quad
 R_{t}^{3} = \sum_{j=1}^{m} \int_{0}^{t} \int_{\eta(s)}^{s} \partial  {V}_{j}(y^{n}_{\eta(s)}) b(y^{n}_{\eta(s)}) du d B^{j}_{s}\,, 
\end{equation*}
and where
\begin{equation*}
R_{t}^{4} =   \frac{1}{2} 
\sum_{j,j'=1}^{m}\int_{0}^{t} \int_{\eta(s)}^{s} \partial  {V}_{j}(y^{n}_{\eta(s)}) 
 \partial {V}_{j'}{V}_{j'} (y^{n}_{\eta(s)})
d(u-\eta(s))^{2H} d B^{j}_{s}\,.
\end{equation*}
Summarizing our decomposition up to now, we have found that $A_{t}^{3} = A_{t}^{310} +R_{t}^{2}+R_{t}^{3} +R_{t}^{4}$.
Denoting
$
R_{t}^{5} = A_{t}^{310}   + A_{t}^{1}$,     {and}  $ R_{t}^{1} = A_{t}^{2}$,
    we can now express our driving process  $\sum_{e=1}^{3}A^{e}$ as a sum of remainder type terms:
\begin{eqnarray} \label{e6.5}
\sum_{e=1}^{3}A^{e}_{t} &=&\sum_{e=1}^{5}R_{t}^{e}.
\end{eqnarray}

\noindent{\it Step 2: Estimation  procedure.} 
We will now upper bound the terms $R^{e}$ given in our decomposition \eqref{e6.5}. For the term $R^{2}$, observe that (due to Proposition \ref{prop5.2}) the couple $(B,y^{n})$ can be seen as a $\ga$-H\"older rough path.   Taking into account all the time increments defining $R^{2}$, we obtain $|\delta R^{2}_{t_{k}t_{k+1}}| \leq Gn^{-3\ga}$ for all $t_{k}=s,\ldots,t_{-}$.
Therefore,  
\begin{eqnarray}\label{e5.4}
|\delta R_{st}^{2}  | \leq \sum_{t_{k}=s}^{t-} |\delta R^{2}_{t_{k}t_{k+1}}|   \leq G(t-s)  n^{1-3\ga}, \quad (s,t )\in \mathcal{S}_{2}(\ll 0,T \rr).
\end{eqnarray}
In the same way we can show that a similar estimate holds for $R^{4}$. In order to bound $R^{5}$, note that for $t \in \ll0,T\rr$ we have:
\begin{eqnarray*}
R_{t}^{5}&=& \sum_{i,j=1}^{m}\sum_{k=0}^{\frac{nt}{T}-1} ( \partial {V}_{j}{V}_{i}) (y^{n}_{t_{k}}) 
\delta F^{ij}_{t_{k}t_{k+1}}.
\end{eqnarray*}
Furthermore, by Proposition \ref{lem5.1}    one can show that the process $[ \partial {V}_{j}{V}_{i}] (y^{n}_{t_{k}})$ satisfies the conditions of Proposition \ref{prop3.6}. Hence, combining Corollary \ref{lem9.1}  and Lemma \ref{lem4.2}, we end up with the following inequality for $\kappa>0$ arbitrarily:
\begin{equation}\label{e6.7i}
|\delta R_{st}^{5}|
\leq
G(t-s)^{\frac12-\kappa}  n^{\frac12-2H+\kappa}
\leq
G(t-s)^{1-\gamma-2\kappa}  n^{1-3\ga}, \quad (s,t)\in \cs_{2}(\ll0,T\rr) ,
\end{equation}
where we have used the fact that $t-s\ge \frac{T}{n}$  for the last step.
The terms $R^{1}$ and $R^{3}$ are bounded along the same lines, in a slightly easier way due to the presence of Lebesgue type integrals. We get:
\begin{eqnarray}\label{e6.8i}
|\delta R_{st}^{e}|&\leq& G(t-s) n^{-\ga}  , \quad (s,t )\in \mathcal{S}_{2}(\ll 0,T \rr)
\end{eqnarray}
for $e=1,3$.
In summary of the estimates \eref{e5.4}, \eref{e6.7i} and \eref{e6.8i}, we have obtained:   
\begin{eqnarray}\label{e6.5i}
 \sum_{e=1}^{5} \left|  \delta R_{st}^{e}  \right| &\leq& G(t-s)^{1-\ga-2\kappa} n^{1-3\ga} , \quad (s,t )\in \mathcal{S}_{2}(\ll 0,T \rr).
\end{eqnarray}

\noindent{\it Step 3: Conclusion.} 
Thanks to our decomposition   \eref{e6.5} we  can  write  
\begin{eqnarray}\label{e5.5}
 \sum_{e=1}^{3}  \int_{s}^{t} f_{u} \otimes dA_{u}^{e}   &=&  
 \sum_{e=1}^{5}  \int_{s}^{t} \delta f_{\eta(u),u} \otimes dR_{u}^{e} + \sum_{e=1}^{3}  \int_{s}^{t}   f_{\eta(u)} \otimes dA_{u}^{e}
 \equiv B_{st}^{1} + B_{st}^{2}.
\end{eqnarray}
Let us start by bounding the term $B^{2}$. Similarly to relation \eqref{e3.11i}, we can decompose $B^{2}$ as:
\begin{equation*}
B_{st}^{2}=\sum_{e=1}^{5}\sum_{t_{k}=s}^{t-} f_{s}  \otimes \delta R_{t_{k}t_{k+1}}^{e}
+\sum_{e=1}^{5} \sum_{t_{k}=s}^{t-} \delta f_{s t_{k}}  \otimes\delta R_{t_{k}t_{k+1}}^{e}.
\end{equation*}
Then recall the assumption $f \in C^{\ga'} $ for any $\frac13<\ga'<H $. By choosing $\ga'$ and $\kappa$ such that $ \ga'+1-\ga-2\kappa >1$, we can apply Proposition \ref{prop3.4}. Taking into account \eref{e6.5i} this yields:
\begin{eqnarray}\label{e6.3}
\left| B_{st}^{2}\right| &\leq& G(t-s)^{1-\ga}n^{1-3\ga},
\end{eqnarray}
for $(s,t) \in \cs_{2}(\ll 0,T \rr)$. As far as the term $B^{1}$ above is concerned, we get: 
\begin{equation}\label{e5.19}
\left| B_{st}^{1} \right|
\leq 
\sum_{t_{k}=s}^{t-}\Big| \sum_{e=1}^{3} \int_{t_{k}}^{t_{k+1}} \delta f_{ t_{k}u}  \otimes dA_{u}^{e} \Big|.
\end{equation}
Moreover, note that $\int_{t_{k}}^{t_{k+1}} \delta f_{ t_{k}u}  \otimes dA_{u}^{e}$ is a $3$rd-order integral of the process $(f,B)$ on the interval $[t_{k},t_{k+1}]$, for all $k $. Therefore, since we have assumed $\|S_{2}(f, B)\|_{\ga} \leq G $, we easily obtain the estimate $ |\int_{t_{k}}^{t_{k+1}} \delta f_{ t_{k}u}  \otimes dA_{u}^{e}| \leq Gn^{ -3\ga} $. Applying this inequality to \eref{e5.19} yields:
\begin{equation}\label{e6.4}
\left| B_{st}^{1} \right|
\le G(t-s)n^{1-3\ga}.
\end{equation}
The lemma follows by applying \eref{e6.3} and \eref{e6.4} to \eref{e5.5}.
\end{proof}

We now wish, as in the case of $y^{n}$, to consider the error process $\ep$ as a rough path. As a first step, let us label the following regularity assumption for further use:

\begin{hyp}\label{lem:bnd-ep-holder-1/2}
Let $\ep$ be the process defined by \eqref{eq:def-ep-t}. 
We suppose that there exists an exponent $\al<2H-\frac12$ and an almost surely finite random variable $G$ such that the error process $\ep$ satisfies:
\begin{equation}\label{e6.14}
| \delta \ep_{st}| 
\leq 
\frac{G}{n^{\al}} \, (t-s)^{\frac12} , \quad (s,t) \in \cs_{2}(\ll 0,T \rr).
\end{equation}
\end{hyp}

%\begin{proof}
%Set $\al=2H-\frac12-\ka$ for a small value of $\ka$. We just apply inequality \eqref{e6.7} with $3\ga-1=\al$. Then it is readily checked that the exponent $1-\ga$ in \eqref{e6.7} becomes:
%\begin{equation*}
%\beta \equiv \frac56 - \frac{2H}{3} - \ka,
%\end{equation*}
%and since $H<\frac12$ we get $\beta>\frac12-\ep$. This yields our claim \eqref{e6.14}.
%\end{proof}

\begin{remark}
It follows from Lemma \ref{prop6.1}   that Hypothesis \ref{lem:bnd-ep-holder-1/2} holds true for $\al =  3\ga -1$. We will see later on how to improve it to larger  values of $\al$. 
\end{remark}

 We are now ready to define and estimate the double iterated integrals of $\ep$, which are a fundamental part of the rough path above $\ep$. 

\begin{lemma} \label{prop6.3}
Let $\ep$ be the process defined by \eqref{eq:def-ep-t} and assume that Hypothesis \ref{lem:bnd-ep-holder-1/2} is satisfied for some $\al<2H-\frac12$.
Then for any $\ka >0$ we have
\begin{eqnarray}\label{eq5.21}
\Big| \int_{s}^{t} 
 \delta \ep_{ s u } \otimes
 d\ep_{u} \Big| &\leq& G  n^{-2\al+ 2\ka}(t-s), \quad (s, t)\in \cs_{2}(\ll 0,T\rr),
\end{eqnarray}
where $G$ is a random variable such that $G \in \cap_{p\geq 1} L^{p}(\Omega)$. 
\end{lemma}

\begin{proof}
Observe that the double integral $\int_{s}^{t}  \delta \ep_{ s u } \otimes d\ep_{u}$ is well defined, since $(y,y^{n},B)$ admits a rough path lift.
Next take $(s, t)\in \cs_{2}(\ll 0,T\rr)$.  We can write:
\begin{eqnarray}\label{e5.20}
\int_{s}^{t} \delta \ep_{su} \otimes d\ep_{u} 
&=&   
\int_{s}^{t}  \delta \ep_{s\eta(u)} \otimes
 d\ep_{u} +   \int_{s}^{t} 
 \delta \ep_{ \eta(u) u } \otimes
 d\ep_{u} 
 \equiv D_{st}^{1} + D_{st}^{2},
\end{eqnarray}
where recall that $\eta(u)$ is defined in \eref{eq:def-eta-A1}.
Let us bound those two terms separately.

The term $D^{1}$ above can be expressed in a more elementary way as $D_{st}^{1}=
\sum_{t_{k}=s}^{t-} \delta \ep_{st_{k}}\otimes \delta \ep_{t_{k}t_{k+1}}$. Moreover, thanks to \eref{e6.14} we   have for any $\ka < \al$ and   $(s,t) \in \mathcal{S}_{2}(\ll0,T\rr)$:
\begin{eqnarray*}
|\delta \ep_{st}| &\leq& G n^{-\al+\ka} (t-s)^{\frac12+\ka}.
\end{eqnarray*}
Taking this estimate into account and applying Proposition \ref{prop3.4} we get:
\begin{eqnarray}\label{e6.8}
\left|   D^{1}_{st} \right| &\leq & G^{2} n^{-2\al+2\ka}(t-s)^{1+2\ka} .
\end{eqnarray}

On the other hand, owing to identity \eref{e6.4i} for $\ep$, we have
\begin{eqnarray}\label{e5.22}
 D^{2}_{st}  &=& 
 \sum_{e,e'=1}^{3} \int_{s}^{t} \int_{\eta(u)}^{u} \Psi^{n}_{v} dA^{e'}_{v}    \otimes \Psi^{n}_{u} dA^{e}_{u} .
\end{eqnarray}
By  the definition of $A_{e}$, $e=1,2,3$ in  \eref{eq:def-eta-A1} and \eref{e6.4ii}, it is easy to see  that each of the nine terms on the right-hand 
side is bounded by $Gn^{1-4\ga}(t-s)$. Indeed, for the term corresponding to   $e=e'=3$, we use  \eref{e6.4ii} to write 
\begin{eqnarray*}
 \int_{s}^{t} \int_{\eta(u)}^{u} \Psi^{n}_{v} dA^{3}_{v}    \otimes \Psi^{n}_{u} dA^{3}_{u} 
&=&  \sum_{j,j'=1}^{m} \sum_{t_{k}=s}^{t-} A^{33,jj'}_{t_{k}t_{k+1}} \,,
\end{eqnarray*}
where 
\begin{eqnarray*}
A^{33,jj'}_{t_{k}t_{k+1}} &=& \int_{t_{k}}^{t_{k+1}} \int_{t_{k}}^{u} \Psi^{n}_{v}  
  \int_{t_{k}}^{v} \partial  {V}_{j}(y^{n}_{r}) dy^{n }_{r} d B^{j}_{v}
   \otimes \Psi^{n}_{u}     \int_{t_{k}}^{u} \partial  {V}_{j'}(y^{n}_{r'}) dy^{n }_{r'} d B^{j'}_{u}.
\end{eqnarray*}
It follows from  Proposition \ref{prop:bnd-Phi-n}   and the     Lyon's lift map theorem (see e.g.   \cite{FV})  that
\begin{eqnarray*}
\left| S_{4}(B, y^{n}, \Psi^{n})_{t_{k},t_{k+1}} \right| &\leq&  {G} \left(\frac Tn\right)^{\ga},
\end{eqnarray*}
so for all $j,j'\leq m$ and $(s,t) \in \cs_{2}(\ll0,T\rr)$ we have
\begin{eqnarray*}
|A^{33,jj'}_{t_{k}t_{k+1}}| &\leq& G\left(\frac Tn\right)^{4\ga}.
\end{eqnarray*}
Therefore, summing this bound  over $j$, $j'$ and $t_{k} $ we end up with 
\begin{eqnarray}\label{eq5.28}
\Big| \int_{s}^{t} \int_{\eta(u)}^{u} \Psi^{n}_{v} dA^{3}_{v}    \otimes \Psi^{n}_{u} dA^{3}_{u}
\Big|
 \leq
   \sum_{j,j'=1}^{m} \sum_{t_{k}=s}^{t-}  |A^{33,jj'}_{t_{k}t_{k+1}}| 
   \leq
    G\Big(\frac Tn\Big)^{4\ga-1}(t-s).
\end{eqnarray}

The other terms of the form $\int_{s}^{t} \int_{\eta(u)}^{u} \Psi^{n}_{v} dA^{e}_{v}    \otimes \Psi^{n}_{u} dA^{e'}_{u} $ on the right-hand side of \eref{e5.22} can be estimated in the same way. 
Therefore, we obtain the estimate
\begin{eqnarray}\label{e6.9}
\left| D^{2}_{st} \right| &\leq& Gn^{1-4\ga}(t-s).
\end{eqnarray}
In conclusion, plugging \eref{e6.8} and \eref{e6.9} into \eref{e5.20}, we obtain the desired estimate  \eref{eq5.21}.
\end{proof}

Recall that we wish to construct a rough path above $(\ep, B, y^{n})$. In the previous lemma we have analyzed the double integral $\int_{s}^{t}\delta \ep_{su} \otimes \delta \ep_{u}$. We now consider the  integral  $\int_{s}^{t} \delta \ep_{su} \otimes d ( B_{u}, y^{n}_{u} )$. 
\begin{lemma} \label{lem6.4}
Denote $z^{n} =  ( B, y^{n} )$, and let the assumptions of Lemma \ref{prop6.1} prevail.  The  following estimate holds true
\begin{eqnarray*}
\Big| \int_{s}^{t} \delta \ep_{su}\otimes dz^{n}_{u} \Big| &\leq&  G(t-s) n^{1-3\ga} ,
\quad\quad (s,t )\in \mathcal{S}_{2}(\ll 0,T \rr).
\end{eqnarray*}
\end{lemma}
\begin{proof}
We use the same kind of decomposition as in Lemma \ref{prop6.3}:
\begin{eqnarray}\label{e6.10}
\int_{s}^{t} \delta \ep_{su} \otimes dz^{n}_{u} &=&   \int_{s}^{t} 
 \delta \ep_{s\eta(u)} \otimes
 dz^{n}_{u} +   \int_{s}^{t} 
 \delta \ep_{ \eta(u) u } \otimes
 dz^{n}_{u} := \hat{D}^{1}_{st} + \hat{D}^{2}_{st} .
\end{eqnarray}
For the first term on the right-hand side of \eref{e6.10}  we have the following   expression:
\begin{eqnarray*}
  \hat{D}^{1}_{st}
 &=&
    \sum_{t_{k}=s}^{t-} \delta \ep_{st_{k}} \otimes \delta z^{n}_{t_{k}t_{k+1}}\,, \quad\quad  (s,t )\in \mathcal{S}_{2}(\ll 0,T \rr). 
\end{eqnarray*}
Moreover, it follows from  Lemma \ref{prop6.1}  and Proposition \ref{prop5.2}  that for all $(s,t) \in \cs_{2}(\ll0,T\rr)$ the following bounds holds true: 
\begin{eqnarray*}
|\delta \ep_{st}| \leq G(t-s)^{1-\ga }n^{1-3\ga} \quad \text{and} \quad 
|\delta z_{st}| \leq G (t-s)^{\ga+\ka},
\end{eqnarray*}
where $\ka < H-\ga$, so by Proposition \ref{prop3.4}
 we have  
\begin{eqnarray*}
\Big|  \hat{D}^{1}_{st} \Big| &\leq& G(t-s) n^{1-3\ga} .
\end{eqnarray*}
On the other hand, in the same way as the estimate of $D^{2}_{st}$ in \eref{e6.9}, we can  show that 
\begin{eqnarray*}
\Big|  \hat{D}^{2}_{st}  \Big| &\leq& G(t-s) n^{1-3\ga}.
\end{eqnarray*}
The lemma   follows from applying the above two inequalities for $\hat{D}^{1}_{st}  $ and $\hat{D}^{2}_{st}  $ to \eref{e6.10}. 
 \end{proof}

 The next result provides further estimates of the rough path above $(\ep, B, y^{n})$   for $(s,t)   \in \cs_{2}( [t_{k},t_{k+1}])$. 
\begin{lemma}\label{lem6.5}
Let $\ep$ be the process defined by \eqref{eq:def-ep-t} and recall that we have set $ z^{n} =  ( B, y^{n} ) $. Take $(s,t)   \in \cs_{2}( [t_{k},t_{k+1}])$ and $k=0,1,\dots, n-1$.  Then the following estimate for the first-order increments of $\ep$ holds true:
\begin{eqnarray}
|\delta \ep_{st}| &\leq& G (t-s)^{\ga} n^{-\ga}.
\label{e6.18}
\end{eqnarray}
In addition,   the second-order iterated integrals of $\ep$ and $z^{n}$ satisfy:
\begin{eqnarray}\label{e6.19}
\Big| \int_{s}^{t} \delta z^{n}_{su} \otimes d \ep_{u} \Big|  \leq  G (t-s)^{2\ga} n^{-\ga}. 
\quad\quad\quad
\Big| \int_{s}^{t} \delta\ep_{su} \otimes d \ep_{u} \Big|  \leq  G (t-s)^{2\ga} n^{-2 \ga}.  
\end{eqnarray}
\end{lemma}
 \begin{proof} 
The estimate \eref{e6.18} follows by showing that the three terms on the right-hand side of  \eref{e6.4i} are all bounded by $G (t-s)^{\ga} n^{-\ga}$. As before, we will focus on a bound for the increment $\int_{s}^{t} \Psi^{n}_{s} dA^{3}_{s} $. In fact, owing to  \eref{e6.4ii} it is easily seen that $\int \Psi^{n} dA^{3}$ can be decomposed as a sum of double iterated integrals: 
\begin{eqnarray*}
\int_{s}^{t}\Psi^{n}_{s} dA^{3}_{s} &=&  \sum_{j=1}^{m}  \int_{s}^{t} \Psi^{n}_{u} \int_{t_{k}}^{u}\partial V_{j} (y^{n}_{v}) dy^{n}_{v} dB^{j}_{u}
\\
&=& 
  \sum_{j=1}^{m}  \int_{s}^{t} \Psi^{n}_{u} \int_{s}^{u}\partial V_{j} (y^{n}_{v}) dy^{n}_{v} dB^{j}_{u}
  +  \sum_{j=1}^{m} \int_{t_{k}}^{s}\partial V_{j} (y^{n}_{v}) dy^{n}_{v}   \int_{s}^{t} \Psi^{n}_{u} dB^{j}_{u} 
 .
\end{eqnarray*}
One can easily bound the two terms above, thanks to the fact that $(y^{n}, B)$ is a rough path. We obtain 
\begin{eqnarray*}
\int_{s}^{t}\Psi^{n}_{s} dA^{3}_{s}&\leq & G (t-s)^{\ga}n^{-\ga}.
\end{eqnarray*}
In the same way we can show that the same estimate holds for the term $\sum_{e=1}^{2}\int_{s}^{t} \Psi^{n}_{s} dA^{e}_{s} $ on the right-hand side of \eref{e6.4i}. This proves our claim \eref{e6.18}.

In order to prove \eref{e6.19}, let us invoke   \eref{e6.4i} again, which yields:
\begin{eqnarray*}
 \int_{s}^{t} \delta z^{n}_{su} \otimes d \ep_{u}  &=& \sum_{e=1}^{3}     \int_{s}^{t}\delta z^{n}_{su} \otimes \Psi^{n}_{u}  dA_{u}^{e} .
\end{eqnarray*}
The estimate \eref{e6.19} then follows from a similar argument as for the estimate of \eref{e6.18}. The estimate of integral $ \int_{s}^{t} \ep_{su} \otimes d \ep_{u}$ can be shown in a similar way.    This completes the proof.
\end{proof}

Following is the main result of this section. Recall that  $S_{2} (n^{3\ga - 1} \ep , z^{n})$ denotes the lift of the process $(n^{3\ga - 1} \ep , z^{n})$, that is,
\begin{eqnarray*}
S_{2} (n^{3\ga - 1} \ep , z^{n})_{st} &=& \Big(  (n^{3\ga - 1} \delta\ep_{st} , \delta z^{n}_{st} ) , \int_{s}^{t}  (n^{3\ga - 1} \delta \ep_{su} , \delta z^{n}_{su} )  \otimes (n^{3\ga - 1} \ep_{u}  , z^{n}_{u}) \Big).
\end{eqnarray*}

\begin{prop}\label{prop5.3}
Let $y$ be the solution of equation \eref{e1.1} and $y^{n}$ be the solution of the   Euler scheme \eref{e1.2}. Take $\frac13 <\ga<H<\frac12$. Then we have the estimate   
\begin{eqnarray*}
 \| S_{2} (  n^{3\ga-1}\ep \,, z^{n}) \|_{\ga} &\leq&   G  \,,
\end{eqnarray*}
where $G$ is a random variable  independent of $n$.
\end{prop}
\begin{proof}
In summary of   Lemma \ref{prop6.1}, \ref{prop6.3}, \ref{lem6.4} and Proposition \ref{prop5.2}, we have 
\begin{eqnarray}\label{eq6.33}
|S_{2} (  n^{3\ga-1}\ep \,, z^{n})_{st}| &\leq& G (t-s)^{\ga}
\end{eqnarray}
for $(s, t)\in \mathcal{S}_{2}(\ll 0,T \rr)$.  
On the other hand,   Lemma \ref{lem6.5} implies that   relation \eref{eq6.33} still holds true  
for $s,t \in [t_{k},t_{k+1}]$. The lemma then follows by applying Lemma \ref{lem4.5} to $S_{2} (  n^{3\ga-1}\ep \,, z^{n}) $.
\end{proof}

\section{The error process as a rough path under new conditions}\label{section6}
In this section, we derive an improved upper-bound estimate of the error process under new conditions. 
Our considerations   relies on the following process,   solution of a linear RDE:
\begin{eqnarray}\label{e6.1 i}
\Phi_{t} &=& \id+ \int_{0}^{t} \partial b(y_{s})\Phi_{s}ds + \sum_{j=1}^{m}\int_{0}^{t} \partial V_{j}(y_{s}) \Phi_{s}dB^{j}_{s}\,.
\end{eqnarray}
Throughout the section,  we assume that $b \in C^{2}_{b}$ and $V\in C^{4}_{b}$. 
The reader might have noticed that $\Phi$ is simply the Jacobian related to equation \eref{e1.1}.
The process $\Phi$ is also the limit of the process $\Phi^{n}$ defined in \eref{e6.1i}, in a sense which    will be made   clear in the next section. We denote by $\Psi_{t}$ the inverse matrix of $\Phi_{t}$.  
Let us   introduce the following process on $S_{2}(\ll0,T\rr)$:
\begin{eqnarray}\label{e7.1}
\delta \wt{\ep}_{st} &=& \delta \ep_{st} -  \delta \hat{\ep}_{st} \,,
\end{eqnarray}
where $\ep$ is defined by \eref{eq:def-ep-t} and
 \begin{eqnarray}\label{e6.2i}
\delta \hat{\ep}_{st} &=& \sum_{j,j'=1}^{m}  \sum_{t_{k}=s}^{t-}  \Psi_{t_{k}} \partial {V}_{j}{V}_{j'} (y_{t_{k}}) \delta F^{jj'}_{t_{k}t_{k+1}} .
\end{eqnarray}
We shall now assume some a priori bounds on $\tilde{\ep}$, similarly to what we did in Hypothesis \ref{lem:bnd-ep-holder-1/2}. 
\begin{hyp}\label{condition6.1} 
The process $\wt{\ep}$ defined in \eref{e7.1} satisfies the following relation  for some $\al>0$:
\begin{eqnarray*} 
\left| \delta \wt{\ep}_{st}    \right| &\leq & G n^{-\al} (t-s)^{1-\ga}, \quad\quad (s,t )\in \mathcal{S}_{2}(\ll 0,T \rr).
\end{eqnarray*}  
\end{hyp}
Our aim is to get a new bound on the rough path above $(\ep, z^{n})$ under Hypothesis~\ref{condition6.1}, similarly to what has been obtained in Proposition \ref{prop5.3}.
Let us first consider  $\int_{s}^{t} \delta \ep_{su}\otimes dB_{u}$.  
\begin{lemma} \label{lem7.2}
 Suppose that Hypothesis \ref{condition6.1} is met for some  $\al: 0<\al< 2H-\frac12$. Then the following estimate holds true
 for all $(s, t)\in \cs_{2}(\ll 0,T\rr)$:
\begin{eqnarray}\label{eq6.5}
\Big| \int_{s}^{t} \delta \ep_{su}\otimes dB_{u} \Big| &\leq& G (n^{-\al}+n^{\frac12 - 2\ga} ) (t-s)^{2\ga}.
\end{eqnarray}
\end{lemma}
\begin{proof}  As in Lemma \ref{lem6.4}, we first write
\begin{eqnarray} \label{e6.17ii}
\int_{s}^{t} \delta \ep_{su} \otimes dB_{u} =   \int_{s}^{t} 
 \delta \ep_{s\eta(u)} \otimes
 dB_{u} +   \int_{s}^{t} 
 \delta \ep_{ \eta(u) u } \otimes
 dB_{u} := \tilde{D}^{1}_{st} +\tilde{D}^{2}_{st}.
\end{eqnarray}
Furthermore, invoking decomposition \eref{e6.4i}  we have
\begin{eqnarray}\label{eq6.4}
 \tilde{D}^{2}_{st} 
 = \sum_{e=1}^{3} \int_{s}^{t} \int_{\eta(u)}^{u} \Psi^{n}_{v} d A^{e}_{v}  \otimes dB_{u}
  := \sum_{e=1}^{3}I^{e}_{st}   .
  \end{eqnarray}
  Note that, recalling expression \eref{e6.4ii} for $A^{2}$, $I_{2}$ can be seen as a triple iterated integral which is interpreted in the Young sense. Then similar to \eref{eq5.28}   it is easy to show that
  \begin{eqnarray}\label{e 6.4}
 | I^{2}_{st}  | &\leq& G n^{-2\ga} (t-s).
\end{eqnarray}
As far as $I_{1}$ is concerned,
some elementary computations reveal that for $(s,t) \in \cs_{2} (\ll 0,T \rr)$:
\begin{eqnarray*}
I^{1}_{st} &=& -\frac12 \sum_{j=1}^{m} \sum_{t_{k}=s}^{t-} \partial V^{i}_{j}V_{j} (y^{n}_{t_{k}}) \int_{t_{k}}^{t_{k+1}} \int_{t_{k}}^{u} \Psi^{n, i}_{v} d  (v-t_{k})^{2H} \otimes dB_{u}.
\end{eqnarray*}
Then
 it follows from \eref{e11.1ii} in Lemma \ref{lem6.11}  that
\begin{eqnarray}\label{e 6.5}
|I^{1}_{st}  |
 \leq  G n^{1-4\ga} (t-s)^{1-\ga}
 \leq  G n^{-\ga} (t-s)^{2\ga} ,
\end{eqnarray}
where we use the fact that $t-s\geq \frac Tn$ for the second inequality. 
For $I_{3}$, we start from relation~\eref{e6.4ii}. Then due to 
the expression \eref{e1.2} for $y^{n}$, we have 
\begin{eqnarray*}
  I^{3}_{st}     & = &  
 I^{31}_{st}   + I^{32}_{st}  + I^{33}_{st}   ,
\end{eqnarray*}
where
\begin{eqnarray*}
I^{31}_{st} &=& \sum_{t_{k} = s}^{t-} \int_{t_{k}}^{t_{k+1}} \int_{t_{k}}^{u} \Psi^{n}_{v} \sum_{j=1}^{m} \int_{t_{k}}^{v} \partial V_{j} (y^{n}_{r}) V(y^{n}_{t_{k}}) dB_{r} dB^{j}_{v}  \otimes d B_{u} ,
\\
I^{32}_{st} &=&  \sum_{t_{k} = s}^{t-} \int_{t_{k}}^{t_{k+1}} \int_{t_{k}}^{u} \Psi^{n}_{v} \sum_{j=1}^{m} \int_{t_{k}}^{v} \partial V_{j} (y^{n}_{r}) \frac12 \sum_{j'=1}^{m} \partial V_{j'}V_{j'}(y^{n}_{t_{k}}) d(r-t_{k})^{2H} dB^{j}_{v}  \otimes d B_{u}
\\
I^{33}_{st} &=& \sum_{t_{k} = s}^{t-} \int_{t_{k}}^{t_{k+1}} \int_{t_{k}}^{u} \Psi^{n}_{v}  \sum_{j=1}^{m} \int_{t_{k}}^{v} \partial V_{j} (y^{n}_{r}) b(y^{n}_{t_{k}}) d {r} dB^{j}_{v}  \otimes d B_{u}\,. 
\end{eqnarray*}
As for $I^{2}$, it is easy to show that  $|I^{32}_{st}|$ and $|I^{33}_{st}|$ are bounded by $Gn^{-2\ga} (t-s)$ and $G n^{1-4\ga} (t-s)$, respectively. On the other hand,  it follows from \eref{e11.1 i} in Lemma  \ref{lem6.11} that
for any $\kappa >0$ we can find a random variable $G$ such that
$|I^{31}_{st}|$ is less than $Gn^{1-4\ga+2\kappa} (t-s)^{1-\ga}$. We now choose $\kappa >0$ small enough such that $1-4\ga+2\kappa<-\ga$. Then,  summarizing our estimates of $I^{31}$, $I^{32}$ and $I^{33}$, we have
\begin{eqnarray}\label{e6.6}
| I^{3}_{st} |  \leq  Gn^{1-4\ga+2\kappa} (t-s)^{1-\ga}
 \leq  G n^{-\ga} (t-s)^{2\ga} .
\end{eqnarray}
Applying \eref{e 6.4}, \eref{e 6.5} and \eref{e6.6}  to \eref{eq6.4} we have thus obtained:
\begin{eqnarray}\label{e6.7ii}
\left|  \tilde{D}^{2}_{s t} \right|
 &\leq & G n^{-\ga} (t-s)^{2\ga} .
\end{eqnarray}

We now turn to the   term $\tilde{D}^{1}_{st}$ on the right-hand side of \eref{e6.17ii}.  
Write 
\begin{eqnarray*}
\tilde{D}^{1}_{st} 
=
\int_{s}^{t} ( \delta  {\ep}_{s\eta(u)} -  \delta \wt{\ep}_{s\eta(u)} ) \otimes dB_{u} 
+ \int_{s}^{t} \delta \wt{\ep}_{s\eta(u)} \otimes dB_{u}
:=
 \tilde{D}^{11}_{st} + \tilde{D}^{12}_{st}   ,
\end{eqnarray*}
where we recall that $\wt{\ep}$ is defined in \eref{e7.1}. 
Applying Proposition \ref{prop3.4} and taking into account   Hypothesis \ref{condition6.1} and the fact that $H>\ga$, we obtain
\begin{eqnarray}\label{e6.4 i}
\Big|  \tilde{D}^{12}_{st}  \Big| &\leq& G n^{-\al} (t-s).
\end{eqnarray}
On the other hand, by \eref{e7.1} and \eref{e6.2i} we have
\begin{eqnarray*}
  \tilde{D}^{11}_{st}   
&=& \int_{s}^{t} 
 \sum_{j,j'=1}^{m}   \int_{s}^{\eta(u)} \Psi_{\eta(v)} \partial {V}_{j}{V}_{j'} (y_{\eta(v)}) dF^{jj'}_{\eta(v) v} \otimes dB_{u}
 \\
 &=& \sum_{j,j'=1}^{m}   \sum_{t_{k}=s+}^{t-} \sum_{t_{k'} = s }^{t_{k-1}}  
 \Psi_{t_{k'}} \partial {V}_{j}{V}_{j'} (y_{ t_{k' }}) 
      F^{jj'}_{t_{k'}t_{k' +1 } } \otimes \delta B_{t_{k}t_{k+1}}.
\end{eqnarray*}
So applying Corollary \ref{lem9.1} to $\tilde{D}^{11}_{st}$ we obtain
\begin{eqnarray*} 
\Big\| \tilde{D}^{11}_{st} 
 \Big\|_{p} &\leq& K n^{\frac12 - 2H} (t-s)^{H+\frac12}, \quad \text{ for all  } p\geq 1. 
\end{eqnarray*}
Taking into account Lemma \ref{lem4.2},  we thus get:
\begin{eqnarray}\label{e6.5ii}
\Big| \tilde{D}^{11}_{st} 
 \Big|  &\leq& G n^{\frac12 - 2\ga} (t-s)^{2\ga}.
\end{eqnarray}
In summary of \eref{e6.7ii}, \eref{e6.4 i} and \eref{e6.5ii}, we end up with: 
\begin{eqnarray*}
\Big| \int_{s}^{t} \delta \ep_{su}\otimes dB_{u} \Big| &\leq&  
 G n^{-\ga} (t-s)^{2\ga} 
 + Gn^{-\al}(t-s) + G n^{\frac12 -2\ga} (t-s)^{2\ga}
, 
\end{eqnarray*}
from which our claim \eref{eq6.5} is easily deduced. 
\end{proof}

In order to complete the study of the rough path above $(\ep, y^{n}, B)$, let us 
  turn to the integral $\int_{s}^{t} \delta y^{n}_{su}   \otimes d\ep_{u} $\,.
\begin{lemma}\label{lem7.4}
Suppose that Hypothesis \ref{lem:bnd-ep-holder-1/2} and \ref{condition6.1} are met for some  $\al\in( 0, 2H-\frac12)$.  
Then the integral $\int_{s}^{t} \delta y^{n}_{su}   \otimes d\ep_{u} $ satisfies the following relation for all $(s,t)\in\cs_{2}(\ll 0,T\rr)$:
\begin{eqnarray}\label{eq:6.12ii}
\Big| \int_{s}^{t} \delta y^{n}_{su}   \otimes d\ep_{u}   \Big| &\leq&  Gn^{-\al} (t-s)^{2\ga}.
\end{eqnarray}
\end{lemma}

\begin{proof} 
We consider a remainder term $R$ defined for $(s,t)\in\cs_{2}(\ll 0,T\rr)$ by:
\begin{eqnarray}\label{e6.12i}
R_{st} &=& \int_{s}^{t} \delta y^{n}_{su}   \otimes d\ep_{u}  - V(y^{n}_{s}) \int_{s}^{t} \delta B_{su} \otimes d \ep_{u}. 
\end{eqnarray}
According to the basic rules of action of $\delta$ on products of increments, we have:
\begin{eqnarray}\label{e6.15i}
\delta R_{sut} &=& ( \delta y^{n}_{su}   - V(y^{n}_{s}) \delta B_{su}) \otimes \delta \ep_{ut} + \delta V(y^{n}_{\cdot})_{su} \int_{u}^{t} \delta B_{uv} \otimes d\ep_{v},
\end{eqnarray}
for all $(s,u,t)\in\cs_{3}(\ll 0,T\rr)$.
  Applying the second inequality \eref{eq:bnd-delta-yn} and Hypothesis \ref{lem:bnd-ep-holder-1/2} to the first term on the right-hand side of \eref{e6.15i} and invoking Lemma \ref{lem7.2} for the the second term,  we obtain:
\begin{eqnarray*}
|\delta R_{sut}| \leq G n^{-\al} (u-s)^{2\ga} (t-u)^{\frac12} + Gn^{-\al} (u-s)^{\ga} (t-u)^{2\ga}
 \leq
Gn^{-\al}(t-s)^{3\ga}.
\end{eqnarray*}
Since $3\ga>1$, we are in a position to apply the discrete sewing Lemma \ref{lem2.4}, which yields:
\begin{eqnarray}\label{e6.12}
\|R\|_{3\ga} \leq K\| \delta R \|_{3\ga} \leq G n^{-\al}.
\end{eqnarray}
We now recast \eqref{e6.12i} as follows:
\begin{equation}\label{eq7.17}
 \int_{s}^{t} \delta y^{n}_{su}   \otimes d\ep_{u}
 =
V(y^{n}_{s}) \int_{s}^{t} \delta B_{su} \otimes d \ep_{u}
+
R_{st} .
\end{equation}
Then resorting to \eref{e6.12} and Lemma \ref{lem7.2}  for \eref{eq7.17}, our claim \eqref{eq:6.12ii} easily follows.
\end{proof}

We can now state the main result of this section, giving a full estimate of the rough path above $(z^{n} , n^{\al-\kappa} \ep)$. Recall that $z^{n}$ designates the couple $(B, y^{n})$.
\begin{prop}\label{prop7.1}
Let $\wt{\ep}$ be defined in \eref{e7.1} and suppose that  Hypothesis   \ref{condition6.1} hold true  for some  $ 0<\al< 2H-\frac12$.  
Then for any $\kappa \in (0,\al)$ and $(s,t)\in\cs_{2}([0,T])$ we have:
\begin{eqnarray}\label{e6.17i}
|S_{2}(z^{n} , n^{\al-\kappa} \ep)_{st}| &\leq& G(t-s)^{\ga}.
\end{eqnarray}
\end{prop}

\begin{proof} 

We start by analyzing the first-order increments of $S_{2}(z^{n} , n^{\al-\kappa} \ep)$. First notice that $\delta z^{n}$ is controlled by Proposition~\ref{prop5.2}. Furthermore, 
according to relation \eref{e7.1}, we have 
\begin{eqnarray}\label{eq6.19}
\delta \ep &=& \delta \hat{\ep} +\delta \tilde{\ep}.
\end{eqnarray}
As in the proof of Lemma \ref{lem7.2}, equation \eref{e6.2i} also asserts that Corollary \ref{lem9.1} can be applied to $\delta \hat{\ep}$, yielding an inequality of the form
\begin{eqnarray}\label{eq6.20}
\| \delta \hat{\ep}_{st} \|_{p} &\leq& Kn^{\frac12 -2H} (t-s)^{\frac12}
\end{eqnarray}
for $(s,t) \in \cs_{2} (\ll0,T\rr)$. Applying Lemma \ref{lem4.2} to relation \eref{eq6.20}, plugging this information into \eref{eq6.19} and invoking Hypothesis   \ref{condition6.1}, 
 we obtain:
\begin{eqnarray}\label{e6.17}
| \delta \ep_{st}| \leq  Gn^{-\al}(t-s)^{\frac12} \leq Gn^{-\al+\kappa} (t-s)^{\frac12 +\kappa},
\end{eqnarray}
for all $(s,t)\in\cs_{2}(\ll 0,T\rr)$. This is compatible with our claim \eqref{e6.17i}.

Let us now handle the 2nd-order increments of $S_{2}(z^{n} , n^{\al-\kappa} \ep)$. According to Lemma~\ref{prop6.3}, 
 \begin{eqnarray}\label{e6.14i}
\Big| \int_{s}^{t} 
 \delta \ep_{ s u } \otimes
 d\ep_{u} \Big| &\leq& G  n^{-2\al+2\kappa}(t-s) ,
\quad\quad
   (s,t)\in\cs_{2}(\ll 0,T\rr).
\end{eqnarray}
In the same way, gathering  Lemma \ref{lem7.2}, Lemma \ref{lem7.4} together with \eref{e6.17}, we get that   
\begin{eqnarray}\label{e6.15}
\Big| \int_{s}^{t} \delta \ep_{su}\otimes dz^{n}_{u} \Big| &\leq& G  n^{-\al+\kappa}  (t-s)^{2\ga} ,
\quad \quad (s,t)\in\cs_{2}(\ll 0,T\rr). 
\end{eqnarray}
Hence, putting together inequalities \eref{e6.14i} and \eref{e6.15} and adding the estimate of $S_{2}(z^{n})$ in Proposition \ref{prop5.2}, we obtain that on the grid $\cs_{2}(\ll 0,T\rr)$:
 \begin{eqnarray}\label{e6.15ii}
|S_{2} (z^{n}, n^{\al-\kappa} \ep)_{st}| &\leq& G (t-s)^{\ga}
\end{eqnarray}
On the other hand,    Lemma \ref{lem6.5} implies that  \eqref{e6.15ii} also holds true for $s,t \in [t_{k},t_{k+1}]$. Therefore,  applying  Lemma \ref{lem4.5} to $S_{2} (z^{n}, n^{\al-\kappa} \ep)$ we obtain the desired estimate \eqref{e6.17i}.
\end{proof}

\section{Rate of convergence for the   Euler scheme}\label{section7}

In this section, we take another look at the strong convergence of the   Euler scheme. Thanks to the information we have gathered on the error process, we shall reach optimality for the convergence rate of the scheme. However, before we can state this optimal result, let us give some preliminaries about the Jacobian $\Phi$ of equation \eqref{e1.1}.

\subsection{Rate of convergence for the Jacobian}\label{section7.1} 

As mentioned in Section \ref{section6}, the Jacobian $\Phi$ of equation \eref{e1.1} should be seen as the limit of the process $\Phi^{n}$. In the current section we shall quantify this convergence. We start by an algebraic identity which is stated as a lemma.

\begin{lemma}\label{lem:eq-mathcalE}
Recall that $\Phi$ and $\Phi^{n}$ are the solutions of   equations \eref{e6.1 i} and \eref{e6.1i}, respectively. For $t\in[0,T]$, we set 
\begin{equation}\label{eq:def-mathcal-E}
\mathcal{E}_{t} = \Psi_{t} \lp \Phi_{t} -\Phi^{n}_{t} \rp.
\end{equation}
Then $\ce$ satisfies the following equation on $[0,T]$:
\begin{equation}\label{e7.1ii}
\mathcal{E}_{t} = 
\int_{0}^{t} \Psi_{s} \sum_{i=1}^{d}  \{ \partial_{i}\partial b(y^{n}) \}_{s}  (\Phi^{n}_{s} \ep_{s})^{i} \Phi^{n}_s ds 
+ \sum_{j=1}^{m}  \int_{0}^{t} \Psi_{s} \sum_{i=1}^{d} \{ \partial_{i}\partial V_{j}(y^{n}) \}_{s}  
 (\Phi^{n}_{s} \ep_{s})^{i} \Phi^{n}_s dB^{j}_{s}\, ,
\end{equation}
where the processes $\{ \partial_{i}\partial b(y^{n}) \}_{s}$ and $\{ \partial_{i}\partial V_{j}(y^{n}) \}_{s}$ are defined by:
\begin{eqnarray}
\{ \partial_{i}\partial b(y^{n}) \}_{s} &=&  \int_{0}^{1} \int_{0}^{1}   \partial_{i}\partial b ( y_{s}+ (1-\mu)(1-\lambda)(y^{n}_{s}-y_{s})  ) (1-\lambda)d\mu d\lambda, \label{eq:def-ddb-yn}
\\
 \{ \partial_{i}\partial V_{j}(y^{n}) \}_{s} &=&  \int_{0}^{1} \int_{0}^{1}   \partial_{i}\partial V_{j} ( y_{s}+ (1-\mu)(1-\lambda)(y^{n}_{s}-y_{s})  ) (1-\lambda)d\mu d\lambda.  \label{eq:def-ddV-yn}
\end{eqnarray}
If we define $\wt{\mathcal{E}}_{t} = \Phi_{t} (\Psi_{t} - \Psi^{n}_{t}) $, then a similar expression can be derived for $\wt{\mathcal{E}}_{t} $.  
\end{lemma}

\begin{proof}
Subtracting  \eref{e6.1i} from \eref{e6.1 i}, it is easily seen that:
\begin{eqnarray*}
\Phi_{t} - \Phi^{n}_{t} &=&  \int_{0}^{t} \partial b(y_{s})(\Phi_{s} - \Phi^{n}_s)ds + \sum_{j=1}^{m}\int_{0}^{t} \partial V_{j}(y_{s})(\Phi_{s} - \Phi^{n}_s)dB^{j}_{s}
+L^{1}_{t}  +L^{2}_{t}, 
\end{eqnarray*}
where 
\begin{eqnarray*}
L^{1}_{t} &=&  \int_{0}^{t} \left(\partial b(y_{s})- \{ \partial b(y^{n}) \}_{s}\right) \Phi^{n}_s ds,
\\
L^{2}_{t} &=& \sum_{j=1}^{m}\int_{0}^{t} \left( \partial V_{j}(y_{s})  -\{ \partial V_{j}(y^{n}) \}_{s} \right)\Phi^{n}_s dB^{j}_{s}.
\end{eqnarray*}
Let  now $\Psi = \Phi^{-1}$ be the inverse of $\Phi$. By means of the variation of the constant method, one can verify that
\begin{eqnarray*}
\Phi_{t} - \Phi^{n}_{t} &=&\sum_{e=1,2} \Phi_{t} \int_{0}^{t} \Psi_{s} dL^{e}_{s}\,.
\end{eqnarray*}
Hence, for $\ce$ defined by \eqref{eq:def-mathcal-E},  we have
\begin{eqnarray}\label{e7.1 i}
\mathcal{E}_{t}  &=&\sum_{e=1,2}   \int_{0}^{t} \Psi_{s} dL^{e}_{s}\,.
\end{eqnarray}
In addition, observe that with \eqref{eq:def-ddb-yn} and \eqref{eq:def-ddV-yn} in mind, the following identities hold true:
\begin{equation*}
\partial b(y_{s})  -\{ \partial b(y^{n}) \}_{s}
=  
\sum_{i=1}^{d} \{ \partial_{i}\partial b(y^{n}) \}_{s}  (y^{i}_{s}-y^{n,i}_{s})
=  
\sum_{i=1}^{d} \{ \partial_{i}\partial b(y^{n}) \}_{s}  (\Psi^{n}_{s}\ep_{s})^{i},
\end{equation*}
and
\begin{equation*} 
\partial V_{j}(y_{s})  -\{ \partial V_{j}(y^{n}) \}_{s}  
=  
\sum_{i=1}^{d} \{ \partial_{i}\partial V_{j}(y^{n}) \}_{s}   (\Psi^{n}_{s}\ep_{s})^{i}.
\end{equation*}
Plugging these relations into the definition of $L^{1}$ and $L^{2}$, our claim \eqref{e7.1ii} easily stems from 
relation \eref{e7.1 i}.
\end{proof}

We shall now assume some a priori bounds on the lift of $(z^{n}, n^{\al}\ep)$.
\begin{hyp}\label{hyp7.2}
The processes $z^{n} = (y^{n}, B)$ and   $\ep $   defined in \eref{eq:def-ep-t} satisfy the following inequality for some $\al >0$:
\begin{eqnarray*}
  |S_{2}(z^{n}, n^{\al} \ep)_{st}| &\leq& G   (t-s)^{\ga}, \quad   \quad (s,t) \in \cs_{2}([0,T]).
\end{eqnarray*}

\end{hyp} 
Notice that Hypothesis \ref{hyp7.2} is a version of Hypothesis \ref{lem:bnd-ep-holder-1/2} for the lift of $(z^{n}, n^{\al}\ep)$. 

Thanks to the previous lemma, we can now consider $(  y^{n}, B, n^{\al} {  \ep},  n^{\al}{ {  {\mathcal{E}} }}  ,  n^{\al}{ { \wt{\mathcal{E}} }}  )$ as a single rough path. This is achieved in the following lemma:

\begin{lemma}\label{lem7.1}
Suppose that
$b \in C^{2}_{b}$, $V\in C^{4}_{b}$, and
 Hypothesis \ref{hyp7.2} is met for  
  $\ga>\frac13$ and $\al<2H-\frac12$.
We also consider the processes $\ce$ and $\tilde{\ce}$ as defined in Lemma \ref{lem:eq-mathcalE}.
Then the vector 
$
(z^{n}, n^{\al} {  \ep},  n^{\al}{ {  {\mathcal{E}} }}  ,  n^{\al}{ { \wt{\mathcal{E}} }})
$
satisfies the following upper bound:
 \begin{eqnarray}\label{eq:upp-bnd-S2-zn-eps-cal-E}
 \| S_{2}(  z^{n}, n^{\al} {  \ep},  n^{\al}{ {  {\mathcal{E}} }}  ,  n^{\al}{ { \wt{\mathcal{E}} }}  ) \|_{\ga}&\leq& G . 
\end{eqnarray}
\end{lemma}

\begin{proof} 
Note that $\Phi$, $\Psi$, $\Phi^{n}$, $\Psi^{n}$ are solutions of equations driven by $z^{n}$. 
Furthermore, owing to relation \eqref{e7.1ii}, it is easily seen that $n^{\al}\ce$ and $n^{\al}\wt\ce$ are solutions of   equations driven by $(z^{n},n^{\al}\ep)$. Thus   \eqref{eq:upp-bnd-S2-zn-eps-cal-E} is a direct consequence of Theorem \ref{thm 3.3} (linear part) and of Hypothesis~\ref{hyp7.2}.
\end{proof}

\begin{remark}\label{remark7.2}
Roughly speaking, Lemma \ref{lem7.1} shows that if the convergence rate of the numerical scheme $y^{n}$   to $y$ is $n^{-\al}$, then so is that of  $(\Phi^{n}, \Psi^{n})$     to $(\Phi , \Psi )$ as $n\rightarrow \infty$. 
\end{remark}

\subsection{Optimal rate of convergence}\label{section7.2}

Recall that $\ep$ is defined by \eref{eq:def-ep-t} and $\delta \wt{\ep}_{st} =\delta\ep_{st} - \delta \hat{\ep}_{st} $ is defined in \eref{e7.1}. With the preliminary results obtained in Section \ref{section7.1}, we can now go further in our analysis of the error process $\ep$.

\begin{prop} \label{prop6.7}
Consider the process $z^{n} = (y^{n}, B)$ and the error process $\ep$ defined in~\eref{eq:def-ep-t}. 
Assume that $b \in C^{2}_{b}$, $V\in C^{4}_{b}$.
As in Lemma \ref{lem7.1}, 
suppose that Hypothesis \ref{hyp7.2} is met 
 for some exponents $\al,\ga$   such that $ \frac13<\ga<H$ and $\al<2H-\frac12$. Take $\kappa>0$ arbitrarily small. Then the following estimate holds true for $(s,t) \in \cs_{2}(\ll0,T\rr)$:
\begin{eqnarray}\label{eq7.7}
\left|  \delta\wt{\ep}_{st}  \right| &\leq & G (n^{1-3\ga -\al +2\kappa} +n^{-\ga}  ) (t-s)^{1-\ga}.
\end{eqnarray}
In addition, for all  $(s,t) \in \cs_{2}(\ll0,T\rr)$ we also have:
\begin{eqnarray}\label{eq7.9}
|\delta\ep_{st}| &\leq& G (n^{1-3\ga -\al +2\kappa}   +n^{\frac12 -2\ga} )(t-s)^{\frac12 - \kappa}.
\end{eqnarray}

\end{prop}
\begin{remark}
In Proposition \ref{prop6.7}, we prove that Hypothesis \ref{condition6.1} is satisfied for $\tilde{\ep}$, with $\al$ replaced by $(3\ga-1+\al)\wedge \ga$. We also prove that Hypothesis \ref{lem:bnd-ep-holder-1/2} for $\ep$ is fulfilled with an improved exponent $\al_{1} = (\al+3\ga-1)\wedge 2\ga-\frac12$, which satisfies $\al_{1}>\al$. 
\end{remark}

\begin{proof}[Proof of Proposition \ref{prop6.7}]
This proof is divided into several steps.

\noindent \textit{Step 1: Decomposition of $\delta \ep$.}\quad
Starting from the decomposition \eref{e6.4i} of $\ep$, one can write:
\begin{eqnarray}\label{e6.13}
\delta \ep_{st} &=&U^{1}_{st} + U^{2}_{st}\,,
\end{eqnarray}
with 
\begin{eqnarray}\label{e7.7i}
U^{1}_{st} &=&  \sum_{e=1}^{3} \int_{s}^{t} \Psi_{u} dA^{e}_{u}  
\\
 U^{2}_{st} &=& \sum_{e=1}^{3} \int_{s}^{t} (\Psi^{n}_{u}- \Psi_{u}) dA^{e}_{u} =-\sum_{e=1}^{3} \int_{s}^{t} \Psi_{u}  {\wt{\mathcal{E}}}_{u} dA^{e}_{u} ,
 \nonumber
\end{eqnarray}
where we recall that $\tilde{\ce}$ has been introduced in Lemma \ref{lem:eq-mathcalE}.
Moreover, the term $U^{2}$ above is easily bounded. Indeed,
applying Lemma \ref{lem6.2} and taking into account the estimate in Lemma \ref{lem7.1}
 we obtain
\begin{eqnarray}\label{e 7.7}
\left|
U^{2}_{st}
\right|
 &\leq &  G (t-s)^{1-\ga} n^{1-3\ga -\al}
\end{eqnarray}
for all $(s,t )\in \mathcal{S}_{2}(\ll 0,T \rr)$.  

\noindent \textit{Step 2: Decomposition of $U^{1}$.}\quad We turn to the quantity $U^{1}_{st}$ given by \eqref{e7.7i}. 
First, from the expression of $A^{2}$ in \eref{e6.4ii} and a discrete-time decomposition similar to the estimate of \eref{e5.5} it is clear that  
\begin{eqnarray}\label{e 7.8}
\Big| \int_{s}^{t} \Psi_{u} dA^{2}_{u} \Big| &\leq& G n^{-\ga}(t-s).
\end{eqnarray}
In the case $e=1$, recall expression \eref{eq:def-eta-A1} for $A^{1} $. Then one can decompose $\int_{s}^{t} \Psi_{u}dA^{1}_{u}$ into
\begin{eqnarray}\label{e7.9}
\int_{s}^{t} \Psi_{u} dA^{1}_{u}  &=&
M^{1}_{st}+M^{2}_{st},
\end{eqnarray}
where $M^{1}$ and $M^{2}$ are defined by:
\begin{eqnarray}
M^{1}_{st} &=& -
 \frac12   \sum_{j=1}^{m} \int_{s}^{t} \Psi_{\eta(u)} 
        \partial {V}_{j}{V}_{j} (y^{n}_{\eta(u)} ) d(u -\eta(u))^{2H} 
\label{eq7.13}
\\
M^{2}_{st}&=& -  \frac12   \sum_{j=1}^{m}\int_{s}^{t}   \delta \Psi_{\eta(u)u}
       \partial {V}_{j}{V}_{j} (y^{n}_{\eta(u)}) d (u -\eta(u))^{2H} .
       \nonumber
\end{eqnarray}
We defer the evaluation for $M^{1}$ to the end of the proof, but $M^{2}$ is easily controlled. Indeed, 
by \eref{e11.1 ii} in Lemma \ref{lem6.11} applied to $ f = \partial V_{j}V_{j} (y^{n})$ and $g=\Psi$,  we have
\begin{eqnarray}\label{e 7.9}
|M^{2}_{st}| &\leq& G n^{1-4\ga} (t-s)^{1-\ga}.
\end{eqnarray}

We now decompose the term $ \int_{s}^{t} \Psi_{u} dA^{3}_{u} $ in \eqref{e7.7i}: it is readily checked that owing to \eref{e6.4ii}, one can write
\begin{eqnarray*} 
 \int_{s}^{t} \Psi_{u} dA^{3}_{u}  &=&  \sum_{j=1}^{m}\int_{s}^{t} \Psi_{u}\int_{\eta(u)}^{u} \partial  {V}_{j}(y^{n}_{v}) dy^{n }_{v} d B^{j}_{u}
  \,.
\end{eqnarray*}
Hence, plugging the equation \eref{e1.2} followed by $y^{n}$ into this relation we can write:
\begin{eqnarray}\label{eq7.18}
\int_{s}^{t} \Psi_{u} dA^{3}_{u}  &=&  I_{st}^{1} +I_{st}^{2}+I_{st}^{3},
\end{eqnarray}
where $ I^{1} ,I^{2},I^{3}$ are given by
\begin{eqnarray}
I^{1}_{st}&=&  \sum_{j=1}^{m}\int_{s}^{t} \Psi_{u}\int_{\eta(u)}^{u}   \partial  {V}_{j}(y^{n}_{v})   b(y^{n}_{\eta(v)}) dv d B^{j}_{u}\,,
\nonumber
\\
I^{2}_{st} &=& \sum_{j=1}^{m}\int_{s}^{t} \Psi_{u}\int_{\eta(u)}^{u}   \partial  {V}_{j}(y^{n}_{v})   V(y^{n}_{\eta(v)}) d B_{v} d B^{j}_{u}\,,
\label{e7.14}
\\
I^{3}_{st} &=& \frac12   \sum_{j,j'=1}^{m}\int_{s}^{t} \Psi_{u}\int_{\eta(u)}^{u} \partial  {V}_{j}(y^{n}_{v})   \partial V_{j'}V_{j'} (y^{n}_{\eta(v)}) d (v-\eta(v))^{2H}  d B^{j}_{u}\,.
\nonumber
\end{eqnarray}

\noindent \textit{Step 3: Estimate of $I^{1}$, $I^{2}$, $I^{3}$.}\quad
We will now evaluate $I^{1}$, $I^{2}$, $I^{3}$ separately. First, invoking a discrete-time decomposition as in \eref{e5.5} again, we get:
\begin{eqnarray}\label{e 7.10}
|I^{1}_{st}|&\leq& G (t-s)n^{-\ga}.
\end{eqnarray}
On the other hand, applying \eref{e11.1ii} in 
Lemma \ref{lem6.11}
 to $  {I}^{3}_{st}   $ we obtain: 
\begin{eqnarray}\label{e 7.11}
|     {I}^{3}_{st}    | &\leq&  Gn^{1 - 4\ga} (t-s)^{1-\ga}.
\end{eqnarray}
 
Let us now consider the term $I^{2} $ defined by \eref{e7.14}. To this aim,    set
\begin{eqnarray}\label{e7.18}
 J^{2}_{st} &=& 
  \sum_{j=1}^{m}\int_{s}^{t} \Psi_{\eta(u)}\int_{\eta(u)}^{u}   \partial  {V}_{j}(y^{n}_{\eta(v)})   V(y^{n}_{\eta(v)}) d B_{v} d B^{j}_{u}.
\end{eqnarray}
By \eref{e 11.1i} in Lemma \ref{lem6.11}, the patient reader can check that for any $\kappa>0$ and $(s,t) \in \cs_{2}([0,T])$ we have:
\begin{eqnarray}\label{e7.7}
|I^{2}_{st} -  {J}^{2}_{st}| &\leq &  G n^{1-4\ga+2\kappa} (t-s)^{1-\ga} .
\end{eqnarray}
In addition, we may consider $\kappa>0$ such that $1-4\ga+2\kappa<-\ga$. In this case, the previous bound becomes:
\begin{eqnarray}\label{e 7.12}
| {I}^{2}_{st} - J^{2}_{st}| &\leq & G n^{-\ga } (t-s)^{1-\ga}.
\end{eqnarray}

\noindent \textit{Step 4: Conclusion.}\quad
So far we have made a sequence of decompositions for $ {\ep}$ in  \eref{e6.13},  \eref{e7.7i}, \eref{e7.9}, \eref{eq7.18}. Taking into account the   estimates \eref{e 7.7}, \eref{e 7.8}, \eref{e 7.9}, \eref{e 7.10}, \eref{e 7.11}, \eref{e 7.12},  it is clear that to prove the  estimate \eref{eq7.7} for $\tilde{\ep}$ it suffices to show that
\begin{eqnarray}\label{e7.19}
\left| \delta \bar{\ep}_{st} \right|  &\leq& G  n^{1-3\ga -\al +2\kappa}   (t-s)^{1-\ga}
\end{eqnarray}
for any $\kappa>0$,
where the increment $\delta \bar{\ep}$ is defined by
\begin{eqnarray*}
\delta \bar{\ep}_{st} &=& J^{2}_{st} + M^{1}_{st} - \delta \hat{\ep}_{st}
\end{eqnarray*}
with $J^{2}$ given by \eref{e7.18} and $M^{1}$ given by \eref{eq7.13}. We also recall that $\hat{\ep}$ has been introduced in \eref{e6.2i} and is given by the following expression:
\begin{eqnarray}\label{e7.23}
\delta \hat{\ep}_{st}&=&
\sum_{jj'=1}^{m} \sum_{t_{k}=s}^{t-}   \Psi_{t_{k}} \partial {V}_{j}{V}_{j'} (y_{t_{k}})  \delta F^{jj'}_{t_{k}t_{k+1}}.
\end{eqnarray}
In order to prove \eref{e7.19}, let us now observe that 
\begin{eqnarray}\label{e7.13}
 \delta \bar{\ep}_{st} &=&
 \sum_{jj'=1}^{m} \sum_{t_{k}=s}^{t-}    \Psi_{t_{k}}( \partial {V}_{j}{V}_{j'} (y^{n}_{t_{k}}) - \partial {V}_{j}{V}_{j'} (y_{t_{k}}) ) \delta F^{jj'}_{t_{k}t_{k+1}}.
\end{eqnarray}
Further, we note  that similarly to \eref{eq5.10}, the following identity holds true:
\begin{eqnarray*}
 \partial {V}_{j}{V}_{j'} (y^{n}_{t_{k}}) - \partial {V}_{j}{V}_{j'} (y_{t_{k}}) &=&- \sum_{i=1}^{d} \int_{0}^{1}  ( \partial_{i}  \partial {V}_{j}{V}_{j'}) ( \lambda  y^{n}_{t_{k}} + (1-\lambda) y_{t_{k}} ) )  d\lambda \cdot ( \Phi^{n}_{t_{k}} \ep_{t_{k}})^{i}    .
\end{eqnarray*}
According to Hypothesis \ref{hyp7.2}, $y^{n}$, $\Psi$ and $n^{\al}\ep$ are $\ga$-H\"older continuous functions. Hence 
 \begin{eqnarray}\label{e7.20}
n^{\al}  \| \Psi ( \partial {V}_{j}{V}_{j'} (y^{n} ) - \partial {V}_{j}{V}_{j'} (y ) ) \|_{\ga} &\leq & G.
\end{eqnarray}
In order to bound the right-hand side of \eref{e7.13}, let us apply a bound on weighted sums of the process $F$ as in relation \eref{e6.7i}. Taking into account \eref{e7.20}, this yields:
\begin{eqnarray*}
|\delta \bar{\ep}_{st}| \leq Gn^{1-3\ga -\al}(t-s)^{1-\ga-2\kappa}
\leq Gn^{1-3\ga -\al +2\kappa}(t-s)^{1-\ga}
\end{eqnarray*}
for an arbitrary $\kappa>0$, which is our claim \eref{e7.19}. 
The proof of \eref{eq7.7} is now complete.

In order to get \eref{eq7.9} from \eref{eq7.7}, we recall once again relation \eref{eq6.19} and we just analyze the term $\delta \hat{\ep}$. This can be done in a similar way as in  \eref{eq6.20} and \eref{e6.17}. Our proof is   complete.
\end{proof}

\begin{theorem}\label{thm7.3}
Let $\ep$ be given by \eref{eq:def-ep-t} and $\wt{\ep}$ be defined in \eref{e7.1}. Suppose that $b \in C^{2}_{b}$, $V\in C^{4}_{b}$. Then the following statements  holds true: 

(i) There exists a constant $\ka_{H}>0$  depending on $H$ such that
\begin{eqnarray}\label{e7.26}
n^{2H- \frac12 }  { \left| \delta \wt{\ep}_{st}  \right| } &\leq &G n^{-\kappa_{H}} { |t-s|^{1-\ga} },
\end{eqnarray}
  In particular, we have the following almost sure convergence:   
\begin{eqnarray*}
\lim_{n\rightarrow \infty} n^{2H- \frac12 }  \delta \wt{\ep}_{st}    &=& 0.
\end{eqnarray*}

(ii) Take a constant $\kappa>0$. The error process $y-y^{n}$ satisfies 
\begin{eqnarray}\label{e7.25}
n^{2H-\frac12 -\kappa } \sup_{t\in [0,T]} |y_{t}-y^{n}_{t}| \rightarrow 0 \quad \text{ as } n\rightarrow \infty,
\end{eqnarray}
meaning that the Euler scheme has a rate of convergence $n^{\frac12 -2H +\kappa}$ for an arbitrary $\kappa >0$.

\end{theorem}
\begin{proof} Item (i): Take $\frac13<\ga<H$. 
According to Proposition \ref{prop5.3} for $S_{2} (n^{3\ga-1} \ep, z^{n})$, Hypothesis \ref{hyp7.2} holds with $\al = 3\ga-1$. Hence one can apply Proposition   \ref{prop6.7} in order to get
\begin{eqnarray}\label{e7.6}
\left|\delta  \wt{\ep}_{st}  \right| &\leq & G (n^{2(1-3\ga) +2\kappa} +n^{-\ga}  ) (t-s)^{1-\ga}.
\end{eqnarray}
In the case $ \frac38< H<\frac12$, it  is easy to see that  $ 3H - 1 <2H-\frac12$ and $ 2 (3H -1) > 2H -\frac12 $. Take $\frac38< \ga<H$ such   that $ 2 (3\ga -1) -2\kappa> 2H -\frac12 $ and $\ga> 2H-\frac12$. Then   
\eref{e7.6} implies that  for $ \frac38< H<\frac12 $ we have $ n^{2H-\frac12} |\delta\wt{\ep}_{st}| \leq G n^{-\kappa_{H}} { |t-s|^{1-\ga} }$  for $\ka_{H} = \left(  (6\ga-2-2\kappa)\wedge \ga\right) - (2H-\frac12 )$. This proves our claim \eref{e7.26} for $\frac38 < H< \frac12$. 

Let us now handle the case $\frac13< H\leq \frac38$. To this aim, set   $H_{k} = \frac{2k-1}{6k-4}$, $k \geq 2$.   We consider the case when $ H_{k+1}< H\leq H_{k}$ for all $k\geq 2$.  It  is easy to verify that  $k(3H - 1) <2H-\frac12$ and $ (k+1) (3H -1) > 2H -\frac12 $. We can thus choose $H_{k+1}<\ga<H$ and $\kappa > 0$  such that $ (k+1) (3\ga -1) - 3k\kappa > 2H -\frac12 $ and $\ga>2H-\frac12$.
It follows from  inequality \eref{e7.6}   that 
\begin{eqnarray*}
\left|  \delta\wt{\ep}_{st}  \right| &\leq & G  n^{2(1-3\ga) +2\kappa}   (t-s)^{1-\ga}.
\end{eqnarray*}
We can now iterate this bound in the following way: 
  apply Proposition \ref{prop7.1} which gives
\begin{eqnarray}\label{e7.22}
|S_{2}(z^{n} , n^{2( 3\ga-1) -3\kappa} \ep)_{st}| &\leq& G(t-s)^{\ga}.
\end{eqnarray}
Then invoke Proposition \ref{prop6.7} again.  Taking into account the   estimate \eref{e7.22}
we obtain
\begin{eqnarray*}
\left| \delta \wt{\ep}_{st}  \right| &\leq & G (n^{3(1-3\ga) +5\kappa } +n^{-\ga}  ) (t-s)^{1-\ga}.
\end{eqnarray*}
We can now repeat   the application of Proposition \ref{prop7.1}  and \ref{prop6.7} in order to get
\begin{eqnarray*}
\left|  \delta\wt{\ep}_{st}  \right| &\leq & G (n^{(k+1)(1-3\ga) + 3k\kappa  } +n^{-\ga}  ) (t-s)^{1-\ga}.
\end{eqnarray*}
This implies that $n^{2H-\frac12} |\delta\wt{\ep}_{st}| \leq G n^{-\kappa_{H}} { |t-s|^{1-\ga} }$  for 
\begin{eqnarray*}
\ka_{H} &=& \left(  (k+1) (3\ga -1) - 3k\kappa \right) \wedge \ga - \big(2H-\frac12 \big).
\end{eqnarray*}

Item (ii): Recall $\delta \ep = \delta \hat{\ep}+\delta \tilde{\ep}$ given by relation \eref{e7.1}. With item (i) in hand, our claim \eref{e7.25} is reduced to prove that 
\begin{eqnarray}\label{e7.28}
\lim_{n\rightarrow \infty} n^{2H-\frac12 -\kappa} 
\sup_{(s,t) \in \cs_{2}(\ll 0,T \rr )}
 |\delta \hat{\ep}_{st} | =0.
\end{eqnarray}
In order to prove \eref{e7.28}, recall the expression \eref{e7.23} for $\delta \hat{\ep}_{st}$ as a weighted sum of the increment $\delta F$. We can thus apply Corollary \ref{lem9.1} with $f= \Psi \partial V_{j}V_{j'} (y)$. Indeed, one can easily see that $f$ satisfies the assumptions of Proposition \ref{prop3.6}: both $\Psi$ and $\partial V_{j}V_{j'} (y)$ are controlled processes admitting moments of any order (see \cite{CLL} for the integrability of $\Psi$). Applying Corollary \ref{lem9.1} we thus get 
\begin{eqnarray*}
\| \delta \hat{\ep}_{st} \|_{p} &\leq& K n^{\frac12 -2H} (t-s)^{\frac12}.
\end{eqnarray*}
Then, invoking Lemma \ref{lem4.2}, we end up with 
\begin{eqnarray*}
|\delta \hat{\ep}_{st}| &\leq& G n^{\frac12-2H+\frac{\kappa}{2}}, 
\end{eqnarray*}
which finishes the proof.
 \end{proof}

\section{Asymptotic error distributions}\label{section9}
In this section, we first review a central limit theorem from \cite{NTU} (see also \cite{HLN}), then in the second part, we prove the asymptotic error distribution of the Euler scheme. 
\subsection{A central limit theorem for the L\'evy area process}\label{section8}

In this subsection,  we recall   a central limit theorem for the process $F$.   
 Let us first define some parameters  that will appear in the limit distribution of $F$. Namely, for $k \in \ZZ$, we set:
 \begin{eqnarray}
Q(k)  = \int_{[0,1]^{2}} R \Big( \begin{array}{c c} 0 &s' \\ k  &k+s\end{array} \Big)  dR(s,s' ), 
\quad
P(k)  = \int_{[0,1]^{2}} R \Big( \begin{array}{c c} 0 &s' \\ k+s  & k+1 \end{array} \Big)  dR(s,s' ),
\label{eq8.4}
\end{eqnarray}
where we recall that $R$ is the   covariance function    defined by~\eqref{eq:cov-fbm}, whose rectangular increments are given by  formula~\eref{eq3.1}.  We now state a slight elaboration of \cite[Theorem 3]{NTU} and~\cite[Proposition 5.1]{HLN}.

\begin{prop}\label{prop8.3}
 Let $B= (B^{1}, \dots, B^{m})$ be a $m$-dimensional standard fBm with Hurst parameter $\frac14< H<\frac12$. Let
 $\bar{F}_{t} = F_{\eta(t)}$    for $t \in [0,T] $, where we 
 recall that the process $F$ is defined  by \eref{e4.1} and $\eta$ is given by \eref{eq:def-eta-A1}. 
  Then the finite dimensional distributions of $\{n^{2H-\frac12} \bar{F} , B \}  $ converge  weakly to those of $(W,B)$, where $W= (W^{ij} )$ is an $m\times m$-dimensional Brownian  motion, independent of $B$, such that 
 \begin{eqnarray}\label{e8.6}
\mE[W^{ij}_{t}W^{i'j'}_{s}] &=& T^{4H-1} (Q\delta_{ii'}\delta_{jj'} + P \delta_{ij'}\delta_{ji'} ) (t\wedge s). 
\end{eqnarray}
In formula \eref{e8.6}, we have set $\delta_{ij}=1$ if $i=j$ and $\delta_{ij}=0$ if $i\neq j$, and 
$Q = \sum_{k\in \ZZ} Q(k)$,  $P = \sum_{k\in \ZZ} P(k) $. 
 \end{prop}
 \begin{remark}
 Proposition \ref{prop8.3} shows that the process $n^{2H-\frac12} \bar{F}$ converges stably to $W$ when $ \frac14<H<\frac12$. We refer the reader to Chapter 8 in \cite{JS} for the definition of stable convergence and its equivalent conditions.
 \end{remark}

\begin{remark}
The following plot of  constants $Q$ and $P$ shows that $Q$ is strictly larger than~$P$ for $H \in (\frac14, \frac34)$.
In particular, this implies that the $m\times m$ random  matrix $W$ defined in Proposition \ref{prop8.3} is not symmetric. Let us also mention that as has been observed in \cite{HLN2}, the fact that $Q>P$ results in   different features of the Crank-Nicolson scheme and the numerical schemes \eref{e2i} and \eref{e3} between the scalar case and  the  multi-dimensional cases.
\end{remark}
\begin{center}
 \includegraphics[scale=.4]{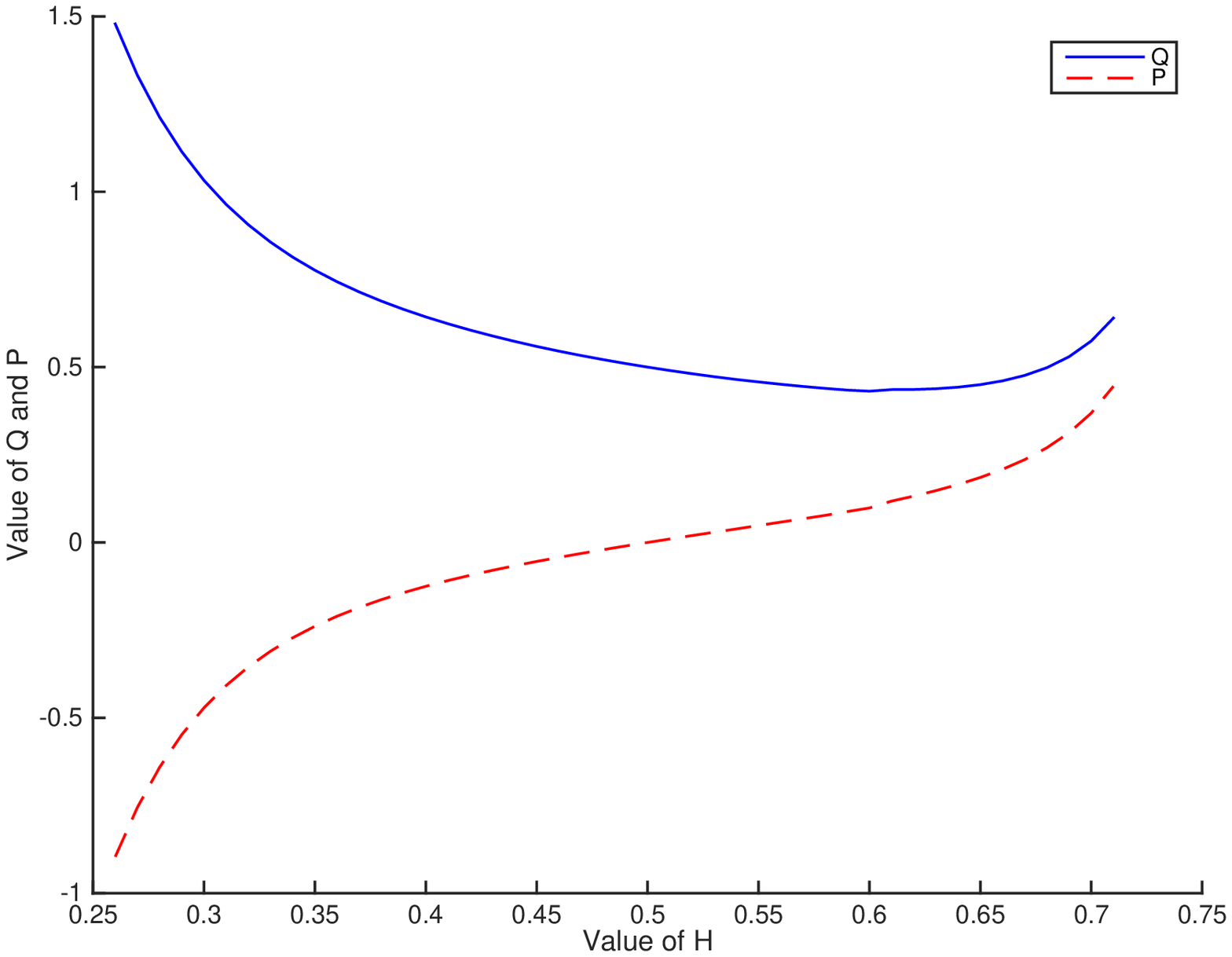}
\end{center}

\subsection{Asymptotic error distributions}
We can now prove the convergence of a renormalized version of the error process $y-y^{n}$ related to the Euler-type scheme $y^{n}$. Namely, we prove the following central limit theorem.

 \begin{theorem}\label{thm9.1}
 Let $y^{n}$ be the   Euler scheme defined in \eref{e4}.
 Suppose $b \in C^{2}_{b}$ and $V\in C^{4}_{b}$.
  Then 
 the sequence of processes 
   $ ( n^{2H-\frac12} ( {y}- {y}^{n}), B  )$ 
   converges weakly in
    $D([0,T])$    to the couple  $(U , B)$  as $n\rightarrow \infty$, where $U$ is the solution of the linear SDE
 \begin{eqnarray}\label{e9.1}
U_{t} &=& \int_{0}^{t} \partial b(y_{s}) U_{s}ds + \sum_{j=1}^{m} \int_{0}^{t} \partial V_{j} (y_{s}) U_{s} dB^{j}_{s} + \sum_{i,j=1}^{m} \int_{0}^{t} \partial V_{i}V_{j } (y_{s}) d W^{ij}_{s},
\end{eqnarray}
 and where $W$ is the Wiener process obtained in Proposition \ref{prop8.3}.
 \end{theorem}
\begin{proof} 
Recalling that $y-y^{n} = \Phi \ep$. 
We consider the following decomposition:
\begin{equation}\label{e9.2}
 {y}_{t}- {y}^{n}_{t} = \Phi^{n}_{\eta(t)} \tilde{\ep}_{\eta(t)} + (\Phi^{n}_{\eta(t)}-\Phi_{\eta(t)})  \hat{\ep}_{\eta(t)} +\Phi_{\eta(t)}   \hat{\ep}_{\eta(t)} + (y_{t}- {y}_{\eta(t)})-(y^{n}_{t} - {y}^{n}_{\eta(t)}  )  ,
\end{equation}
 where   recall that $\ep$ is defined by \eref{eq:def-ep-t}, and  $\tilde{\ep}$, $\hat{\ep}$ are respectively introduced in \eref{e7.1} and~\eref{e6.2i}.

Note that, thanks to Theorem \ref{thm7.3}, Lemma \ref{lem7.1} and Corollary \ref{lem9.1} we have almost surely:
\begin{eqnarray}\label{eq9.3}
\lim_{n\to \infty} \sup_{t\in [0,T]} 
n^{2H-\frac12} \left( \Phi^{n}_{\eta(t)} \tilde{\ep}_{\eta(t)} + (\Phi^{n}_{\eta(t)}-\Phi_{\eta(t)})  \hat{\ep}_{\eta(t)}
\right)
 =0.
\end{eqnarray}
 On the other hand,   thanks to  Theorem \ref{thm 3.3} and  equation \eref{e1.2} governing $y^{n}$ it is clear that    
 \begin{eqnarray}\label{eq9.4}
|y_{t}- {y}_{\eta(t)} | + | {y}^{n}_{t} -  y^{n}_{\eta(t)}| \leq G  {n^{-\ga}} = G  {n^{\frac12 -2H-\kappa}}
\end{eqnarray}
  for $G= K(1+\|B\|^{1/\ga}_{\ga})$ and for any $\ga<H$, where we use the fact that $H>2H-\frac12$ for the last inequality and we take $\kappa = \ga-2H+\frac12>0$.   Therefore, going back to \eref{e9.2}  the convergence of the finite dimensional distributions of $(n^{2H-\frac12}( {y}- {y}^{n}), B)$ can be reduced to the convergence of $(n^{2H-\frac12} \Phi_{\eta{(t)}}   \hat{\ep}_{\eta{(t)}}, B_{t}, \,t\in [0,T])$. 
 Furthermore,   Proposition \ref{prop9.3}
   delivers a central limit theorem for general weighted sums of the process $F$. Taking into account the expression \eref{e6.2i} for $\hat{\ep}$, it can be applied in order to get the convergence of the finite dimensional distributions of   
 $(n^{2H-\frac12} \Phi_{\eta(\cdot)}   \hat{\ep}_{\eta(\cdot)}, B)$ to  $(U,B)$, where
\begin{eqnarray}\label{eq9.5}
 U_{t} &=&\sum_{jj'=1}^{m} \Phi_{t} \int_{0}^{t} \Psi_{u} \partial {V}_{j}{V}_{j'} (y_{u})  d W^{jj'}, \quad t\in [0,T]
\end{eqnarray}
as $n\rightarrow \infty$.
Similarly to \eref{e6.4j}, an easy variation of parameter argument shows that $U$ defined by \eref{eq9.5} solves the linear SDE \eref{e9.1}. Summarizing our considerations so far, we have obtained the finite dimensional distribution convergence of 
 $(n^{2H-\frac12} (y-y^{n}), B)$ to $(U, B)$.

 It remains to show the tightness of the error $n^{2H-\frac12}( {y}- {y}^{n})  $.   To this end,  
 we invoke relations \eref{eq9.3}, \eref{eq9.4} and apply   Lemma~3.31 in  Chapter 6 \cite{JS} to   our decomposition  \eref{e9.2}. 
 Therefore, we are   reduced to show   the tightnesss of    $n^{2H-\frac12} \Phi_{\eta(\cdot)}   \hat{\ep}_{\eta(\cdot)}$. This follows immediately from Corollary \ref{lem9.1} and a tightness criterion in (13.14) of \cite{B}. The proof is now complete.
\end{proof}

We now state the limit theorem on which Theorem \ref{thm9.1} relies.

\begin{prop}\label{prop9.3}
Let $f $, $g $ be   processes defined as in Proposition \ref{prop3.6} and $W$ be the Brownian motion defined in Proposition \ref{prop8.3}.  
Set 
\begin{eqnarray*}
\Theta^{n}_{t} = n^{2H-\frac12} \sum_{k=0}^{\left\lfloor \frac{nt}{T} \right\rfloor} f_{t_{k}} \otimes \delta F_{t_{k}t_{k+1}}
\quad\text{and }\quad
\Theta_{t} = \int_{0}^{t} f_{s}\otimes dW_{s}\,.
\end{eqnarray*}
 Then the following relation holds true as $n\to \infty$:
 \begin{eqnarray*}
\left( \Theta^{n}, B \right)  &\xrightarrow{\rm{f.d.d.}}& \left( \Theta , B \right),\quad \text{as} \quad n\to \infty.
\end{eqnarray*}
\end{prop}
\begin{proof}
The proposition is an application of Theorem \ref{prop9.2}. As in  Corollary \ref{lem9.1}, we take $\ga>\frac13$, $h = n^{2H-\frac12} F $ and $\al = \frac12$. It suffices to verify   the conditions   \eref{e3.8i},  \eref{eq9.6} and \eref{eq9.9}. According to Corollary \ref{lem9.1} and Proposition \ref{prop8.3}, conditions   \eref{e3.8i} and \eref{eq9.6} holds true for our $h$. 
Applying Lemma \ref{lem11.4} and take  $\kappa: \frac12 +H-\ka +\ga>1$ we obtain the relation \eref{eq9.9}.
\end{proof}

\section{Appendix}

\subsection{Estimates for the H\"older semi-norm of a rough path}
The following   lemma is convenient while deriving upper bound estimates for the H\"older semi-norm  of  a rough path.
  \begin{lemma}\label{lem4.5}
 Let $X$ and $Y$ be   functions on $[0,T]$ and $\ZZ $ be a two parameter path on $\cs_{2}([0,T])$ such that $\delta \ZZ_{sut} = \delta X_{su} \otimes \delta Y_{ut}$. 
 We recall the notation $(t_{0},t_{1},\dots, t_{n})$ for a partition of $[0,T]$ and $\ll s,t\rr$ for discrete intervals given in the introduction. 
 Suppose that
 \begin{eqnarray*}
|\delta X_{st}| + | \delta Y_{st} | + |\ZZ_{st}|^{\frac12}&\leq& 
\begin{cases}
K|t-s|^{\be}, \quad (s,t )\in \mathcal{S}_{2}(\ll 0,T \rr)
\\
 K|t-s|^{\beta}, \quad s, t\in [t_{k},t_{k+1}], \, k=0,1,\dots,n-1
\end{cases}
\end{eqnarray*}
  for some  $   \beta>0$  and   $K>0$. 
   Then the following relations hold for all $(s,t)  \in \cs_{2}( [0,T])$:
 \begin{eqnarray*}
 |\delta X_{st}| + |\delta Y_{st}| \leq K(t-s)^{\be} \quad\quad \quad |\ZZ_{st}| \leq  K  (t-s)^{2\beta}.
\end{eqnarray*}
 \end{lemma}
 
 \begin{proof} We first consider $\delta X$ and $\delta Y$. 
 Take $t_{k-1}\leq s\leq t_{k}\leq t_{k'}\leq t\leq t_{k'+1}$. We have
\begin{eqnarray*}
| \delta X_{st}  | &\leq& | \delta X_{t_{k'}t} | +| \delta X_{t_{k}t_{k'}} | +| \delta X_{st_{k}}  |
\\
&\leq&K\left(  (t-t_{k'})^{\beta} + (t_{k'}-t_{k})^{\be}+ (t_{k}-s)^{\beta} \right)
 \\
&\leq& K(t-s)^{\beta}
. 
\end{eqnarray*}
The same estimate holds for $\delta Y$.
We now turn to the estimate for $\ZZ_{st}$\,. We consider $t_{k-1}\leq s\leq t_{k}\leq t_{k'}\leq t\leq t_{k'+1}$ again, and we  have 
\begin{eqnarray*}
|\ZZ_{s t}|  &=&  |\ZZ_{s t_{k}}+\ZZ_{t_{k} t_{k'}}+\ZZ_{t_{k'}  t}+ \delta X_{t_{k} t_{k'}} \otimes \delta Y_{t_{k'} t} + \delta X_{s t_{k}} \otimes \delta Y_{t_{k} t} |
\\
&\leq& K  \left( (t_{k} - s)^{2\beta} + (t-t_{k'})^{2\beta} + (t_{k'}-t_{k})^{2\be} 
\right)
\\
&&+ K^{2} ( (t_{k'} - t_{k})^{\beta}(t-t_{k'})^{\beta}  + (t_{k} - s)^{\beta}(t-t_{k})^{\beta} ) 
\\
&\leq& K (t-s)^{2\beta} .
\end{eqnarray*}
The proof is complete. 
\end{proof}

\subsection{Estimates for some iterated integrals}\label{section10.2}

 This section summarizes some estimates for weighted sums involving double or triple iterated integrals.

 \begin{lemma}\label{lem6.11}
 Let $B$ be our $m$-dimensional fBm with Hurst parameter $H>\frac13$. 
Let $f$ be a real-valued path on $[0,T]$ such that $  \|   f  \|_{\ga} \leq K  $ for $ \frac13< \ga< H$. 
Suppose that
$g$ is another real-valued path and $\tilde{g} = (\tilde{g}^{1},\tilde{g}^{2},\tilde{g}^{3})$ is a continuous path in $\cl(\RR^{m}, \RR) \times \RR \times \RR $, such that 
 $S_{2}( g, \tilde{g}, B)$ is well defined as a   $\ga$-rough path,
and that $\delta g$ can be decomposed as 
\begin{eqnarray*}
 \delta g_{t_{k}u} =\int_{s}^{t} \tilde{g}^{1}_{u} dB_{u}+\int_{s}^{t} \tilde{g}^{2}_{u} du+\int_{s}^{t} \tilde{g}^{3}_{u} d (u-t_{k})^{2H}  
\end{eqnarray*}
for all $(s,t) \in \cs_{2} (\ll 0,T \rr)$.  
 Consider an arbitrarily small parameter $\kappa>0$.
Then the following  inequalities hold  true for $(s,t) \in \cs_{2} (\ll0,T\rr)$:
\begin{eqnarray}
\Big| \sum_{t_{k}=s}^{t-} f_{t_{k}} \int_{t_{k}}^{t_{k+1}}  \delta g_{t_{k}u}      d(u-t_{k})^{2H} \Big|
&\leq& G n^{1-4\ga+2\kappa} (t-s)^{1-\ga},  
\label{e11.1 ii}
\\
\Big| \sum_{t_{k}=s}^{t-} f_{t_{k}} \int_{t_{k}}^{t_{k+1}} \int_{t_{k}}^{u} g_{v}   d(v-t_{k})^{2H} dB_{u} \Big|
&\leq& G n^{1-4\ga+2\kappa} (t-s)^{1-\ga},
\label{e11.1ii}
\\
\Big| \sum_{t_{k}=s}^{t-} f_{t_{k}} \int_{t_{k}}^{t_{k+1}}  \delta g_{t_{k}u}   \int_{t_{k}}^{u} dB_{v} \otimes dB_{u} \Big|
&\leq& G n^{1-4\ga+2\kappa} (t-s)^{1-\ga} .
\label{e 11.1i}
\end{eqnarray}
 
\end{lemma}
\begin{proof}
For the sake of clarity, we will only prove our claims for $g$ whose increments can be written as
$ \delta g_{st} = \int_{s}^{t} \tilde{g}_{s} d B_{s} $, where $S_{2} (g, \tilde{g}, B)$ defines a rough path. 
We   also focus on  the inequality \eref{e11.1 ii}. For convenience we denote by $D$ the following increment defined on $\cs_{2} (\ll 0,T \rr)$: 
\begin{eqnarray*}
D_{st}&=&   \sum_{t_{k}=s}^{t-} f_{t_{k}} \int_{t_{k}}^{t_{k+1}} \delta g_{t_{k}u}     d(u-t_{k})^{2H}.
\end{eqnarray*}
 Since we have assumed that $\delta g_{st}=\int_{s}^{t} \tilde{g} dB$, we can write
\begin{eqnarray*}
D_{st}
&=&
     \sum_{t_{k}=s}^{t-} f_{t_{k}} \int_{t_{k}}^{t_{k+1}}
     \int_{t_{k}}^{u}   \tilde{g}_{v} d B_{v}
           d(u-t_{k})^{2H}.
\end{eqnarray*}
We now consider the following decomposition of $D_{st}$:
\begin{eqnarray}\label{e11.5iii}
D_{st} &=&  \sum_{t_{k}=s}^{t}(   D_{k}^{1} +D_{k}^{2}),
\end{eqnarray}
where $D^{1}_{k}$ and $D^{2}_{k}$ are given by:
     \begin{eqnarray*}
   D^{1}_{k}=
  f_{t_{k}}    \tilde{g}_{t_{k}} \int_{t_{k}}^{t_{k+1}}
     \int_{t_{k}}^{u}   d B_{v}
           d(u-t_{k})^{2H}   
,
\quad
 D^{2 }_{k}=
 f_{t_{k}}    \int_{t_{k}}^{t_{k+1}}
   \Big(  \int_{t_{k}}^{u} \int_{t_{k}}^{v} d \tilde{g}_{r} d B_{v} \Big)
           d(u-t_{k})^{2H}.
\end{eqnarray*}
Both $D^{1}_{k}$ and $D^{2}_{k}$ are easily bounded. Indeed, one can
note that $D^{2}_{k}$ is a Young integral,  and
since $(\tilde{g}, B)$ admits a lift as a $\ga$-rough path 
 it is easy to show that 
\begin{eqnarray}\label{e11.5ii}
\left|  D^{2}_{k}
\right| 
&\leq & G n^{-2H-2\ga}. 
\end{eqnarray}
Therefore, summing up both sides of \eref{e11.5ii} from $s$ to $t$ we obtain
\begin{eqnarray}\label{e11.6i}
\Big|\sum_{t_{k}=s}^{t-}  D^{2}_{k}\Big| &\leq & G n^{1-4\ga}(t-s). 
\end{eqnarray}
In order to bound $D^{1}_{k}$, we apply a change of variable formula for Young integrals, which yields
\begin{eqnarray*}
D^{1}_{k} = f_{t_{k}}\tilde{g}_{t_{k}}\delta B_{t_{k}t_{k+1}} (t_{k+1}-t_{k})^{2H}
- f_{t_{k}} \tilde{g}_{t_{k}} \int_{t_{k}}^{t_{k+1}} (u-t_{k})^{2H}dB_{u} \equiv D^{11}_{k}+D^{12}_{k}.
\end{eqnarray*}
Then $\sum_{t_{k}=s}^{t-} D^{11}_{k}$ is bounded by elementary considerations. The terms $D^{12}_{k}$ are handled by decomposing $  f_{t_{k}} \tilde{g}_{t_{k}} $ as $\delta (f\tilde{g})_{st_{k}}+  f_{s} \tilde{g}_{s} $ and by a direct application of   Lemma \ref{cor3.13}. It is thus readily seen that  
\begin{eqnarray}\label{e11.7}
\Big| \sum_{t_{k}=s}^{t-} D^{1}_{k} \Big| &\leq & G n^{1-4\ga+2\kappa} (t-s)^{1-\ga}. 
\end{eqnarray}
The estimate \eref{e11.1 ii} follows by applying \eref{e11.6i} and \eref{e11.7} to \eref{e11.5iii}.

Inequalities \eref{e11.1ii} and \eref{e 11.1i} can be shown in a similar way by invoking Lemma \ref{cor3.13} and Lemma \ref{lem11.2}, respectively. The proof is omitted.
\end{proof}
The following lemma considers almost sure bounds of a triple integral. It can be shown along the same lines as for Lemma \ref{lem6.11}. The proof, which hinges on Lemma \ref{lem11.2}, is omitted for sake of conciseness. 
\begin{lemma}
Let $f$, $\kappa$ be as in Lemma  \ref{lem6.11}.
Let $h=(h^{1},h^{2},h^{3})$ and $\tilde{h}=(\tilde{h}^{1},\tilde{h}^{2},\tilde{h}^{3})$ be continuous paths such that
$h^{e}$ takes values in $\RR$ and $\tilde{h}^{e}$ takes values in $\cl(\RR^{m},\RR)$. We also assume that
 $S_{2}(h, \tilde{h}, B)$  is a $\ga$-rough path for $H>\ga>\frac13$, and  that  $\delta h^{e}_{st}=\int_{s}^{t} \tilde{h}^{e} dB$ for $(s,t) \in \cs_{2}([0,T])$ and $e=1,2,3$. Then we have the following estimate for all $(s,t) \in \cs_{2} (\ll0,T\rr)$: 
\begin{eqnarray}
\Big| \sum_{t_{k}=s}^{t-} f_{t_{k}} \int_{t_{k}}^{t_{k+1}}    \int_{t_{k}}^{u}  \int_{t_{k}}^{v} h^{1}_{r} dB_{r} \otimes  h^{2}_{v} dB_{v} \otimes  h^{3}_{u} dB_{u} \Big|
&\leq& G n^{1-4\ga+2\kappa} (t-s)^{1-\ga}.
\label{e11.1 i}
\end{eqnarray}
\end{lemma}

The following results provide some upper-bound estimates for the $L_{p}$-norm of a ``discrete'' rough double integral. Recall that $0=t_{0}< \cdots<t_{n}=T$ and $0=u_{0}<\cdots<u_{\nu}=T$ are two uniform partitions on $[0,T]$. 

\begin{lemma}\label{lem11.4}
Let $F$ be defined in \eref{e4.1} and $H>\frac14$.   We set 
\begin{eqnarray*}
\hat{\zeta}^{n}_{r} : = \sum_{l=0}^{\frac{\nu r}{T}-1} \zeta_{l}^{n}
 \quad \text{ with }\quad
\zeta_{l}^{n} :=\zeta_{l}^{n,ijj'}= n^{2H-\frac12} \sum_{ t_{k} \in \tilde{D}_{l} }  \delta B^{i}_{t_{k_{l}} t_{k}}   \delta   F^{jj'}_{t_{k}t_{k+1}},
\end{eqnarray*}
 where     $r \in \{u_{1},\dots,u_{\nu}\}$.  
Then the following estimate holds true   for all $p\geq 1$:
\begin{eqnarray}\label{e9.32}
\| \delta \hat{\zeta}^{n}_{r,r'} \|_{p} \leq   K \nu^{-H}  n^{\frac12 -2H} (r'-r)^{\frac12} \,,
\quad\quad r,r' \in \{ u_{0},\dots, u_{\nu} \}.
\end{eqnarray}
\end{lemma}
\begin{proof} 
Let us treat the special case $i=j\neq j'$.    
Then we can expand the variance of $\zeta^{n}_{l}$ as:
\begin{eqnarray}\label{e9.18i}
\mE\Big(\Big| \sum_{l=0}^{  \frac{\nu r}{T} -1 }   \zeta^{n}_{l} \Big|^{2}\Big) 
&=&  
n^{4H-1}
\sum_{l,l' = 0}^{\frac{\nu r}{T}-1} \sum_{  t_{k} \in \tilde{D}_{l}  }    \sum_{ t_{k'} \in \tilde{D}_{l'}  } J(k,k',t_{k_{l}},t_{k_{l'}}) ,
\end{eqnarray}
where for all $k$, $k'$, $u$ and $v$ we have set:
\begin{eqnarray}\label{e9.11i}
J(k,k',u,v) &=& 
\mE\left(
\delta B^{i}_{u t_{k}}   \delta  F^{jj'}_{t_{k}t_{k+1}} 
\delta B^{i}_{v t_{k'}}   \delta  F^{jj'}_{t_{k'}t_{k'+1}} 
\right) .
\end{eqnarray}

In order to evaluate \eref{e9.11i}, we first condition the expected value on $B^{j'}$. We are thus left with the expected value of a product of four centered Gaussian random variables, for which we can use the Gaussian identity, and the isometry property stated in Definition \ref{def3.2}. 
We use the same isometry again in order to integrate with respect to $B^{j'}$, which yields:
\begin{eqnarray}\label{eq9.18}
J(k,k',u,v)  
&=& \sum_{e=1}^{3}  J_{e}(k,k',u,v) ,
 \end{eqnarray}
where  we have:
\begin{eqnarray}
J_{1}(k,k',u,v)  &=& \langle  \mathbf{1}_{[u, t_{k}] } , \mathbf{1}_{[v, t_{k'}] }    \rangle_{\mathcal{H}}  \langle  \beta_{k}  ,  \beta_{k'}     \rangle_{\mathcal{H}^{\otimes 2}} ,
\label{e9.11}
\\
J_{2}(k,k',u,v)  &=& \left\langle
\langle  \mathbf{1}_{[u, t_{k} ]} , \beta_{k}    \rangle_{\mathcal{H}} ,  \langle    \mathbf{1}_{[v, t_{k'} ]} ,  \beta_{k'}     \rangle_{\mathcal{H}} 
\right\rangle_{\mathcal{H}},
\label{e9.12}
\\
J_{3}(k,k',u,v)  &=& \left\langle
\langle  \mathbf{1}_{[v, t_{k'} ]} , \beta_{k}    \rangle_{\mathcal{H}} ,  \langle    \mathbf{1}_{[u, t_{k} ]} ,  \beta_{k'}     \rangle_{\mathcal{H}} 
\right\rangle_{\mathcal{H}},
\label{e9.13i}
\end{eqnarray}
and where the function $\beta$ is defined by 
   $ \beta_{k} (u,v) =   \mathbf{1}_{   t_{k}<u<v< t_{k+1} }$.

Next observe that we have $ \langle   \beta_{k}  ,  \beta_{k'}      \rangle_{\mathcal{H}^{\otimes 2}}\geq 0$ for all $k$ and $k'$.
Indeed, when $k=k'$, this stems from the fact that $\langle   \beta_{k}  ,  \beta_{k'}      \rangle_{\mathcal{H}^{\otimes 2}}$ can be identified with $\mE [ | \delta F^{jj'}_{t_{k}t_{k+1}}|^{2}]$, while for $k\neq k'$ the expression for $\langle   \beta_{k}  ,  \beta_{k'}      \rangle_{\mathcal{H}^{\otimes 2}}$ is given by \eref{e3.3}, and the product of the measures $\mu$ therein gives a positive contribution. 
 So the Cauchy-Schwarz inequality implies that
\begin{eqnarray}\label{e9.15}
|J_{1} (k,k',u,v)| &\leq& K |t_{k}-u|^{H} |t_{k'} - v|^{H}   \langle   \beta_{k}  ,  \beta_{k'}      \rangle_{\mathcal{H}^{\otimes 2}} ,
\end{eqnarray}
and we easily get the following bound: 
\begin{eqnarray*}
\Big|\sum_{l,l' = 0}^{\frac{\nu r}{T} -1 } \sum_{  t_{k} \in \tilde{D}_{l}  }    \sum_{ t_{k'} \in \tilde{D}_{l'}  }  J_{1} (k,k',t_{k_{l}},t_{k_{l'}})\Big| 
  \leq 
  K \nu^{-2H} \sum_{k,k'=0}^{\lfloor\frac{nr}{T}\rfloor-1}   \langle   \beta_{k}  ,  \beta_{k'}      \rangle_{\mathcal{H}^{\otimes 2}}
  \\
  = K \nu^{-2H}  \sum_{k,k'=0}^{\lfloor\frac{nr}{T}\rfloor-1}  
  \mE\left[ \delta F^{jj'}_{t_{k}t_{k+1}}  \delta F^{jj'}_{t_{k'}t_{k'+1}}  \right] 
  = K \nu^{-2H}
  \mE\Big[ \Big|F^{jj'}_{\eta(r)}\Big|^{2} \Big] 
   .
\end{eqnarray*}
It then follows from Lemma \ref{lem4.1} that
\begin{eqnarray}\label{eq9.24}
\Big|
\sum_{l,l' = 0}^{\frac{\nu r}{T} -1 } \sum_{  t_{k} \in \tilde{D}_{l}  }    \sum_{ t_{k'} \in \tilde{D}_{l'}  }  J_{1} (k,k',t_{k_{l}},t_{k_{l'}})
\Big|
 \leq  K \nu^{-2H}  n^{1 -4H} \eta(r) 
 \leq  K \nu^{-2H}  n^{1 -4H} r .
\end{eqnarray}

Let us now handle the terms $J_{2}$ and $J_{3}$ in \eref{eq9.18}. First,  by self-similarity of $B$  we obtain
\begin{eqnarray}\label{eq11.18}
J_{2}(k,k',u,v)  
&=& 
n^{-6H}T^{6H}
\big\langle
\langle  \mathbf{1}_{[\frac{nu}{T},  {k} ]} (a) , \phi_{k}  (b, a)  \rangle_{\mathcal{H}} ,  \langle    \mathbf{1}_{[\frac{nv}{T},  {k'} ]}(c) ,  \phi_{k'}    (b,c) \rangle_{\mathcal{H}} 
\big\rangle_{\mathcal{H}},
\end{eqnarray}
where we have denoted $\phi_{k}(u,v) = \mathbf{1}_{  k< u < v < k+1 }   $, and
   the letters  $a$, $b$, $c$  designate  the pairing for our inner product  in $\ch$. In order to estimate the quantity \eref{eq11.18}, we assume first that   $k,k'$ satisfies $|k-k'|>2$.  In this case, we can  approximate the   functions $\mathbf{1}_{k< u<v< k+1}$ in the definition of $\phi_{k}$ by sums of indicators of rectangles. Namely, for $k \leq \left\lfloor\frac{nt}{T} \right\rfloor$  we set
\begin{eqnarray}\label{eq11.19}
\phi^{\ell}_{k} (u,v)  &=& \sum_{i=0}^{\ell-1}  \mathbf{1}_{[k+\frac{i}{\ell}, k+\frac{i+1}{\ell}]} (u) \times   \mathbf{1}_{[   k+ \frac{i+1 }{\ell}, k+1]} (v) ,
\end{eqnarray}
 then the convergence $\lim_{\ell\to \infty} \| \phi_{k}-\phi^{\ell}_{k} \|_{\mathcal{H}^{\otimes 2}} = 0$ holds true whenever $H>\frac14$. Applying the convergence of $\phi_{k}^{\ell}$ to \eref{eq11.18} and taking into account   expression \eref{eq11.19}  we obtain
 \begin{eqnarray*}
J_{2}(k,k',u,v)  
&=& n^{-6H}T^{6H}
\lim_{\ell\to \infty}
\Big\langle
\langle  \mathbf{1}_{[\frac{nu}{T},  {k} ]} (a) , \phi^{\ell}_{k}  (b, a)  \rangle_{\mathcal{H}} ,  \langle    \mathbf{1}_{[\frac{nv}{T},  {k'} ]}(c) ,  \phi^{\ell}_{k'}    (b,c) \rangle_{\mathcal{H}} 
\Big\rangle_{\mathcal{H}}
\\
&\leq&n^{-6H}T^{6H} \lim_{\ell\to \infty} \sum_{i,j=1}^{\ell-1}   d_{ij}\tilde{d}_{ij}
,
\end{eqnarray*}
where we denote
\begin{eqnarray*}
d_{ij}= \langle 
 \mathbf{1}_{[k+\frac{i}{\ell}, k+\frac{i+1}{\ell}]}   ,
  \mathbf{1}_{[k'+\frac{j}{\ell}, k'+\frac{j+1}{\ell}]}  
 \rangle_{\mathcal{H}}\,,
\quad\quad
\tilde{d}_{ij} = \langle 
 \mathbf{1}_{[\frac{nu}{T},  {k} ]},  \mathbf{1}_{[   k+ \frac{i+1 }{\ell}, k+1]}
 \rangle_{\mathcal{H}} \langle 
 \mathbf{1}_{[\frac{nv}{T},  {k'} ]},  \mathbf{1}_{[   k'+ \frac{j+1 }{\ell}, k'+1]}
 \rangle_{\mathcal{H}}.
\end{eqnarray*}
 It is easy to see that, for all $i,j\leq \ell-1$, we have $|\tilde{d}_{ij}|\leq K$. Therefore, taking into account the fact that $d_{ij}<0$ for $k,k':|k-k'|>2$, we obtain
 \begin{eqnarray}
J_{2}(k,k',u,v) &\leq& \frac{K}{ n^{6H}}  \lim_{\ell\to \infty} \sum_{i,j=1}^{\ell-1}   |d_{ij}| = \frac{K}{ n^{6H}} \lim_{\ell\to \infty} \sum_{i,j=1}^{\ell-1}    d_{ij} 
\nonumber
\\
&=&
\frac{K}{ n^{6H}}  \langle \mathbf{1}_{[k,k+1]}, \mathbf{1}_{[k', k'+1]} \rangle_{\mathcal{H}} \leq 
K n^{-6H} |k-k'|^{2H-2}.
\label{e9.6}
\end{eqnarray}
 The estimate \eref{e9.6} also holds true for $J_{3}$, and the proof is similar. 
In addition, for $|k - k'|\leq 2$ the relation \eref{eq11.18}  shows that 
\begin{eqnarray}\label{e9.7}
|J_{e} (k,k',u,v)| &\leq& K n^{-6H} .
\end{eqnarray}
We now invoke   \eref{e9.6} and \eref{e9.7}, together with the fact  that $\sum_{k>1} |k|^{2H-2}<\infty$ and 
\begin{eqnarray*}
\#\Big( \bigcup_{l=0}^{\frac{\nu r}{T} -1} \tilde{D}_{l}\Big)  = \# \{ k: t_{k}\leq     t \text{ \ and \ } t_{k}<r\}\leq \frac{nr}{T},
\end{eqnarray*}
 which yields for $e=2,3$: 
\begin{eqnarray}\label{e9.27}
\Big|\sum_{l,l' = 0}^{\frac{\nu r}{T} -1 } \sum_{  t_{k} \in \tilde{D}_{l}  }    \sum_{ t_{k'} \in \tilde{D}_{l'}  }  J_{e} (k,k',t_{k_{l}},t_{k_{l'}})\Big|
 &\leq& K    n^{1 -6H} r. 
\end{eqnarray}
 Gathering our bounds \eref{eq9.24} and \eref{e9.27} for $J_{1}$, $J_{2}$ and $J_{3}$, it is readily checked from our decompositions 
 \eref{e9.18i} and \eref{eq9.18} that for  $i=j\neq j'$ we have:
\begin{eqnarray}\label{e9.13}
\mE\Big(\Big| \sum_{l=0}^{  \frac{\nu r}{T} -1 }   \zeta^{n}_{l} \Big|^{2}\Big) 
& \leq &   K \nu^{-2H}  n^{1 -4H} r
 .
\end{eqnarray}
Furthermore, using the stationarity of the increments of $F$ and $B$, plus the equivalence of $L_{p}$-norms in finite chaos,  we obtain from \eref{e9.13} the desired estimate \eref{e9.32}.   
 Moreover, the estimate \eref{e9.13} holds true for other $i,j,j'$. The proof is similar and is omitted. 
\end{proof}
\begin{lemma}\label{lem:sum-delta-B-F}
Let $F$ be defined in \eref{e4.1}. 
Then the following estimate holds true for $H>\frac14$:
\begin{eqnarray*}
\mE\Big( \Big|
\sum_{t_{{k }}=s}^{t-}   \delta B_{s t_{k }} \otimes F_{t_{{k }}t_{{k }+1}}
\Big|^{2}\Big) &\leq& K n^{1-4H}(t-s)^{2H+1} \,, \quad\quad (s,t) \in \cs_{2}(\ll0,T\rr).
\end{eqnarray*}
\end{lemma}
\begin{proof} Since $B$ has stationary increment, it suffices to prove the lemma for $s=0$. As in the proof of Proposition \ref{prop9.2}, 
we consider the sum
\begin{eqnarray*}
M^{ijj'}_{t} = \sum_{t_{k}=0}^{t-} B^{i}_{t_{k}} F^{jj'}_{t_{k}t_{k+1}},
\end{eqnarray*}
for   $i=j\neq j'$. The other cases can be considered similarly. Let $J_{e}$, $e=1,2,3$ be the quantities defined in \eref{e9.11}, \eref{e9.12}, \eref{e9.13i}. 
By \eref{e9.11i} we have
\begin{eqnarray}\label{e11.5i}
\mE \Big( \Big|
\sum_{t_{k}=0}^{t-}    B^{i}_{t_{k } } F^{jj'}_{t_{k}t_{k+1}} 
\Big|^{2}
\Big)
&=& 
\sum_{t_{k_{1}}, t_{k_{2}}=0}^{t-}  \sum_{e=1}^{3} J_{e}(k_{1},k_{2}, 0,0)
 \,.
\end{eqnarray} 
Applying \eref{e9.15} and taking into account Lemma \ref{lem4.1} yields
\begin{eqnarray}\label{e11.5}
\sum_{t_{k_{1}}, t_{k_{2}}=0}^{t-}  J_{1}(k_{1},k_{2}, 0,0) &\leq& n^{1-4H} t^{1+2H}.
\end{eqnarray}
On the other hand, it follows from \eref{e9.6} and \eref{e9.7} that  for $e=2,3$ we have
\begin{eqnarray}\label{e11.6}
\sum_{t_{k_{1}}, t_{k_{2}}=0}^{t-}  J_{e}(k_{1},k_{2}, 0,0)  \leq  n^{1-6H} t 
 \leq  n^{1-4H}t^{1+2H}.
\end{eqnarray}
The lemma then follows by applying \eref{e11.5} and \eref{e11.6} to \eref{e11.5i}.
\end{proof}

%\hfill { \footnotesize Date: \today}

\end{document}